\documentclass[reqno]{amsart}
\usepackage{tikz,xcolor,listings,ytableau,bbm,mathdots,mathtools,arydshln,amssymb}
\usepackage[all,cmtip]{xy}
\usetikzlibrary{arrows,matrix,patterns}

\usepackage[english]{babel}
\usepackage[margin=1.0in]{geometry}

\usepackage[colorlinks=true, allcolors=blue]{hyperref}

\definecolor{darkred}{rgb}{0.7,0,0} 
\newcommand{\defn}[1]{{\color{darkred}\emph{#1}}} 

\definecolor{UQgold}{RGB}{196, 158, 54} 
\definecolor{UQpurple}{RGB}{73, 7, 94} 

\newcommand{\ds}{\displaystyle}

\newcommand{\xx}{\mathbf{x}}

\newcommand{\zz}{\mathbf{z}}
\newcommand{\tw}[1]{\widetilde{#1}}
\newcommand{\ov}[1]{\overline{#1}}
\newcommand{\wh}[1]{\widehat{#1}}

\newcommand{\mb}[1]{\mathbf{#1}}
\newcommand{\mc}[1]{\mathcal{#1}}
\newcommand{\vp}{\varphi}
\newcommand{\ve}{\varepsilon}
\newcommand{\la}{\lambda}
\newcommand{\Net}{\Lambda} 
\newcommand{\lx}[2]{x_{#1}^{(#2)}}
\newcommand{\lE}[2]{E_{#1}^{(#2)}}
\newcommand{\lQ}[2]{Q_{#1}^{(#2)}}
\newcommand{\lrQ}[2]{\tw{Q}_{#1}^{(#2)}}
\newcommand{\brP}[2]{\tw{\ov{P}}_{#1}^{(#2)}}
\newcommand{\ls}[2]{s_{#1}^{(#2)}}
\newcommand{\lH}[2]{H_{#1}^{(#2)}}

\newcommand{\bx}[2]{x_{(#1)}^{#2}}
\newcommand{\bE}[2]{\ov{E}_{#1}^{(#2)}}
\newcommand{\bP}[2]{\ov{P}_{#1}^{(#2)}}
\newcommand{\bs}[2]{\ov{s}_{#1}^{(#2)}}
\newcommand{\bH}[2]{\ov{H}_{#1}^{(#2)}}

\newcommand{\braceabove}[2]{\overbrace{\hspace{#1 cm}}^{#2}}
\newcommand{\bracebelow}[2]{\underbrace{\hspace{#1 cm}}_{#2}}

\DeclareMathOperator{\wind}{wind}
\DeclareMathOperator{\Mat}{Mat}
\DeclareMathOperator{\sh}{sh}

\DeclareMathOperator{\RSK}{RSK}
\DeclareMathOperator{\gRSK}{gRSK}
\DeclareMathOperator{\SSYT}{SSYT}
\DeclareMathOperator{\Cyl}{CylTab}

\DeclareMathOperator{\SL}{SL}
\DeclareMathOperator{\GL}{GL}

\DeclareMathOperator{\GT}{GT}
\DeclareMathOperator{\Trop}{Trop}
\DeclareMathOperator{\Frac}{Frac} 
\DeclareMathOperator{\Sym}{Sym} 
\DeclareMathOperator{\LSym}{LSym} 

\def\Inv{{\rm Inv}}

\DeclareMathOperator{\wt}{wt} 

\newcommand{\ZZ}{\mathbb{Z}}
\newcommand{\Zmod}[1]{\ZZ/{#1}\ZZ}

\newcommand{\RR}{\mathbb{R}}
\newcommand{\CC}{\mathbb{C}}

\newcommand{\abs}[1]{\lvert #1 \rvert}

\theoremstyle{plain}
\newtheorem{thm}{Theorem}[section]

\newtheorem{question}[thm]{Question}
\newtheorem{lemma}[thm]{Lemma}
\newtheorem{conj}[thm]{Conjecture}
\newtheorem{prop}[thm]{Proposition}
\newtheorem{cor}[thm]{Corollary}
\theoremstyle{definition}
\newtheorem{dfn}[thm]{Definition}
\newtheorem{ex}[thm]{Example}
\newtheorem{remark}[thm]{Remark}
\numberwithin{equation}{section}
\newcommand{\ThmInvEX}{Thm.~1.1}
\newcommand{\LemExplicitBasic}{Lem.~2.3}
\newcommand{\LemPhiPsi}{Lem.~3.3}
\newcommand{\LemShapePreserved}{Lem.~3.6}
\newcommand{\LemGTDec}{Lem.~3.9}
\newcommand{\LemExplicitGT}{Lem.~3.10}

\newcommand{\ThmGRSKIsom}{Thm.~4.10}
\newcommand{\CorBicrystal}{Cor.~4.11}
\newcommand{\GRSKAndDecorations}{Thm.~4.13}

\newcommand{\CorInvEEX}{Cor.~5.4(3)}
\newcommand{\PropWeylR}{Prop.~6.1}
\newcommand{\ThmFTLSF}{Thm.~6.6}

\usepackage[colorinlistoftodos]{todonotes}

\makeatletter
\newcommand{\vast}{\bBigg@{3}}
\newcommand{\Vast}{\bBigg@{4.5}}
\makeatother

\title{Crystal invariant theory II: pseudo-energies}

\author[B.~Brubaker]{Benjamin Brubaker}
\address[B. Brubaker]{School of Mathematics, University of Minnesota, 206 Church St. SE, Minneapolis, MN 55455}
\email{brubaker@math.umn.edu}
\urladdr{http://www-users.math.umn.edu/~brubaker/}

\author[G.~Frieden]{Gabriel Frieden}
\address[G. Frieden]{LaCIM, Universit\'e du Qu\'ebec \`a Montr\'eal, Montr\'eal, QC, Canada}
\email{gabriel.frieden@lacim.ca}
\urladdr{https://sites.google.com/a/umich.edu/gfrieden/}

\author[P.~Pylyavskyy]{Pavlo Pylyavskyy}
\address[P. Pylyavksyy]{School of Mathematics, University of Minnesota, 206 Church St. SE, Minneapolis, MN 55455}
\email{ppylyavs@math.umn.edu}
\urladdr{https://sites.google.com/site/pylyavskyy/}

\author[T.~Scrimshaw]{Travis Scrimshaw}
\address[T.~Scrimshaw]{OCAMI, Osaka Metropolitan University, 3--3--138 Sugimoto, Sumiyoshi-ku, Osaka 558-8585, Japan}
\curraddr{5 Ch\=ome Kita 8 J\=onishi, Kita Ward, Sapporo, Hokkaid\=o 060--0808, Japan}
\email{tcscrims@gmail.com}
\urladdr{https://tscrim.github.io/}

\begin{document}

\begin{abstract}
The geometric crystal operators and geometric $R$-matrices (or geometric Weyl group actions) give commuting actions on the field of rational functions in $mn$ variables. We study the invariants of various combinations of these actions, which we view as ``crystal analogues'' of the invariants of $S_m$, $\SL_m$, $S_n \times S_m$, $\SL_n \times \, S_m$, and $\SL_n \times \SL_m$ acting on the polynomial ring in an $m \times n$ matrix of variables. The polynomial invariants of the $S_m$-action generated by the $\GL_m$-geometric $R$-matrices were described by Lam and the third-named author as the ring of loop symmetric functions. In a previous paper of the authors, the polynomial invariants of the $\GL_m$-geometric crystal operators were described as a subring of the ring of loop symmetric functions.

In this paper, we give conjectural generating sets for the fields of rational invariants in the remaining cases, and we give formulas expressing a large class of loop symmetric functions in terms of these conjectural generators. Our results include new positive formulas for the central charge and energy function of a product of single-row geometric crystals, and a new derivation of Kirillov and Berenstein's piecewise-linear formula for cocharge. The formulas manifest the symmetries possessed by these functions.
\end{abstract}

\maketitle

\tableofcontents

\section{Introduction}

\subsection{RSK as a crystal isomorphism}
\label{sec:intro RSK}

The irreducible polynomial representations of the general linear group $\GL_m$ (over $\CC$) are labeled by integer partitions with at most $m$ parts.
For such a representation $V(\lambda)$ indexed by the partition $\lambda$, its bases can be labeled by the set $\SSYT_{\leq m}(\lambda)$ of semistandard Young tableaux of shape $\lambda$, with entries in $\{1, \ldots, m\}$.
In particular, the single-row partition $(L)$ corresponds to $\Sym^L \CC^m$, the $L$-fold symmetric power of the defining representation of $\GL_m$. This representation has a basis consisting of the degree $L$ monomials $x_1^{a_1} \cdots x_m^{a_m}$ in $m$ variables; thus, the set of all exponent vectors $\mb{a} = (a_1, \ldots, a_m) \in (\ZZ_{\geq 0})^m$ labels a basis of the representation $\Sym \CC^m = \bigoplus_{L \geq 0} \Sym^L \CC^m$. We represent the monomial $x_1^{a_1} \cdots x_m^{a_m}$ as the unique semistandard Young tableau of a single row in which the number $i$ appears $a_i$ times.

Let $A = (a_i^j)_{i \in \{1, \ldots, m\}, j \in \{1, \ldots, n\}}$ be an $m \times n$ matrix of nonnegative integers. The matrix $A$ corresponds to a basis vector of the $n$-fold tensor product $(\Sym \CC^m)^{\otimes n}$, with the $j$th column $\mb{a}^j = (a_1^j, \ldots, a_m^j)$ representing\footnote{As in our previous paper, our row/column convention is the reverse of the usual convention in tensor calculus for covariant/contravariant components.} a basis vector of the $j$th factor. Alternatively, $A$ can be viewed as a basis vector of the $m$-fold tensor product $(\Sym \CC^n)^{\otimes m}$, with the $i$th row $\mb{a}_i = (a_i^1, \ldots, a_i^n)$ labeling a basis vector of the $i$th factor. Thus, Kashiwara's theory of crystal bases~\cite{K90,K91} endows the set of such matrices with two sets of crystal operators: $\GL_m$-crystal operators $e_1, \ldots, e_{m-1}$, with $e_i$ modifying the entries in rows $i$ and $i+1$, and $\GL_n$-crystal operators $\ov{e}_1, \ldots, \ov{e}_{n-1}$, with $\ov{e}_j$ modifying the entries in columns $j$ and $j+1$. The $\GL_m$-crystal operators commute with the $\GL_n$-crystal operators, so one can ask how the set of $m \times n$ non-negative integer matrices decomposes as a $\GL_n \times \GL_m$-crystal.

This is given combinatorially by the Robinson--Schensted--Knuth (RSK) correspondence~\cite{Schensted,Knuth}, a celebrated bijection between $m \times n$ matrices $A$ of nonnegative integers and pairs $(P,Q)$ of semistandard Young tableaux of the same shape, where the entries of $P$ lie in $\{1, \ldots, n\}$, and the entries of $Q$ lie in $\{1, \ldots, m\}$. Endow the set $\bigsqcup_{\lambda} \SSYT(\lambda)_{\leq n} \times \SSYT(\lambda)_{\leq m}$ with a $\GL_n \times \GL_m$-crystal structure in which the $\GL_n$-crystal operators $\ov{e}_j$ act on the first factor, and the $\GL_m$-crystal operators $e_i$ act on the second factor. This is the crystal version of $(\GL_n, \GL_m)$ Howe duality~\cite{Howe89,Howe95}.

\begin{thm}[{\cite{Lascoux03,DanilovKoshevoy05,vanLeeuwen06}}]
\label{thm_RSK_isom}
With respect to the $\GL_n \times \GL_m$-crystal structures defined above, the RSK correspondence $A \mapsto (P,Q)$ is an isomorphism of $\GL_n \times \GL_m$-crystals.
\end{thm}

There is one further wrinkle that we will need. The crystal operators $\ov{e}_j, e_i$ can be used to define crystal reflection operators $\ov{s}_j, s_i$ which generate an action of the Weyl group $S_n \times S_m$ on either side of the RSK correspondence. On the $m \times n$ matrix side, these reflection operators are precisely the combinatorial $R$-matrices acting on adjacent rows or columns. That is, the $\GL_m$-reflection operator $s_i$, which acts on rows $i$ and $i+1$, agrees with the combinatorial $R$-matrix $R_i$ acting on rows $i$ and $i+1$, and $\ov{s}_j$ agrees with the combinatorial $R$-matrix $\ov{R}_j$ acting on columns $j$ and $j+1$.\footnote{Since $s_i$ acts only on the $Q$-tableau, one sees immediately from Theorem~\ref{thm_RSK_isom} that the map $R_i$ does not affect the $P$-tableau, which is perhaps a more well-known property of the combinatorial $R$-matrix.} Combinatorial $R$-matrices arise from the theory of Kirillov--Reshetikhin crystals as a $q \to 0$ limit of $R$-matrices for finite-dimensional representations of affine Lie algebras.
They have many applications in combinatorial representation theory and mathematical physics; see, e.g.,~\cite{KKMMNN91,HHIKTT01,KMO15,LS19,Shimozono02} and references therein.

\subsection{Geometric lifting and fields of invariants}
\label{sec:intro geometric lifting}

The setting for this paper is a ``birational'' or ``geometric'' version of the combinatorial story described above. All the maps from the previous section have been described as piecewise-linear maps~\cite{KB95,Kir01,HHIKTT01}. These piecewise-linear maps have subsequently been ``lifted,'' or ``de-tropicalized,'' to subtraction-free rational maps on algebraic varieties, essentially by replacing the operations $(\min, + , -)$ with $(+, \times, \div)$.

The geometric version of RSK was introduced by Kirillov~\cite{Kir01}, and extensively studied by Noumi and Yamada~\cite{NoumiYamada}. This map was originally called the tropical RSK correspondence, but following the convention of more recent work on this map, we will call it the \defn{geometric RSK correspondence} (gRSK). Geometric RSK is a birational automorphism of the variety of $m \times n$ matrices over $\CC^*$. For example, when $m=2$ and $n=3$, gRSK has the form
\[
\gRSK \colon \begin{pmatrix}
x_1^1 & x_1^2 & x_1^3 \medskip \\
x_2^1 & x_2^2 & x_2^3
\end{pmatrix}
\mapsto
\begin{pmatrix}
\textcolor{blue}{z_{2,2}} & \textcolor{blue}{z_{2,3}} = \textcolor{red}{z'_{2,2}} & \textcolor{red}{z'_{1,1}} \medskip \\
\textcolor{blue}{z_{1,1}} & \textcolor{blue}{z_{1,2}} & \textcolor{blue}{z_{1,3}} = \textcolor{red}{z'_{1,2}}
\end{pmatrix} \; ,
\]
where $z_{a,b}$ and $z'_{a,b}$ are rational functions in the $x_i^j$ with positive integer coefficients. The rational functions $z_{a,b}$ and $z'_{a,b}$ tropicalize to piecewise-linear formulas for the semistandard tableaux $P$ and $Q$ (more precisely, for the entries of the corresponding Gelfand--Tsetlin patterns) associated to the matrix $\mb{a}$ by RSK. In particular, the rational functions $z_{k,n} = z'_{k,m}$ tropicalize to formulas for the entries of the common shape of $P$ and $Q$. We refer to the arrays consisting of the $z_{a,b}$ and $z'_{a,b}$ as the $P$-pattern and $Q$-pattern, respectively.

Berenstein and Kazhdan developed the theory of \defn{geometric crystals}~\cite{BK00, BK07}, which defines geometric liftings of crystal operators. The variety of $m \times n$ matrices over $\CC^*$ has $\GL_m$- and $\GL_n$-geometric crystal structures, which were shown to commute in~\cite{LPaff}; we denote the resulting $\GL_n \times \GL_m$-geometric crystal by $X^{\Mat}$ (see~\ref{sec:basic}). The set of $P$-patterns (resp., $Q$-patterns) is endowed with a $\GL_n$- (resp., $\GL_m$-)geometric crystal structure (see~\ref{sec:GT}). These crystal structures are shape-preserving, so by viewing an $m \times n$ matrix as a $P$-pattern and $Q$-pattern glued along their common shape, one obtains another $\GL_n \times \GL_m$-geometric crystal structure on $\Mat_{m \times n}(\CC^*)$, which we denote by $X^{\GT}$. We proved in~\cite{BFPS} that geometric RSK is an isomorphism of $\GL_n \times \GL_m$-geometric crystals $X^{\Mat} \rightarrow X^{\GT}$.
 
Yamada introduced a geometric lifting of the combinatorial $R$-matrix, which we call the \defn{geometric $R$-matrix}~\cite{Yamada01}. The action of the geometric $R$-matrix on rows $i$ and $i+1$ (resp., columns $j$ and $j+1$) of an $m \times n$ matrix agrees with the geometric analogue of the $\GL_m$- (resp., $\GL_n$-) crystal reflection operator $s_i$ (resp., $\ov{s}_j$) on $X^{\Mat}$. Thus, we have a geometric lifting of the full combinatorial picture described in~\S\ref{sec:intro RSK}.

The geometric crystal operators and geometric $R$-matrices can be viewed as birational maps acting on the field of rational functions in $mn$ variables $x_i^j$. In this paper, we continue our study of the invariants of various combinations of these birational maps. Following the convention of the previous section, we denote the maps acting on rows by $e_i$ and $R_i$ ($i \in \{1, \ldots, m-1\}$), and the maps acting on columns by $\ov{e}_j$ and $\ov{R}_j$ ($j \in \{1, \ldots, n-1\}$). We write $\Inv_R$ (resp., $\Inv_{\ov{R}}, \Inv_e, \Inv_{\ov{e}}$) for the subfields of $\CC(x_i^j)$ consisting of fixed points of the $R_i$ (resp., $\ov{R}_j, e_i, \ov{e}_j$), and we refer to these fixed points as \defn{$R$-invariants} (resp., \defn{$\ov{R}$-invariants}, \defn{$e$-invariants}, \defn{$\ov{e}$-invariants}). We also consider intersections of these fields, which we denote by $\Inv_{ab} = \Inv_a \cap \Inv_b$ for $a,b \in \{R,\ov{R},e,\ov{e}\}$, and call \defn{$ab$-invariants}. Since $R_i$ (resp., $\ov{R}_j$) is a special case of $e_i$ (resp., $\ov{e}_j$), the non-trivial intersections are $\Inv_{R\ov{R}}, \Inv_{R \ov{e}}, \Inv_{\ov{R} e},$ and $\Inv_{e \ov{e}}$. Altogether, there are eight subfields of invariants, which are depicted in Figure~\ref{fig:nine}.

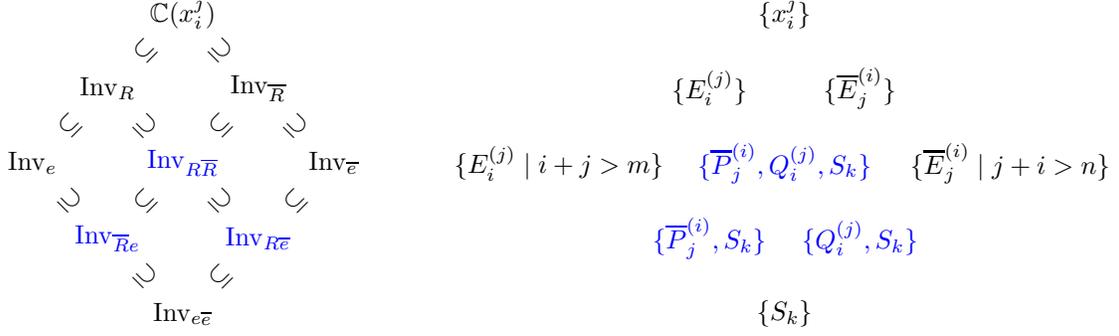
\begin{figure}
\begin{tikzpicture}

\draw (3,3) node{$\CC(x_i^j)$};
\draw (2,2) node{$\Inv_R$};
\draw (4,2) node{$\Inv_{\ov{R}}$};
\draw (1,1) node{$\Inv_e$};
\draw[blue] (3,1) node{$\Inv_{R\ov{R}}$};
\draw (5,1) node{$\Inv_{\ov{e}}$};
\draw[blue] (2,0) node{$\Inv_{\ov{R}e}$};
\draw[blue] (4,0) node{$\Inv_{R\ov{e}}$};
\draw (3,-1) node{$\Inv_{e\ov{e}}$};

\draw (2.5,2.5) node{\rotatebox{45}{$\subseteq$}};
\draw (1.5,1.5) node{\rotatebox{45}{$\subseteq$}};
\draw (3.5,2.5) node{\rotatebox{-45}{$\supseteq$}};
\draw (4.5,1.5) node{\rotatebox{-45}{$\supseteq$}};
\draw (3.5,1.5) node{\rotatebox{45}{$\subseteq$}};
\draw (2.5,1.5) node{\rotatebox{-45}{$\supseteq$}};
\draw (2.5,-0.5) node{\rotatebox{-45}{$\supseteq$}};
\draw (1.5,0.5) node{\rotatebox{-45}{$\supseteq$}};
\draw (3.5,-0.5) node{\rotatebox{45}{$\subseteq$}};
\draw (4.5,0.5) node{\rotatebox{45}{$\subseteq$}};
\draw (3.5,0.5) node{\rotatebox{-45}{$\supseteq$}};
\draw (2.5,0.5) node{\rotatebox{45}{$\subseteq$}};

\begin{scope}[xshift=8cm,scale=1]
\draw (3,3) node{$\{ x_i^j \}$};
\draw (2,2) node{$\{ \lE{i}{j} \}$};
\draw (4,2) node{$\{ \bE{j}{i} \}$};
\draw (0,1) node{$\{ \lE{i}{j} \mid i+j > m \}$};
\draw[blue] (3,1) node{$\{ \bP{j}{i}, \lQ{i}{j}, S_k \}$};
\draw (6,1) node{$\{ \bE{j}{i} \mid j+i > n \}$};
\draw[blue] (2,0) node{$\{ \bP{j}{i}, S_k \}$};
\draw[blue] (4,0) node{$\{ \lQ{i}{j}, S_k \}$};
\draw (3,-1) node{$\{ S_k \}$};
\end{scope}

\end{tikzpicture}
\caption{Various combinations of the geometric crystal operators and geometric $R$-matrices acting on $\CC(x_i^j)$ yield eight invariant subfields, shown in the diagram on the left. The diagram on the right shows an algebraically independent subset of polynomials in each subfield. These subsets are known to generate the subfield in the case of $R$-, $\ov{R}$-, $e$-, $\ov{e}$-, and $e\ov{e}$-invariants, and are conjectured to generate the subfield in the other three cases.}
\label{fig:nine}
\end{figure}

\subsection{Previous results}
\label{sec:intro prev results}

The invariants of the birational $S_m$-action generated by the geometric $R$-matrices $R_1, \ldots, R_{m-1}$ have been previously studied~\cite{LP12,LP13,LPS15, LPaff, Lamloop}, and it is known that the ring of polynomial invariants is generated by a family of polynomials
\[
\{\lE{i}{j}(\mb{x}_1, \ldots, \mb{x}_m) \mid i \in [m], j \in [n]\}
\]
($[k]$ denotes the set $\{1, \ldots, k\}$). Here $\mb{x}_i = (\lx{i}{1}, \ldots, \lx{i}{n})$ is a certain reordering of the variables in the $i$th row of the matrix $\mb{x} = (x_i^j)$, and we view $\lx{i}{j}$ as the $j$th ``color'' of the ``variable'' $\mb{x}_i$. For each $j \in [n]$, the polynomial $\lE{i}{j}$ is an analogue of the elementary symmetric polynomial of degree $i$ in $m$ variables. The $\lE{i}{j}$ are called \defn{loop elementary symmetric functions}. The ring generated by the $\lE{i}{j}$ is called the ring of \defn{loop symmetric functions} in $m$ variables and $n$ colors, and is denoted $\LSym_m(n)$. When $n=1$, the geometric $R$-matrices are simply adjacent transpositions of the variables $\lx{1}{1}, \ldots, \lx{m}{1}$, and $\LSym_m(1)$ is the ring of symmetric polynomials in these variables.

We proved in~\cite{BFPS} that $\Inv_e$ (and its polynomial subring) is generated by the subset
\[
\{\lE{i}{j}(\mb{x}_1, \ldots, \mb{x}_m) \mid i \in [m], j \in [n], i+j > m\}
\]
of \defn{$P$-type loop elementary symmetric functions}. We give a brief review of the proof of this result. The loop elementary symmetric functions are naturally organized into an infinite matrix $\widetilde{M} = \widetilde{M}(\mb{x}_1, \ldots, \mb{x}_m)$, whose entries are given by
\[
\widetilde{M}_{i,j} = \lE{m+j-i}{i},
\]
with the color in the superscript interpreted modulo $n$. For example, when $n = 3$ and $m=4$, we have
\[
\widetilde{M} =
\left(\begin{array}{ccc | ccc | ccc}
\lE{4}{1} & 0 & 0 & 0 & 0 & 0 & \hdots \\
\lE{3}{2} & \lE{4}{2} & 0 & 0 & 0 & 0 \\
\lE{2}{3} & \lE{3}{3} & \lE{4}{3} & 0 & 0 & 0 \\ \hline
\lE{1}{1} & \lE{2}{1} & \lE{3}{1} & \lE{4}{1} & 0 & 0 \\
1 & \lE{1}{2} & \lE{2}{2} & \lE{3}{2} & \lE{4}{2} & 0 \\
0 & 1 & \lE{1}{3} & \lE{2}{3} & \lE{3}{3} & \lE{4}{3} \\ \hline
0 & 0 & 1 & \lE{1}{1} & \lE{2}{1} & \lE{3}{1} \\
0 & 0 & 0 & 1 & \lE{1}{2} & \lE{2}{2} \\
0 & 0 & 0 & 0 & 1 & \lE{1}{3} \\ \hline
\vdots &  &  & & & & \ddots
\end{array}\right).
\]
As this example illustrates, $\widetilde{M}$ consists of $n \times n$ blocks which repeat along (block) diagonals. The $P$-type loop elementary symmetric functions are the $\lE{i}{j}$ appearing in the top right nonzero block $M$.

By definition, the entries of the gRSK $P$-pattern associated to the matrix $\xx$ are ratios of minors of the $n \times n$ matrix $M$. We showed that these ratios of minors generate the same field as the entries of $M$. Since the operators $e_i$ act only on the $Q$-pattern, it follows that the $P$-type $\lE{i}{j}$-s are $e$-invariant. We then showed that the action of the $e_i$ on the set of $Q$-patterns with fixed shape has a Zariski dense orbit. Finally, we used a result from algebraic geometry about dimensions of fibers to complete the proof.

The geometric RSK correspondence, like RSK, has the symmetry
\[
\gRSK(\mb{x}) = (P,Q) \iff \gRSK(\mb{x}^t) = (Q,P),
\]
where $\mb{x}^t$ is the transpose of the matrix $\mb{x}$. This allows us to transport properties of $\Inv_R$ and $\Inv_e$ to $\Inv_{\ov{R}}$ and $\Inv_{\ov{e}}$. In particular, $\Inv_{\ov{R}}$ is generated by the \defn{barred loop elementary symmetric functions} $\bE{j}{i}(\mb{x}^1, \ldots, \mb{x}^n)$ (for $j \in [n], i \in [m]$), where $\mb{x}^j$ is the $j$-th column of $\mb{x}$; this invariant field is isomorphic to $\Frac(\LSym_n(m))$. In keeping with our conventions for maps, we write $\LSym$ for the ring of $R$-invariants and $\ov{\LSym}$ for the ring of $\ov{R}$-invariants.

The ring of loop symmetric functions possesses distinguished elements called \defn{loop (skew) Schur functions}, which are determined by a skew shape $\la/\mu$ and a ``color'' $r \in [n]$. We call the pair $(\la/\mu, r)$ a \defn{colored skew shape}, and we depict this pair as the Young diagram of $\la/\mu$ with the cell in row $a$ and column $b$ assigned the color $(r+a-b) \mod n$. In other words, the cells on the main diagonal (in English notation) have color $r$, the cells on the diagonal to the right of the main diagonal have color $r-1$, etc. The loop skew Schur function $s_{\la/\mu}^{(r)}(\mb{x}_1, \ldots, \mb{x}_m)$ is the generating function for semistandard Young tableaux of shape $\la/\mu$ with respect to a weight that takes into account the color of the cells.

We defined the \defn{shape invariants} $S_1, \ldots, S_{\min(m,n)}$ to be the loop Schur functions
\[
S_k = s_{(n-k+1)^{m-k+1}}^{(n)}(\mb{x}_1, \ldots, \mb{x}_m).
\]
The colored shapes associated to $S_1, S_2, S_3$ in the case $m=4,n=3$ are shown below; red cells have color 1, green cells have color 2, and blue cells have color 3.

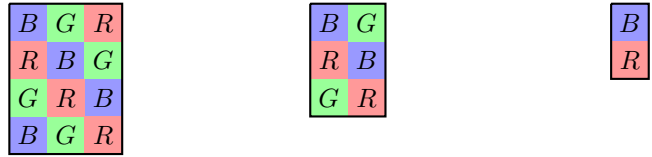
\begin{center}
\begin{tikzpicture}[baseline=-40,scale=0.5]

\begin{scope}[xshift=16cm]
\foreach \x/\y in {2/1} {\fill[blue!40] (\x,-\y) rectangle (\x+1,-\y-1); \draw (\x+0.5,-\y-0.5) node {$B$};}
\foreach \x/\y in {2/2} {\fill[red!40] (\x,-\y) rectangle (\x+1,-\y-1); \draw (\x+0.5,-\y-0.5) node {$R$};}
\draw[thick] (2,-1) -- (3,-1) -- (3,-3) -- (2,-3) -- (2,-1);
\end{scope}

\begin{scope}[xshift=8cm]
\foreach \x/\y in {2/1,3/2} {\fill[blue!40] (\x,-\y) rectangle (\x+1,-\y-1); \draw (\x+0.5,-\y-0.5) node {$B$};}
\foreach \x/\y in {3/1,2/3} {\fill[green!40] (\x,-\y) rectangle (\x+1,-\y-1); \draw (\x+0.5,-\y-0.5) node {$G$};}
\foreach \x/\y in {2/2,3/3} {\fill[red!40] (\x,-\y) rectangle (\x+1,-\y-1); \draw (\x+0.5,-\y-0.5) node {$R$};}
\draw[thick] (2,-1) -- (4,-1) -- (4,-4) -- (2,-4) -- (2,-1);
\end{scope}

\begin{scope}[xshift=0cm]
\foreach \x/\y in {2/1,3/2,4/3,2/4} {\fill[blue!40] (\x,-\y) rectangle (\x+1,-\y-1); \draw (\x+0.5,-\y-0.5) node {$B$};}
\foreach \x/\y in {3/1,2/3,4/2,3/4} {\fill[green!40] (\x,-\y) rectangle (\x+1,-\y-1); \draw (\x+0.5,-\y-0.5) node {$G$};}
\foreach \x/\y in {4/1,2/2,3/3,4/4} {\fill[red!40] (\x,-\y) rectangle (\x+1,-\y-1); \draw (\x+0.5,-\y-0.5) node {$R$};}
\draw[thick] (2,-1) -- (5,-1) -- (5,-5) -- (2,-5) -- (2,-1);
\end{scope}

\end{tikzpicture}\ 
\end{center}

\noindent The shape invariants are related to the common shape of the $P$- and $Q$-patterns by
\[
(z_{1,n}, \ldots, z_{\min(m,n),n}) = (z'_{1,m}, \ldots, z'_{\min(m,n),m}) = \left(\frac{S_1}{S_2}, \ldots, \frac{S_{n-1}}{S_n}, S_n\right).
\]
This, together with the result about the generators of $\Inv_e$ and the corresponding result for $\Inv_{\ov{e}}$, implies that $\Inv_{e \ov{e}}$ is generated by the shape invariants.

\subsection{New results}
\label{sec:intro main}

In this paper, the main object of study is the field $\Inv_{R\ov{e}}$. On the one hand, we can view this as the field of $\ov{e}$-invariant (ratios of) loop symmetric functions. On the other hand, the results described in the previous section imply that $\Inv_{\ov{e}}$ is the field of rational functions in the entries $z'_{j,i}$ of the $Q$-pattern, so $\Inv_{R\ov{e}}$ can be viewed as the subfield of rational functions in the $z'_{j,i}$ that are invariant under the action of $S_m$ on the $Q$-pattern.

It turns out that the former perspective is more useful for producing elements of $\Inv_{R \ov{e}}$. This is due to a result from~\cite{LPaff} that says that the loop skew Schur function associated to a colored skew shape is $\ov{e}$-invariant if all its northwest corners have color $n$ and all its southeast corners have color (congruent to) $m$. The shape invariants $S_k$ have this property, as the reader can see in the above example. Another example is the ``stretched staircase'' loop Schur function $s^{(n)}_{(n-1)\delta_{m-1}}$, where $\delta_{m-1} = (m-1, \ldots, 2, 1)$. It was shown in~\cite{LP13} that this polynomial tropicalizes to the \defn{energy function}, an important function on tensor products of certain affine crystals appearing in statistical mechanics~\cite{HKOTY99,HKOTT02,KKMMNN91,KSS02} and representation theory~\cite{FSS07,ST12}. Because of this example, we call a loop skew Schur function satisfying the corner color condition a \defn{pseudo-energy}.

Since $\Inv_{R \ov{e}}$ is a subfield of $\Inv_{\ov{e}}$ fixed by a finite group of automorphisms, its transcendence degree over $\CC$ is the same as that of $\Inv_{\ov{e}}$, namely, the number of entries in the $Q$-pattern. This is equal to the number of shape invariants plus the number of $\lE{i}{j}$ such that $i + j \leq m$. The loop elementary symmetric function $\lE{i}{j}$ is the loop Schur function associated to a column of height $i$, with top cell of color $j$, so most of these functions are not pseudo-energies. By attaching the colored shapes of shape invariants to both sides of the column (if necessary), we obtain pseudo-energies $\lQ{i}{j}$, which we call \defn{$Q$-invariants}. The shapes of these loop skew Schur functions in the case $m=4,n=3$ are shown below:

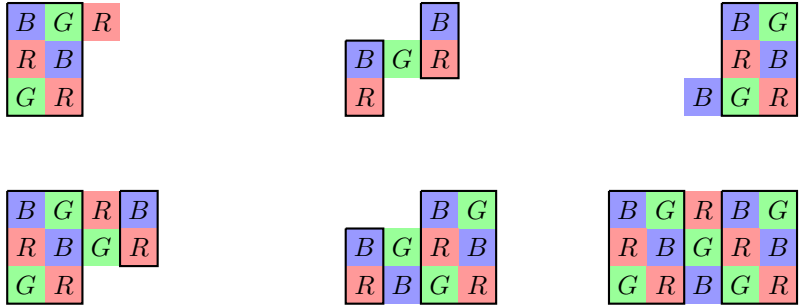
\begin{center}
\begin{tikzpicture}[baseline=-40,scale=0.5]

\begin{scope}[xshift=0cm]
\foreach \x/\y in {2/1,3/2} {\fill[blue!40] (\x,-\y) rectangle (\x+1,-\y-1); \draw (\x+0.5,-\y-0.5) node {$B$};}
\foreach \x/\y in {3/1,2/3} {\fill[green!40] (\x,-\y) rectangle (\x+1,-\y-1); \draw (\x+0.5,-\y-0.5) node {$G$};}
\foreach \x/\y in {4/1,2/2,3/3} {\fill[red!40] (\x,-\y) rectangle (\x+1,-\y-1); \draw (\x+0.5,-\y-0.5) node {$R$};}
\draw[thick] (2,-1) -- (4,-1) -- (4,-4) -- (2,-4) -- (2,-1);
\end{scope}

\begin{scope}[xshift=9cm,yshift=-1cm]
\foreach \x/\y in {2/1,4/0} {\fill[blue!40] (\x,-\y) rectangle (\x+1,-\y-1); \draw (\x+0.5,-\y-0.5) node {$B$};}
\foreach \x/\y in {3/1} {\fill[green!40] (\x,-\y) rectangle (\x+1,-\y-1); \draw (\x+0.5,-\y-0.5) node {$G$};}
\foreach \x/\y in {4/1,2/2} {\fill[red!40] (\x,-\y) rectangle (\x+1,-\y-1); \draw (\x+0.5,-\y-0.5) node {$R$};}
\draw[thick] (2,-1) -- (3,-1) -- (3,-3) -- (2,-3) --  (2,-1);
\draw[thick] (4,0) -- (5,0) -- (5,-2) -- (4,-2) -- (4,0);
\end{scope}

\begin{scope}[xshift=19cm]
\foreach \x/\y in {2/1,1/3,3/2} {\fill[blue!40] (\x,-\y) rectangle (\x+1,-\y-1); \draw (\x+0.5,-\y-0.5) node {$B$};}
\foreach \x/\y in {3/1,2/3} {\fill[green!40] (\x,-\y) rectangle (\x+1,-\y-1); \draw (\x+0.5,-\y-0.5) node {$G$};}
\foreach \x/\y in {2/2,3/3} {\fill[red!40] (\x,-\y) rectangle (\x+1,-\y-1); \draw (\x+0.5,-\y-0.5) node {$R$};}
\draw[thick] (2,-1) -- (4,-1) -- (4,-4) -- (2,-4) -- (2,-1);
\end{scope}

\begin{scope}[xshift=9cm,yshift=-6cm]
\foreach \x/\y in {2/1,5/1,3/2,4/0} {\fill[blue!40] (\x,-\y) rectangle (\x+1,-\y-1); \draw (\x+0.5,-\y-0.5) node {$B$};}
\foreach \x/\y in {3/1,4/2,5/0} {\fill[green!40] (\x,-\y) rectangle (\x+1,-\y-1); \draw (\x+0.5,-\y-0.5) node {$G$};}
\foreach \x/\y in {4/1,2/2,5/2} {\fill[red!40] (\x,-\y) rectangle (\x+1,-\y-1); \draw (\x+0.5,-\y-0.5) node {$R$};}
\draw[thick] (2,-1) -- (3,-1) -- (3,-3) -- (2,-3) -- (2,-1);
\draw[thick] (4,0) -- (6,0) -- (6,-3) -- (4,-3) -- (4,0);
\end{scope}

\begin{scope}[xshift=0cm,yshift=-5cm]
\foreach \x/\y in {2/1,5/1,3/2} {\fill[blue!40] (\x,-\y) rectangle (\x+1,-\y-1); \draw (\x+0.5,-\y-0.5) node {$B$};}
\foreach \x/\y in {3/1,4/2,2/3} {\fill[green!40] (\x,-\y) rectangle (\x+1,-\y-1); \draw (\x+0.5,-\y-0.5) node {$G$};}
\foreach \x/\y in {4/1,2/2,5/2,3/3} {\fill[red!40] (\x,-\y) rectangle (\x+1,-\y-1); \draw (\x+0.5,-\y-0.5) node {$R$};}
\draw[thick] (2,-1) -- (4,-1) -- (4,-4) -- (2,-4) -- (2,-1);
\draw[thick] (5,-1) -- (6,-1) -- (6,-3) -- (5,-3) -- (5,-1);
\end{scope}

\begin{scope}[xshift=16cm,yshift=-5cm]
\foreach \x/\y in {2/1,5/1,3/2,6/2,4/3} {\fill[blue!40] (\x,-\y) rectangle (\x+1,-\y-1); \draw (\x+0.5,-\y-0.5) node {$B$};}
\foreach \x/\y in {3/1,4/2,2/3,6/1,5/3} {\fill[green!40] (\x,-\y) rectangle (\x+1,-\y-1); \draw (\x+0.5,-\y-0.5) node {$G$};}
\foreach \x/\y in {4/1,2/2,5/2,3/3,6/3} {\fill[red!40] (\x,-\y) rectangle (\x+1,-\y-1); \draw (\x+0.5,-\y-0.5) node {$R$};}
\draw[thick] (2,-1) -- (4,-1) -- (4,-4) -- (2,-4) -- (2,-1);
\draw[thick] (5,-1) -- (7,-1) -- (7,-4) -- (5,-4) -- (5,-1);
\end{scope}

\end{tikzpicture}\ 
\end{center}

We conjecture that the $Q$-invariants and shape invariants generate $\Inv_{R \ov{e}}$. By symmetry, we have a conjectural generating set for $\Inv_{\ov{R} e}$ consisting of shape invariants and \defn{barred $P$-invariants} $\bP{j}{i}$. We prove that the set of $Q$-invariants, barred $P$-invariants, and shape invariants is algebraically independent, which makes it natural to conjecture that this set generates $\Inv_{R \ov{R}}$.

Looking at the (conjectural) picture of generating sets in Figure~\ref{fig:nine}, two questions immediately present themselves:
\begin{enumerate}
\item How does one express a pseudo-energy in terms of the $Q$-invariants and shape invariants?
\item How does one express a pseudo-energy in terms of the generators $\{\bE{j}{i} \mid j + i > n\}$ of $\Inv_{\ov{e}}$? In particular, how does one express the $Q$-invariants in terms of these generators?
\end{enumerate}

The answers to these questions are the two main results of this paper. To answer the first question, we express each pseudo-energy as the determinant of a matrix consisting of $Q$-invariants, shape invariants, and inverses of shape invariants (Theorem~\ref{thm:new det formula}). This formula is closely related to the Jacobi--Trudi formula, a classical determinantal formula for Schur functions that generalizes to $\LSym$. We view this result as strong evidence that the $Q$-invariants and shape invariants do in fact generate $\Inv_{R \ov{e}}$, and as justification for choosing these elements as a distinguished (conjectural) generating set.

The second question is resolved by Theorem~\ref{thm:unfolded sum of minors}, which expresses each pseudo-energy as a polynomial with nonnegative coefficients in barred loop skew Schur functions. The barred loop skew Schur functions that appear in these expressions are products of determinants of an $m \times m$ matrix consisting of the generators $\{\bE{j}{i} \mid j + i > n\}$ of $\Inv_{\ov{e}}$. The proof of Theorem~\ref{thm:unfolded sum of minors} is bijective: by representing the relevant tableaux as non-intersecting paths in a planar network, we are able to equate polynomials in $\LSym$ with polynomials in $\ov{\LSym}$.

\subsection{Pseudo-energies as functions on the $Q$-pattern}
\label{sec:intro cocharge}

Theorem~\ref{thm:unfolded sum of minors} implies that each pseudo-energy is a Laurent polynomial with positive integer coefficients in the entries $z'_{j,i}$ of the $Q$-pattern (this is explained in detail in \S\ref{sec:back to tableau}). Because of the positivity, we may tropicalize these Laurent polynomials and obtain $S_m$-invariant piecewise-linear functions on a Gelfand--Tsetlin pattern. The most important example of an $S_m$-invariant function on tableaux is the \defn{cocharge} function, which was introduced by Lascoux and Sch\"utzenberger in order to give a combinatorial rule for the transition coefficients between the bases of Schur functions and Hall--Littlewood symmetric functions~\cite{LS78}. It was shown by Nakayashiki and Yamada that the energy of a tensor product of one-row $\GL_n$-crystals is sent, under RSK, to the cocharge of the $Q$-tableau~\cite{NY97}.

Kirillov and Berenstein~\cite{KB95} defined a Laurent polynomial in the $z'_{j,i}$ that conjecturally tropicalizes to cocharge. By using Nakayashiki and Yamada's result and a new formula for energy coming from Theorem~\ref{thm:folded sum of minors} (a ``cylindric analogue'' of Theorem~\ref{thm:unfolded sum of minors}), we are able to prove Kirillov and Berenstein's conjecture. (An outline of a different proof was provided by Kirillov in~\cite{Kir01}.)

Kirillov and Berenstein also posed the question of finding a generating set for all $S_m$-invariant piecewise-linear functions on a Gelfand--Tsetlin pattern~\cite{KB95}. If our conjecture about the generators of $\Inv_{R \ov{e}}$ holds, it would follow that the field of $S_m$-invariant rational functions in the $z'_{j,i}$ is generated by the $Q$-invariants and shape invariants (viewed as functions of the $z'_{j,i}$). It seems possible that the tropicalizations of the $Q$-invariants and shape invariants generate the semifield of $S_m$-invariant piecewise-linear functions on a Gelfand--Tsetlin pattern. Even if our conjecture is true, however, it may well be necessary to include the tropicalizations of additional pseudo-energies, since psuedo-energies cannot in general be expressed as positive rational functions in the $Q$-invariants and shape invariants.

Another important $S_m$-invariant function on tableaux is the \defn{decoration} (or \defn{potential}), which was introduced by Berenstein and Kazhdan~\cite{BK07} as a mechanism for ``cutting out'' the vertex sets of combinatorial crystals in the tropicalization of a geometric crystal. The decoration also plays a role in writing Whittaker functions as integrals over geometric crystals (see~\cite{Chhaibi13, Lam13, OSZ}). Berenstein and Kazhdan introduced a related function called the \defn{central charge} on the product of geometric crystals, which they used to define the notion of a strongly positive geometric crystal~\cite{BK07}. We show that both of these functions have simple expressions in terms of $Q$-invariants and shape invariants.

\subsection{Outline of paper}
\label{sec:intro outline}

In \S\ref{sec:geometric crystals}, we recall the definitions and main properties of the geometric crystals $X^{\Mat}$ and $X^{\GT}$, the geometric $R$-matrices, and geometric RSK (the paper~\cite{BFPS} gives a much more comprehensive treatment of these topics). In \S\ref{sec:single}, we review the definition and basic properties of the ring of loop symmetric functions, and the above-mentioned results about $\Inv_e$, $\Inv_{\ov{e}}$, and $\Inv_{e\ov{e}}$.

In \S\ref{sec:double}, we introduce $Q$-invariants and prove our first main result (Theorem~\ref{thm:new det formula}), which gives a determinantal formula for pseudo-energies in terms of $Q$-invariants and shape invariants. In \S\ref{sec:unfolded sum of minors}, we prove our second main result (Theorem~\ref{thm:unfolded sum of minors}), which expresses pseudo-energies in terms of barred loop skew Schur functions. In \S\ref{sec:back to tableau}, we consider pseudo-energies as functions on the $Q$-pattern. In \S\ref{sec:cylindric}, we prove analogues of our two main theorems for the class of $\ov{e}$-invariant \defn{cylindric loop Schur functions}, which are generating functions for tableaux embedded in a cylinder. The second of these analogues (Theorem~\ref{thm:folded sum of minors}) is the key to our application to energy and cocharge in \S\ref{sec:applications}.

In \S\ref{sec:applications}, we apply our results to obtain new formulas for the energy and cocharge functions, and for the decoration and central charge. Finally, in \S\ref{sec:final}, we briefly discuss an affine generalization of our questions, and review what is known about the invariants of $\CC[x_i^j]$ under the usual actions of $\SL_n,\SL_m,S_n,$ and $S_m$ by left or right matrix multiplication, and permutation of rows or columns.

\subsection*{Note on citations}
In our previous paper~\cite{BFPS}, we sought to give a largely self-contained presentation of geometric crystals, geometric $R$-matrices, and geometric RSK. To this end, we included proofs of several key results that appeared previously in the literature. When such results are cited below, we include, for the convenience of the reader, a reference to our previous paper in addition to an original (or at least earlier) reference.

\subsection*{Notation}
For integers $a,b$, we write $[a,b]$ for the interval $\{k \in \ZZ \mid a \leq k \leq b\}$. We often abbreviate $[1,b]$ to $[b]$. For a set $S$, we write $\binom{S}{k}$ for the set of $k$-element subsets of $S$.

\subsection*{Acknowledgements}

We would like to thank Thomas Lam, during discussions with whom some of the ideas in this paper were born.

This work benefited from computations using {\sc SageMath}~\cite{sage,combinat}.
B.B.~was supported in part by NSF grant DMS-2101392.
G.F.~was supported in part by the Canada Research Chairs program.
P.P.~was supported in part by NSF grant DMS-1949896.
T.S.~was supported in part by Grant-in-Aid for JSPS Fellows 21F51028.
This work was partly supported by Osaka City University Advanced Mathematical Institute (MEXT Joint Usage/Research Center on Mathematics and Theoretical Physics JPMXP0619217849).

\section{Background on geometric crystals and geometric RSK}
\label{sec:geometric crystals}

\subsection{Basic geometric crystal}
\label{sec:basic}

Let $T = (\CC^*)^m$ be the maximal torus of $\GL_m$. For $i \in [m-1]$, let $\alpha_i \colon T \rightarrow \CC^*$ and $\alpha_i^\vee \colon \CC^* \rightarrow T$ denote the character and co-character
\[
\alpha_i(x_1, \ldots, x_m) = \dfrac{x_i}{x_{i+1}} \qquad \text{ and } \qquad \alpha_i^\vee(c) = (1, \ldots, c, c^{-1}, \ldots, 1),
\]
where $c$ and $c^{-1}$ are in positions $i$ and $i+1$.

\begin{dfn}
\label{defn:geom crystal}
A \defn{$\GL_m$-geometric crystal} is an irreducible complex algebraic variety $X$, together with a rational map $\gamma \colon X \rightarrow T$, and for each $i \in [m-1]$, rational functions\footnote{Berenstein and Kazhdan's definition~\cite{BK00,BK07} differs from ours by replacing $\ve_i, \vp_i$ with $1/\ve_i, 1/\vp_i$. Our convention has the advantage that $\ve_i, \vp_i$ tropicalize to the corresponding functions $\tw{\ve}_i, \tw{\vp}_i$ on combinatorial crystals, rather than their negatives.} $\ve_i, \vp_i \colon X \rightarrow \CC$ and a rational $\CC^*$-action $e_i : \CC^* \times X \rightarrow X$. We write $e_i^c(x)$ for the action of $c \in \CC^*$ on $x \in X$. These maps must satisfy the following identities (when they are defined):
\begin{enumerate}
\item $\dfrac{\vp_i(x)}{\ve_i(x)} = \alpha_i\bigl(\gamma(x)\bigr)$;
\item
$\gamma\bigl(e_i^c(x)\bigr) = \alpha_i^\vee(c) \gamma(x), \qquad \ve_i\bigl(e_i^c(x)\bigr) = c^{-1} \ve_i(x), \qquad \vp_i\bigl(e_i^c(x)\bigr) = c \vp_i(x);$
\item
\begin{enumerate}
\item if $\abs{i-j} > 1$, then $e_i^c e_j^{c'} = e_j^{c'} e_i^c$;
\item if $\abs{i-j} = 1$, then $e_i^c e_j^{cc'} e_i^{c'} = e_j^{c'} e_i^{cc'} e_j^c$.
\end{enumerate}
\end{enumerate}
The maps $e_i^c$ are called \defn{geometric crystal operators}.

An \defn{isomorphism} of $\GL_m$-geometric crystals is a birational isomorphism of underlying varieties that commutes with the geometric crystal structures.
\end{dfn}

The geometric analogue of the $\GL_m$-representation $\Sym \CC^m$ is the \defn{basic geometric crystal} $X_m$ introduced in~\cite{KOTY03}, which has underlying variety $X_m  = (\CC^*)^m$, and the following geometric crystal structure: for $\xx = (x_1, \ldots, x_m) \in X_m$,
\begin{equation}
\label{eq:basic geometric crystal}
\gamma(\mb{x}) = \mb{x},
\qquad\quad
\varepsilon_i(\mb{x}) = x_{i+1},
\qquad\quad
\varphi_i(\mb{x}) = x_i,
\qquad\quad
e_i^c(\mb{x}) = (x_1, \dotsc, c x_i, c^{-1} x_{i+1}, \dotsc, x_m).
\end{equation}

Given two $\GL_m$-geometric crystals $X$ and $X'$, Berenstein and Kazhdan~\cite{BK00} define on $X \times X'$
\begin{equation}
\label{eq:functions prod}
\gamma(x, x') = \gamma(x) \gamma(x'),
\quad
\varepsilon_i(x, x') = \frac{\varepsilon_i(x) \varepsilon_i(x')}{\varepsilon_i(x) + \varphi_i(x')},
\quad
\varphi_i(x, x') = \frac{\varphi_i(x) \varphi_i(x')}{\varepsilon_i(x) + \varphi_i(x')},
\end{equation}
and
\begin{equation}
\label{eq:e prod}
e_i^c(x, x') = \bigl( e_i^{c^+}(x), e_i^{c/c^+}(x') \bigr),
\qquad \text{ where } \,
c^+ = \frac{c \varphi_i(x') + \varepsilon_i(x)}{\varphi_i(x') + \varepsilon_i(x)}.
\end{equation}
This definition makes the product $(X_m)^n$ into a $\GL_m$-geometric crystal.\footnote{It is not at all obvious that axiom (3) of Definition~\ref{defn:geom crystal} holds for $(X_m)^n$. The proof of this fact relies on the notion of \defn{unipotent crystal}~\cite{BK00}; see~\cite[\S 2.2]{BFPS} for an overview of unipotent crystals.} Using this definition, one can derive explicit closed formulas for the geometric crystal structure on $(X_m)^n$ (see \cite[\LemExplicitBasic]{BFPS}). It is usually more convenient, however, to work with a matrix formulation of the geometric crystal structure, which we now recall.

For $(x_1, \ldots, x_m) \in X_m$, define an $m \times m$ matrix
\[
W(x_1, \ldots, x_m) = \begin{pmatrix}
x_1 & 0 & 0 & \cdots & 0 & 0 \\
1 & x_2 & 0 & \cdots & 0 & 0 \\
0 & 1 & x_3 & \cdots & 0 & 0 \\
 &&& \ddots && \\
0 & 0 & 0 & \cdots & x_{m-1} & 0 \\
0 & 0 & 0 & \cdots & 1 & x_m
\end{pmatrix}.
\]
Given points $\mb{x}^j \in X_m$ for $j = 1, \ldots, n$, define
\[
M(\mb{x}^1, \ldots, \mb{x}^n) = W(\mb{x}^1) \cdots W(\mb{x}^n).
\]
The following result is a consequence of Berenstein and Kazhdan's general theory of unipotent crystals.

\begin{lemma}[{\cite[Lem.~3.9]{BK00}}]
\label{lem:unipotent}
For $\xx = (\xx^1, \ldots, \xx^n) \in (X_m)^n$, the maps $\gamma, \ve_i, \vp_i$ are given by
\[
\gamma(\xx) = (M(\xx)_{1,1}, \ldots, M(\xx)_{m,m}),
\qquad\quad
\varepsilon_i(\xx) = \dfrac{M(\xx)_{i+1,i+1}}{M(\xx)_{i+1,i}},
\qquad\quad
\varphi_i(\xx) = \dfrac{M(\xx)_{i,i}}{M(\xx)_{i+1,i}},
\]
and the geometric crystal operators satisfy
\begin{equation}
\label{eq:unipotent}
M(e_i^c(\mb{x})) = x_i\left((c-1)\varphi_i(\mb{x})\right) \cdot M(\mb{x}) \cdot x_i\left((c^{-1} - 1)\varepsilon_i(\mb{x})\right),
\end{equation}
where $x_i(a) = I + aE_{i,i+1}$ is the $m \times m$ matrix with 1's on the diagonal, $a$ in position $(i,i+1)$, and zeroes elsewhere.
\end{lemma}

Let $\xx = (x_i^j)_{i \in [m], j \in [n]} \in \Mat_{m \times n}(\CC^*)$ be an $m \times n$ matrix with nonzero complex entries. Write $\mb{x}_i = (x_i^1, \ldots, x_i^n)$ for the $i$th row of $\xx$, and $\mb{x}^j = (x_1^j, \ldots, x_m^j)$ for the $j$th column of $\xx$. We view $\xx$ as a point $(\mb{x}^1, \ldots, \mb{x}^n)$ in the $\GL_m$-geometric crystal $(X_m)^n$, and as a point $(\mb{x}_1, \ldots, \mb{x}_m)$ in the $\GL_n$-geometric crystal $(X_n)^m$. Thus, we have a $\GL_m$-geometric crystal structure and a $\GL_n$-geometric crystal structure acting on the space $\Mat_{m \times n}(\CC^*)$. To distinguish the two structures, we write bars over the maps associated with the $\GL_n$-geometric crystal: $\ov{\gamma}, \ov{\varepsilon}_j, \ov{\varphi}_j, \ov{e}_j^c$. Note that the unbarred geometric crystal operators $e_i^c$ modify the entries in rows $i$ and $i+1$, and the barred geometric crystal operators $\ov{e}_j^c$ modify the entries in columns $j$ and $j+1$.

\begin{thm}[{\cite{LP13II}, \cite[\CorBicrystal]{BFPS}}]
\label{thm:bicrystal}
Taken together, the barred and unbarred geometric crystal structures make $\Mat_{m \times n}(\CC^*)$ into a $\GL_n \times \GL_m$-geometric crystal.\footnote{A $\GL_n \times \GL_m$-geometric crystal is a variety with a $\GL_n$-geometric crystal structure and a $\GL_m$-geometric crystal structure, such that the two structures commute with each other; that is, $\ve_i \ov{e}^c_j = \ve_i$, $\ov{\ve}_j e_i^c = \ov{\ve}_j$, $e_i^{c_1} \ov{e}_j^{c_2} = \ov{e}_j^{c_2} e_i^{c_1}$, etc.}
\end{thm}

\subsection{Weyl group action and geometric \texorpdfstring{$R$}{R}-matrix}
\label{sec:geometric R}

Let $X$ be a $\GL_m$-geometric crystal. For $i \in [m-1]$, define a rational map $s_i \colon X \to X$ by
\[
s_i(x) = e_i^{\frac{1}{\alpha_i(\gamma(x))}}(x) = e_i^{\frac{\ve_i(x)}{\vp_i(x)}}(x).
\]
Berenstein and Kazhdan proved that the maps $s_i$ generate a birational action of the Weyl group $S_m$ on $X$~\cite[Prop. 2.3]{BK00} (in fact, the purpose of axiom (3) in Definition~\ref{defn:geom crystal} is to ensure that the $s_i$ satisfy the braid relations). It turns out that the birational Weyl group action on $(X_m)^n$ agrees with the action of the geometric $R$-matrix, a map coming from the representation theory of quantum groups of type $A_{n-1}^{(1)}$.

The \defn{geometric $R$-matrix (of type $A_{n-1}^{(1)}$)} is the rational map
\begin{align*}
R \colon (\CC^*)^n \times (\CC^*)^n &\rightarrow (\CC^*)^n \times (\CC^*)^n \\
\bigl((x_1, \ldots, x_n), (y_1, \ldots, y_n)\bigr) &\mapsto \bigl((y'_1, \ldots, y'_n), (x'_1, \ldots, x'_n)\bigr)
\end{align*}
given by
\[
y'_j = y_j \dfrac{\kappa_{j+1}}{\kappa_j}, \quad x'_j = x_j\dfrac{\kappa_j}{\kappa_{j+1}}, \qquad \kappa_r = \kappa_r(\mb{x},\mb{y}) = \sum_{k=0}^{n-1} y_r \cdots y_{r+k-1} x_{r+k+1} \cdots x_{r+n-1},
\]
with subscripts interpreted modulo $n$. For $i \in [m-1]$, define $R_i \colon (X_n)^m \rightarrow (X_n)^m$ by
\[
R_i(\mb{x}_1, \ldots, \mb{x}_i, \mb{x}_{i+1}, \ldots, \mb{x}_m) = (\mb{x}_1, \ldots, \mb{x}'_{i+1}, \mb{x}'_i, \ldots, \mb{x}_m),
\]
where $(\mb{x}'_{i+1}, \mb{x}'_i) = R(\mb{x}_i, \mb{x}_{i+1})$.

\begin{prop}[{\cite{KajNouYam}, \cite[\PropWeylR]{BFPS}}]
\label{prop:R=Weyl}
If $(\mb{x}^1, \ldots, \mb{x}^n) \in (X_m)^n$ are the columns of a matrix $\xx \in \Mat_{m \times n}(\CC^*)$, and $(\mb{x}_1, \ldots, \mb{x}_m) \in (X_n)^m$ are the rows of $\xx$, then for $i \in [m-1]$ we have
\[
s_i(\mb{x}^1, \ldots, \mb{x}^n) = R_i(\mb{x}_1, \ldots, \mb{x}_m).
\]
\end{prop}

Following the convention of the previous section, we write $R_i$ (resp., $\ov{R}_j$) for the birational automorphism of $\Mat_{m \times n}(\CC^*)$ which acts by the geometric $R$-matrix on rows $i$ and $i+1$ (resp., columns $j$ and $j+1$). It follows from Theorem~\ref{thm:bicrystal} and Proposition~\ref{prop:R=Weyl} that these maps generate a rational action of $S_n \times S_m$ on $\Mat_{m \times n}(\CC^*)$.

\begin{remark}
The \defn{combinatorial $R$-matrix} is the unique crystal isomorphism that permutes adjacent factors in a tensor product of single row (affine) crystals (see, e.g.,~\cite{HKOTY99,HKOTT02}). As mentioned in the introduction, this map corresponds, on the other side of RSK, to Lascoux and Sch\"utzenberger's symmetric group action~\cite{LS81} on the recording tableau (see, e.g.,~\cite{Shimozono02}).

Yamada~\cite{Yamada01} defined the geometric (or birational) $R$-matrix as a geometric lift of the combinatorial $R$-matrix. Several different proofs of the fact that the maps $R_i$ generate a birational $S_m$-action have appeared in the literature~\cite{Yamada01,KajNouYam,Etingof,LP12,LP13II}. The fact that the maps $R_i$ commute with the maps $\ov{R}_j$ was proved in~\cite{KajNouYam}.
\end{remark}

\subsection{Gelfand--Tsetlin geometric crystal}
\label{sec:GT}

A \defn{Gelfand--Tsetlin (GT) pattern of height $n$} is a triangular array of nonnegative integers $(a_{i,j}), 1 \leq i \leq j \leq n,$ satisfying the inequalities 
\begin{equation}
\label{eq:GT def}
a_{i,j+1} \geq a_{i,j} \geq a_{i+1,j+1}
\end{equation}
whenever $j < n$. We will represent GT patterns as triangles whose entries weakly increase from right to left along every diagonal, as illustrated below in the case $n = 4$:
\[
\begin{array}{ccccccc}
&&& a_{11} \\
&& a_{12} && a_{22} \\
& a_{13} && a_{23} && a_{33} \\
a_{14} && a_{24} && a_{34} && a_{44}
\end{array} \, .
\]

GT patterns of height $n$ are in bijection with semistandard tableaux consisting of entries at most $n$: the partition $(a_{1,j}, \ldots, a_{j,j})$ in the $j$th row corresponds to the shape formed by the entries less than or equal to $j$ in the tableau. In particular, the bottom row $(a_{1,n}, \ldots, a_{n,n})$ is the shape of the tableau, so we define the \defn{shape} of the GT pattern to be this $n$-tuple. Another description of the bijection is that the number of $j$'s in the $i$th row of the tableau is $a_{i,j} - a_{i,j-1}$ (with $a_{i,i-1} = 0$).

\ytableausetup{centertableaux}
\begin{ex}
Here is a GT pattern of height $4$ and the corresponding semistandard tableau:
\[
\begin{array}{ccccccc}
&&& 3 \\
&& 6 && 1 \\
& 6 && 4 && 1 \\
8 && 5 && 3 && 0
\end{array}
\quad \longleftrightarrow \quad
\ytableaushort{11122244,23334,344} \;.
\]
\end{ex}

More generally, define a \defn{Gelfand--Tsetlin pattern of height $n$ and width $m$} to be a trapezoidal array of nonnegative integers $(a_{i,j})_{1 \leq i \leq m, i \leq j \leq n}$ which satisfy the inequalities~\eqref{eq:GT def} whenever both sides of an inequality are in the array. For example, a GT pattern of height 5 and width 3 looks like this:
\[
\begin{array}{ccccccc}
&&&& a_{11} \\
&&& a_{12} && a_{22} \\
&& a_{13} && a_{23} && a_{33} \\
& a_{14} && a_{24} && a_{34} \\
a_{15} && a_{25} && a_{35}
\end{array} \; .
\]
These patterns are in bijection with semistandard tableaux with entries at most $n$ and at most $m$ rows; again, we define the shape of the GT pattern to be the bottom row $(a_{1,n}, \ldots, a_{\min(m,n),n})$, since this is the shape of the corresponding tableau. Note that if $m \geq n$, then a GT pattern of height $n$ and width $m$ is simply a GT pattern of height $n$.

At the geometric level, we consider the torus
\[
\GT_n^{\leq m} = \{\mb{z} = (z_{i,j}) \mid z_{i,j} \in \CC^*, 1 \leq i \leq m, i \leq j \leq n\}.
\]
We refer to points of $\GT_n^{\leq m}$ as \defn{patterns}, and we define the \defn{shape} of $\mb{z} \in \GT_n^{\leq m}$ to be the vector
\[
\sh(\mb{z}) = (z_{1,n}, \ldots, z_{p,n}) \in (\CC^*)^p,
\]
where $p = \min(m,n)$.

We now give the $\GL_n$-geometric crystal structure on $\GT_n^{\leq m}$ following~\cite{BFPS}. For $1 \leq i \leq n$, define an $n \times n$ matrix
\begin{equation}
\label{eq:row matrix}
W^i(z_i, \ldots, z_n) = \sum_{k = 1}^{i-1} E_{kk} + \sum_{k = i}^n z_k E_{kk} + \sum_{k = i}^{n-1} E_{k+1,k},
\end{equation}
where $E_{ij}$ is the matrix unit with $1$ in position $(i,j)$, and zeroes elsewhere. Given $\mb{z} = (z_{i,j}) \in \GT_n^{\leq m}$, define $\Phi_n^{\leq m}(\mb{z}) \in \GL_n$ by
\[
\Phi_n^{\leq m}(\mb{z}) = W^p\left(z_{p,p}, \frac{z_{p,p+1}}{z_{p,p}}, \ldots, \frac{z_{p,n}}{z_{p,n-1}}\right) \cdots W^1\left(z_{1,1}, \frac{z_{1,2}}{z_{1,1}}, \ldots, \frac{z_{1,n}}{z_{1,n-1}}\right).
\]

\begin{ex}
\label{ex_Phi}
For $n=4$ and $m = 2$, we have
\[
\Phi_4^{\leq 2}(\zz) =
\begin{pmatrix}
1 & 0 & 0 & 0 \\
0 & z_{22} & 0 & 0 \\
0 & 1 & \frac{z_{23}}{z_{22}} & 0 \\
0 & 0 & 1 & \frac{z_{24}}{z_{23}}
\end{pmatrix}
\begin{pmatrix}
z_{11} & 0 & 0 & 0 \\
1 & \frac{z_{12}}{z_{11}} & 0 & 0 \\
0 & 1 & \frac{z_{13}}{z_{12}} & 0 \\
0 & 0 & 1 & \frac{z_{14}}{z_{13}}
\end{pmatrix}
=
\begin{pmatrix}
z_{11} & 0 & 0 & 0 \\
z_{22} & \frac{z_{12}z_{22}}{z_{11}} & 0 & 0 \\
1 & \frac{z_{12}}{z_{11}} + \frac{z_{23}}{z_{22}} & \frac{z_{13}z_{23}}{z_{12}z_{22}} & 0 \\
0 & 1 & \frac{z_{13}}{z_{12}} + \frac{z_{24}}{z_{23}} & \frac{z_{14}z_{24}}{z_{13}z_{23}}
\end{pmatrix}.
\]
\end{ex}

\begin{figure}
\begin{center}
\begin{tikzpicture}

\foreach \a/\b/\c/\d in {0/1/1/0, 0/2/2/0, 0/3/3/0, 0/4/4/0, 1/3/1/0, 2/2/2/0, 3/1/3/0} {\draw (\a,\b) -- (\c,\d);}
\foreach \a/\b in {0/1,0/2,0/3,0/4} {\filldraw (\a,\b) circle[radius=.04cm] node[left]{$\b$}; \filldraw (\b,\a) circle[radius=.04cm] node[below]{$\b'$};}
\foreach \a in {1,2,3,4} {\draw (0.6,\a-0.2) node{$z_{\a\a}$};}
\foreach \a/\b in {1/2,2/3,3/4} {\draw (1.6,\a-0.2) node{$\frac{z_{\a\b}}{z_{\a\a}}$};}
\foreach \a/\b/\c in {1/2/3,2/3/4} {\draw (2.6,\a-0.2) node{$\frac{z_{\a\c}}{z_{\a\b}}$};}
\draw (3.6,0.8) node{$\frac{z_{14}}{z_{13}}$};

\begin{scope}[xshift=7cm]
\foreach \a/\b/\c/\d in {0/1/1/0, 0/2/2/0, 0/3/3/0, 0/4/4/0, 1/4/5/0, 2/4/6/0, 3/4/7/0, 1/4/1/0, 2/4/2/0, 3/4/3/0, 4/3/4/0, 5/2/5/0, 6/1/6/0} {\draw (\a,\b) -- (\c,\d);}
\foreach \a/\b in {0/1,0/2,0/3,0/4} {\filldraw (\a,\b) circle[radius=.04cm] node[left]{$\b$}; \filldraw (\b,\a) circle[radius=.04cm] node[below]{$\b'$};}
\foreach \a/\b/\c in {1/4/5, 2/4/6, 3/4/7} {\filldraw (\a,\b) circle[radius=.04cm] node[above]{$\c$};}
\foreach \a/\b in {0/4, 0/5, 0/6, 0/7} {\filldraw (\b,\a) circle[radius=.04cm] node[below]{$\b'$};}
\foreach \a in {1,2,3,4} {\draw (0.6,\a-0.2) node{$z_{\a\a}$};}
\foreach \a/\b in {1/2,2/3,3/4,4/5} {\draw (1.6,\a-0.2) node{$\frac{z_{\a\b}}{z_{\a\a}}$};}
\foreach \a/\b/\c in {1/2/3,2/3/4,3/4/5,4/5/6} {\draw (2.6,\a-0.2) node{$\frac{z_{\a\c}}{z_{\a\b}}$};}
\foreach \a/\b/\c in {1/3/4,2/4/5,3/5/6,4/6/7} {\draw (3.6,\a-0.2) node{$\frac{z_{\a\c}}{z_{\a\b}}$};}
\foreach \a/\b/\c in {1/4/5,2/5/6,3/6/7} {\draw (4.6,\a-0.2) node{$\frac{z_{\a\c}}{z_{\a\b}}$};}
\foreach \a/\b/\c in {1/5/6,2/6/7} {\draw (5.6,\a-0.2) node{$\frac{z_{\a\c}}{z_{\a\b}}$};}
\draw (6.6,0.8) node{$\frac{z_{17}}{z_{16}}$};
\end{scope}

\end{tikzpicture}
\end{center}
\caption{Examples of the planar network $\Gamma_n^{\leq m}$. The network on the left is $\Gamma_4$ (i.e., $\Gamma_4^{\leq m}$ for any $m \geq 4$), and the network on the right is $\Gamma_7^{\leq 4}$. Edges are directed to the south and southeast, and all vertical edges have weight 1.}
\label{fig:network}
\end{figure}

Our main tool for computing with the map $\Phi_n^{\leq m}$ is the planar network $\Gamma_n^{\leq m}$ illustrated in Figure~\ref{fig:network} (a formal description of this network is given in \cite[\S 3.2]{BFPS}). Given a planar, edge-weighted, directed network $\Gamma$ with distinguished sources $1, \ldots, n$ and sinks $1', \ldots, n'$, we associate an $n \times n$ matrix $A(\Gamma)$ as follows: the weight of a path in the network is defined to be the sum of the weights of the edges in the path, and the entry $A(\Gamma)_{i,j}$ is defined to be the sum of the weights of all paths from source $i$ to sink $j'$. It is easy to see that $\Phi_n^{\leq m}(\zz) = A(\Gamma_n^{\leq m})$.

We now recall the explicit formula for the (birational) inverse of $\Phi_n^{\leq m}$. Define
\[
(B^-)^{\leq m} = \{A \in \GL_n \mid A_{ij} = 0 \text{ if } i < j \text{ or } i-j > m, \text{ and } A_{ij} = 1 \text{ if } i-j = m\}.
\]
For example, the matrix in Example~\ref{ex_Phi} is in $(B^-)^{\leq 2}$. For an $n \times n$ matrix $A$ and subsets $I,J \subset [n]$ of the same size, let $\Delta_{I,J}(A)$ be the determinant of the submatrix of $A$ consisting of the rows in $I$ and the columns in $J$. When $J$ consists of the first several columns, we will often omit it from the notation, so that $\Delta_I(A) = \Delta_{I,[1,\abs{I}]}(A)$ and we call it a \defn{flag minor}. We use the convention that $\Delta_{\emptyset}(A) = 1$. Define a rational map $\Psi_n^{\leq m} \colon (B^-)^{\leq m} \rightarrow \GT_n^{\leq m}$ by $A \mapsto \mb{z} = (z_{i,j})$, where
\begin{equation}
\label{eq_Psi_def}
z_{i,j} = \frac{\Delta_{[i,j]}(A)}{\Delta_{[i+1,j]}(A)}
\end{equation}
for $1 \leq i \leq m, i \leq j \leq n$.

When $m \geq n$, the dependence on $m$ in the above definitions disappears. In this case, we will sometimes omit the superscript ``$\leq m$'' and write $B^-, \GT_n, \Phi_n, \Psi_n, \Gamma_n$.

\begin{lemma}[{\cite{BFZ}, \cite[\LemPhiPsi]{BFPS}}]
\label{lem:Phi inverse}
The map $\Phi_n^{\leq m}$ is a birational isomorphism from $\GT_n^{\leq m}$ to $(B^-)^{\leq m}$, with birational inverse $\Psi_n^{\leq m}$. In particular, $\Phi_n$ is a birational isomoprhism from $\GT_n$ to $B^-$, with birational inverse $\Psi_n$.
\end{lemma}

\begin{dfn}
\label{defn:GT geom cryst}
Define $\ov{\gamma} \colon \GT_n^{\leq m} \rightarrow (\CC^*)^n$ and $\ov{\ve}_j, \ov{\vp}_j \colon \GT_n^{\leq m} \rightarrow \CC$ (for $j \in [n-1]$) by
\[
\ov{\gamma}(\mb{z}) = (M_{1,1}, \ldots, M_{n,n}), \qquad\qquad \ov{\ve}_j(\mb{z}) = \frac{M_{j+1,j+1}}{M_{j+1,j}}, \qquad\qquad \ov{\vp}_j(\mb{z}) = \frac{M_{j,j}}{M_{j+1,j}},
\]
where $M = \Phi_n^{\leq m}(\mb{z})$. Define $\ov{e}_j \colon \CC^* \times \GT_n^{\leq m} \rightarrow \GT_n^{\leq m}$ by
\begin{equation*}
\ov{e}_j^c(\mb{z}) =  \Psi_n^{\leq m}\bigg(x_j\left((c-1)\ov{\vp}_j(\mb{z})\right) \cdot \Phi_n^{\leq m}(\mb{z}) \cdot x_j\left((c^{-1}-1) \ov{\ve}_j(\mb{z})\right)\bigg).
\end{equation*}
\end{dfn}

As explained in \cite[\S 3.2]{BFPS}, these maps make $\GT_n^{\leq m}$ into a $\GL_n$-geometric crystal, which we call the \defn{Gelfand--Tsetlin geometric crystal}. Explicit formulas for this geometric crystal structure in the case $m \geq n$ are given in \cite[\LemExplicitGT]{BFPS}.

\begin{remark}
The geometric crystal $\GT_n^{\leq 1}$ is isomorphic to the basic $\GL_n$-geometric crystal $X_n$ defined in \S\ref{sec:basic} via the map $(z_{1,1}, \ldots, z_{1,n}) \mapsto (z_{1,1}, z_{1,2}/z_{1,1}, \ldots, z_{1,n}/z_{1,n-1})$.
\end{remark}

The following result shows that a fundamental property of the combinatorial crystal operators on semistandard tableaux carries over to the geometric setting.

\begin{lemma}[{\cite{BK07}, \cite[\LemShapePreserved]{BFPS}}]
\label{lem:shape preserved}
The geometric crystal operators on $\GT_n^{\leq m}$ preserve the shape of a pattern.
\end{lemma}

\subsection{Tropicalization}
\label{sec:trop}

Let $x_1, \ldots, x_d$ be variables. For $I = (i_1, \ldots, i_d) \in (\ZZ_{\geq 0})^d$, let $x^I = x_1^{i_1} \cdots x_d^{i_d}$. A rational function $f \in \CC(x_1, \ldots, x_d)$ is \defn{positive} (or \defn{subtraction-free}) if it is nonzero, and has an expression of the form
\begin{equation}
\label{eq:positive expression}
f(x_1, \ldots, x_d) = \dfrac{\sum_I a_I x^I}{\sum_I b_I x^I}
\end{equation}
with $a_I,b_I \in \RR_{\geq 0}$, and $a_I,b_I = 0$ for all but finitely many $I$. If $f$ is positive, define the \defn{tropicalization} of $f$ to be the piecewise-linear function given by
\[
\Trop(f)(x_1, \ldots, x_d) = \min_{I \colon a_I \neq 0} (i_1x_1 + \cdots + i_dx_d) - \min_{I \colon b_I \neq 0} (i_1x_1 + \cdots + i_dx_d).
\]
More generally, a rational map $f = (f_1, \ldots, f_k) \colon (\CC^*)^d \rightarrow (\CC^*)^k$ is positive if each $f_i$ is positive, and in this case its tropicalization is defined coordinate-wise. We say that a positive rational function $f$ is a \defn{geometric lift} of a piecewise-linear function $F$ if $\Trop(f) = F$; note that every piecewise-linear function has many geometric lifts.

\begin{ex}
The rational function $f = \dfrac{3x^2 + xy^3 + 5}{2x^5y^3 + 6xz}$ has tropicalization
\[
\Trop(f) = \min(2x, x + 3y, 0) - \min(5x + 3y, x+z).
\]
\end{ex}

It is straightforward to verify (see, e.g.,~\cite[Lem.~2.1.6]{BFZ}) that tropicalization is independent of the choice of positive expression~\eqref{eq:positive expression} and that it satisfies
\begin{equation}
\label{eq:trop is a homom}
\Trop(f+g) = \min(\Trop(f),\Trop(g)), \qquad \Trop(fg^{\pm 1}) = \Trop(f) \pm \Trop(g).
\end{equation}
In other words, tropicalization is a homomorphism from the semi-field of positive rational functions over $\CC$ with operations $(+,\cdot,\div)$ to the semi-field of piecewise-linear functions over $\ZZ$ with operations $(\min,+,-)$.

The geometric crystal operators on $(X_m)^n$ (resp., $\GT_n^{\leq m}$) tropicalize to piecewise-linear formulas for the combinatorial crystal operators on tensor products of one-row tableaux (resp., Gelfand--Tsetlin patterns or, equivalently, semistandard tableaux). See \cite[\S 2.5 and \S 3.3]{BFPS} for a detailed explanation.

\subsection{Geometric RSK}
\label{sec_gRSK}

Before defining geometric RSK, we recall how the RSK correspondence can be viewed as a bijection from the set of all $m \times n$ nonnegative integer matrices to the set of $m \times n$ nonnegative integer matrices with weakly increasing rows and columns.

Let $\mb{a} = (a_i^j)_{i \in [m], j \in [n]}$ be an $m \times n$ matrix of nonnegative integers. We interpret the $i$th row $\mb{a}_i = (a_i^1, \ldots, a_i^n)$ of $\mb{a}$ as a weakly increasing word in the alphabet $[n]$ consisting of $a_i^1$ 1's, $a_i^2$ 2's, etc. The RSK correspondence sends $\mb{a}$ to a pair $(P,Q)$ of semistandard tableaux of the same shape such that $P$ has entries in $[n]$, and $Q$ has entries in $[m]$. Specifically, one takes $P_0$ to be the empty tableau and then recursively defines $P_i$ to be the tableau formed by row inserting the word $\mb{a}_i$ into $P_{i-1}$. The \defn{insertion tableau} $P$ is defined to be $P_m$, and the \defn{recording tableau} $Q$ is the tableau such that for each $i \in [m]$, the subtableau consisting of entries less than or equal to $i$ has the same shape as $P_i$. (We refer the reader to~\cite{EC2} for more details.)

To go from the pair $(P,Q)$ to an $m \times n$ matrix with weakly increasing rows and columns, one forms the GT patterns of width $\min(m,n)$ and heights $n$ and $m$ associated with $P$ and $Q$ (see \S\ref{sec:GT}) and glues the $P$-pattern to the transpose of the $Q$-pattern along the diagonal specifying their common shape. In this context, we represent GT patterns as left-justified arrays, with the $i$th row from the bottom containing the entries $a_{i,i}, \ldots, a_{i,n}$.

\ytableausetup{centertableaux}
\begin{ex}
Consider $n = 3$ and $m = 2$.
Under the RSK correspondence, we computer
\[
\mb{a} = \begin{pmatrix}
1 & 4 \\
2 & 1 \\
1 & 0
\end{pmatrix}
\;
\xrightarrow{\displaystyle \RSK}
\;
(P,Q) = \left( \, \ytableaushort{111122,222} \; , \; \ytableaushort{111112,223} \, \right).
\]
By transforming the pair of semistandard tableaux $(P, Q)$ into Gelfand--Tsetlin patterns of width $2$, we glue the patterns together along their common shape $(6,3)$ to form the matrix
\[
(P,Q) \longleftrightarrow \left(\textcolor{blue}{\begin{matrix}
3 & \\
4 & 6
\end{matrix}} \quad
, \quad
\textcolor{red}{\begin{matrix}
2 & 3 & \\
5 & 6 & 6
\end{matrix}}\right)
\xrightarrow{\displaystyle \text{glue}}
\begin{pmatrix}
\textcolor{red}{2} & \textcolor{red}{5} \\
\textcolor{purple}{3} & \textcolor{red}{6} \\
\textcolor{blue}{4} & \textcolor{purple}{6}
\end{pmatrix}
\]
with weakly increasing rows and columns.
\end{ex}

A fundamental property of RSK is that transposing the matrix $\mb{a}$ corresponds to interchanging the tableaux $P$ and $Q$. In other words, $Q$ is the tableau formed by inserting the words associated to the columns $\mb{a}^j = (a_1^j, a_2^j, \ldots, a_m^j)$, and $P$ records this insertion procedure. If we view RSK as a map of $m \times n$ matrices, this property says that RSK commutes with transposition.

The \defn{geometric RSK correspondence (gRSK)} is a birational isomorphism
\[
\gRSK \colon \Mat_{m \times n}(\CC^*) \rightarrow \Mat_{m \times n}(\CC^*).
\]
By analogy with classical RSK, we identify the output matrix with a pair of patterns $(P,Q) \in \GT_n^{\leq m} * \GT_m^{\leq n}$, where
\[
\GT_n^{\leq m} * \GT_m^{\leq n} = \{(P,Q) \in \GT_n^{\leq m} \times \GT_m^{\leq n} \mid \sh(P) = \sh(Q)\}.
\]
As in the combinatorial setting, our convention is that $P$ sits in the bottom-left corner of the matrix, and the transpose of $Q$ sits in the top-right corner.

\begin{dfn}
\label{defn:gRSK}
Suppose $\xx \in \Mat_{m \times n}(\CC^*)$. Let $\mb{x}_a = (x_a^1, \ldots, x_a^n)$ denote the $a$th row of $\xx$, and let $M(\mb{x}_1, \ldots, \mb{x}_k)$ be the $n \times n$ matrix defined in \S\ref{sec:basic}. Define $\gRSK(\xx) = (P,Q)$, where $P = (z_{i,j}) \in \GT_n^{\leq m}$ and $Q = (z'_{j',i'}) \in \GT_m^{\leq n}$ are given by
\[
z_{i,j} = \dfrac{\Delta_{[i,j]}\bigl( M(\mb{x}_1, \ldots, \mb{x}_m) \bigr)}{\Delta_{[i+1,j]}\bigl( M(\mb{x}_1, \ldots, \mb{x}_m) \bigr)},
\qquad\qquad z'_{j',i'} = \dfrac{\Delta_{[j',n]}\bigl( M(\mb{x}_1, \ldots, \mb{x}_{i'}) \bigr)}{\Delta_{[j'+1,n]}\bigl( M(\mb{x}_1, \ldots, \mb{x}_{i'}) \bigr)},
\]
for $1 \leq i \leq m, \, i \leq j \leq n$ and $1 \leq j' \leq n, \, j' \leq i' \leq m$.
\end{dfn}

It is clear that the shape $(z_{1,n}, \ldots, z_{p,n})$ of $P$ is equal to the shape $(z'_{1,m}, \ldots, z'_{p,m})$ of $Q$, so we do in fact have $(P,Q) \in \GT_n^{\leq m} * \GT_m^{\leq n} = \Mat_{m \times n}(\CC^*)$ as required. Moreover, if we define
\[
P_k = \Psi_n^{\leq k}(M(\mb{x}_1, \ldots, \mb{x}_k)) \in \GT_n^{\leq k},
\]
then it follows from Lemma~\ref{lem:Phi inverse} that $P = P_m$ and the $k$th diagonal $(z'_{1,k}, z'_{2,k}, \ldots, z'_{\min(k,n),k})$ of $Q$ is equal to the shape of the pattern $P_k$. We interpret this as saying that $P$ is the result of ``geometrically inserting'' the rows of $\xx$ (starting with the top row $\xx_1$) and $Q$ ``records the growth'' of $P$.

\begin{ex}
\label{ex:gRSK insertion}
When $m=3$ and $n=2$, the geometric RSK correspondence is given by
\[
\begin{pmatrix}
x_1^1 & x_1^2 \medskip \\
x_2^1 & x_2^2 \medskip \\
x_3^1 & x_3^2
\end{pmatrix}
\xrightarrow{\displaystyle \gRSK}
\begin{pmatrix}
\textcolor{red}{x_1^2 + x_2^1} & \textcolor{red}{x_1^1x_1^2} \bigskip \\
\textcolor{purple}{x_1^2 x_2^2 + x_1^2 x_3^1 + x_2^1x_3^1} & \textcolor{red}{\dfrac{x_1^1 x_2^1 x_1^2 x_2^2}{x_1^2 + x_2^1}} \bigskip \\
\textcolor{blue}{x_1^1 x_2^1 x_3^1} & \textcolor{purple}{\dfrac{x_1^1 x_2^1 x_3^1 x_1^2 x_2^2 x_3^2}{x_1^2 x_2^2 + x_1^2 x_3^1 + x_2^1x_3^1}}
\end{pmatrix}.
\]
\end{ex}

\begin{thm}[\cite{NoumiYamada,OSZ}]
\label{thm:gRSK props}
\
\begin{enumerate}
\item gRSK is a birational isomorphism.
\item gRSK and its inverse are positive. This implies that gRSK restricts to a homeomorphism from $\Mat_{m \times n}(\RR_{> 0})$ to itself.
\item gRSK satisfies the symmetry property
\[
\gRSK(\xx) = (P,Q) \iff \gRSK(\xx^t) = (Q,P),
\]
where $\xx^t$ is the transpose of $\xx$. Thus, the patterns $P = (z_{i,j})$ and $Q = (z'_{j',i'})$ are given by the alternative formulas
\begin{equation}
\label{eq:gRSK transposed}
z_{i,j} = \dfrac{\Delta_{[i,m]}(M(\mb{x}^1, \ldots, \mb{x}^j))}{\Delta_{[i+1,m]}(M(\mb{x}^1, \ldots, \mb{x}^j))},
\qquad\qquad z'_{j',i'} = \dfrac{\Delta_{[j',i']}(M(\mb{x}^1, \ldots, \mb{x}^n))}{\Delta_{[j'+1,i']}(M(\mb{x}^1, \ldots, \mb{x}^n))},
\end{equation}
for $1 \leq i \leq m, \, i \leq j \leq n$ and $1 \leq j' \leq n, \, j' \leq i' \leq m$, where $\mb{x}^j = (x_1^j, \ldots, x_m^j)$ is the $j$th column of~$\xx$.
\end{enumerate}
\end{thm}

\begin{remark}
\label{rem:gRSK}
The definition of gRSK given above is essentially due to Noumi and Yamada~\cite{NoumiYamada}, although our version differs from theirs by the transformation $f(x) \mapsto 1/f(x^{-1})$. The result is that our map tropicalizes to RSK under the min-plus tropicalization defined in~\S\ref{sec:trop}, whereas theirs tropicalizes to RSK under max-plus tropicalization (see~\cite[\S 4.4]{BFPS} for more details).

O'Connell, Sepp\"al\"ainen, and Zygouras~\cite{OSZ} gave an alternative definition of gRSK in terms of local transformations. This definition is closely related to the growth diagram formulation of RSK and makes all the assertions of Theorem~\ref{thm:gRSK props} completely transparent. The min-plus version of this alternative definition is presented in~\cite[\S 4.2]{BFPS}.
\end{remark}

In \S\ref{sec:basic}, we defined a $\GL_n \times \GL_m$-geometric crystal structure on $\Mat_{m \times n}(\CC^*)$. Using the identification of $\Mat_{m \times n}(\CC^*)$ with $\GT_n^{\leq m} * \GT_m^{\leq n}$, we obtain another $\GL_n \times \GL_m$-geometric crystal structure on $\Mat_{m \times n}(\CC^*)$ in which the $\GL_n$-geometric crystal operators act on the $P$-pattern, and the $\GL_m$-geometric crystal operators act on the $Q$-pattern:
\[
\ov{e}_j^c(P,Q) = (\ov{e}_j^c(P), Q), \qquad e_i^c(P,Q) = (P, e_i^c(Q)).
\]
This is well-defined because the geometric crystal operators are shape-preserving. We denote the former geometric crystal on $\Mat_{m \times n}(\CC^*)$ by $X^{\Mat}$ and the latter by $X^{\GT}$.

\begin{thm}[{\cite[\ThmGRSKIsom]{BFPS}}]
\label{thm:gRSK isom}
The map $\gRSK \colon X^{\Mat} \rightarrow X^{\GT}$ is an isomorphism of $\GL_n \times \GL_m$-geometric crystals.
\end{thm}


\section{The fields \texorpdfstring{$\Inv_R, \Inv_{\ov{R}}, \Inv_e, \Inv_{\ov{e}}$, and $\Inv_{e \ov{e}}$}{of R, Rbar, e, ebar, and e-ebar invariants}}
\label{sec:single}

For the remainder of the paper, we work in the field $\CC(x_i^j)$ of rational functions in the $mn$ indeterminates $\{x_i^j \mid i \in [m], j \in [n]\}$, which is the fraction field of the variety $\Mat_{m \times n}(\CC^*)$. The geometric crystal operators $e_i$ and $\ov{e}_j$ on $X^{\Mat} = \Mat_{m \times n}(\CC^*)$ induce $\CC^*$-actions on $\CC(x_i^j)$, and the geometric $R$-matrices $R_i$ and $\ov{R}_j$ induce automorphisms of $\CC(x_i^j)$. Slightly abusing notation, we denote the induced maps on $\CC(x_i^j)$ by the same symbols as the maps on $X^{\Mat}$.

For $a \in \{R, \ov{R}, e, \ov{e}\}$, the \defn{$a$-invariants} is the subfield $\Inv_a$ of $\CC(x_i^j)$ consisting of fixed points of $\{a_1, a_2, \ldots\}$. For example, $\Inv_R$ denotes the $R$-invariants. We also consider intersections of these fields, which we denote by $\Inv_{ab} = \Inv_a \cap \Inv_b$ for $a,b \in \{R,\ov{R},e,\ov{e}\}$ and call \defn{$ab$-invariants}. Since $R_i$ (resp., $\ov{R}_j$) is a special case of $e_i$ (resp., $\ov{e}_j$) by Proposition~\ref{prop:R=Weyl}, the distinct intersections are $\Inv_{R\ov{R}}, \Inv_{R \ov{e}}, \Inv_{\ov{R} e}, \Inv_{e \ov{e}}$.

In this section, we describe algebraically independent generating sets for each of the fields $\Inv_R$, $\Inv_{\ov{R}}$, $\Inv_e$, $\Inv_{\ov{e}}$, and $\Inv_{e \ov{e}}$; we consider the remaining three fields in~\S\ref{sec:double}.

\subsection{Loop symmetric functions}
\label{sec:LSym}

The geometric $R$-matrices $R_1, \ldots, R_{m-1}$ generate a group of automorphisms of the field $\CC(x_i^j)$ isomorphic to $S_m$. The fixed field $\Inv_R$ of this group of automorphisms has been studied in several papers, starting with~\cite{LP12}, under the name of loop symmetric functions. It is more convenient to describe loop symmetric functions in terms of variables $\lx{i}{r} = x_i^{r-i+1}$, where the superscript $r$ is considered modulo $n$, or equivalently $x_i^j = \lx{i}{j+i-1}$. In other words, we make the identification
\[
\mb{x}_i = (x_i^1, x_i^2, \ldots, x_i^n) = (\lx{i}{i}, \lx{i}{i+1}, \ldots \lx{i}{i+n-1}).
\]

\begin{dfn}
For $k \in [m]$ and $r \in [n]$, define the \defn{loop elementary symmetric function}
\[
\lE{k}{r} = \lE{k}{r}(\mb{x}_1, \ldots, \mb{x}_m) = \sum_{1 \leq i_1 < i_2 < \cdots < i_k \leq m} \lx{i_1}{r} \lx{i_2}{r+1} \cdots \lx{i_k}{r+k-1}.
\]
It is convenient to allow arbitrary integers $r$ in the superscript by setting $\lE{k}{r} = \lE{k}{r \mod n}$, and to define $\lE{0}{r} = 1$ and $\lE{k}{r} = 0$ if $k < 0$ or $k > m$.

Define $\LSym = \LSym_m(n)$ to be the ring generated over $\CC$ by the $\lE{k}{r}$. We call this the \defn{ring of loop symmetric functions in $m$ variables and $n$ colors}.
\end{dfn}

\begin{thm}
\label{thm:LSym}
\
\begin{enumerate}
\item \cite[Thm.~4.1]{LPaff} The loop elementary symmetric functions are algebraically independent, and they generate the field $\Inv_R$. That is, we have
\[
\Inv_R = \Frac(\LSym) = \CC(\lE{k}{r}).
\]
\item \cite[\ThmFTLSF]{BFPS} The ring $\CC[x_i^j] \cap \Inv_R$ of polynomial $R$-invariants is equal to $\LSym$.
\end{enumerate}
\end{thm}

\begin{ex}
\
\begin{enumerate}
\item[(a)] When $n = 1$, $\lE{i}{1}$ is the elementary symmetric polynomial in the $m$ variables $\lx{1}{1}, \ldots, \lx{m}{1}$, and $R_i$ simply swaps the variables $\lx{i}{1}$ and $\lx{i+1}{1}$. Theorem~\ref{thm:LSym}(2) reduces in this case to Newton's fundamental theorem of symmetric functions, which says that the ring of symmetric polynomials in $m$ variables is the polynomial ring in the elementary symmetric polynomials $e_1, \ldots, e_m$.

\item[(b)] When $m = 1$, there are no maps $R_i$, and $\LSym$ is the polynomial ring in the $n$ variables $\lx{1}{r} = \lE{1}{r}$, $r \in [n]$, so Theorem~\ref{thm:LSym} is trivial in this case.
\end{enumerate}
\end{ex}

We now recall some notions from tableau combinatorics. A \defn{partition} is a finite, weakly decreasing sequence $\la = (\la_1, \ldots, \la_d)$ of nonnegative integers. The \defn{length} of $\la$, denoted $\ell(\la)$, is the number of parts $\la_i$ which are nonzero. We identify partitions which differ only by trailing zeroes. We also identity a partition $\la$ with its \defn{Young diagram}, which we view, following the English convention, as a left-justified collection of unit cells in the plane, with $\la_1$ cells in the top row, $\la_2$ in the next row, etc. Let $\lambda'$ denote the \defn{conjugate} or \defn{transpose} of $\la$, the partition obtained by reflecting the Young diagram of $\la$ over the line $y=-x$.

For partitions $\mu$ and $\la$, we write $\mu \subseteq \la$ if the Young diagram of $\mu$ is contained in that of $\la$. For $\mu \subseteq \la$, we define the \defn{skew shape} $\la/\mu$ to be the collection of cells in $\la$ which are not in $\mu$. For a cell $s = (i, j)$ in the $i$th row and $j$th column (using matrix coordinates), the \defn{content} of $s$ is $c(s) = i - j$ (this is the opposite of the usual definition). A \defn{semistandard Young tableau (SSYT) of shape $\la/\mu$} is a filling $T \colon \lambda/\mu \to \ZZ_{>0}$ of the cells of $\la/\mu$ with positive integers such that rows are weakly increasing from left to right and columns are strictly increasing from top to bottom. Let $\SSYT_{\leq m}(\la/\mu)$ denote the set of semistandard Young tableaux of shape $\la/\mu$ with entries at most $m$.

\begin{dfn}
\label{defn:loop Schur}
For partitions $\mu \subseteq \la$ and a color $r \in \ZZ/n\ZZ$, define the \defn{loop skew Schur function}
\[
\ls{\lambda/\mu}{r} = \ls{\la/\mu}{r}(\mb{x}_1, \ldots, \mb{x}_m) = \sum_{T \in \SSYT_{\leq m}(\la)} \xx^T,
\]
where $\xx^T = \prod_{s \in \lambda/\mu} \lx{T(s)}{c(s)+r}$. (By convention, $s^{(r)}_{\la/\mu} = 1$ if $\la = \mu$.) If $\mu = \emptyset$, then $\ls{\la/\emptyset}{r} = \ls{\la}{r}$ is a \defn{loop Schur function}. We refer to $c(s)+r$ (considered modulo $n$) as the \defn{color} of the cell $s$.
\end{dfn}

It is clear that $\ls{(1^k)}{r} = \lE{k}{r}$. Define the \defn{loop homogeneous symmetric function}
\[
\lH{k}{r} = \ls{(k)}{r} = \sum_{1 \leq i_1 \leq \cdots \leq i_k \leq m} \lx{i_1}{r} \lx{i_2}{r-1} \cdots \lx{i_k}{r-k+1}.
\]

The following analogue of the Jacobi--Trudi formula shows that the loop skew Schur functions are in $\LSym$ and plays a crucial role throughout this paper.

\begin{prop}[{\cite[Thm.~7.6]{LP12}}]
\label{prop:JT}
Suppose $\mu \subseteq \la$ and the conjugate partition $\la'$ has length at most $\ell$. Then
\[
\ls{\lambda/\mu}{r} = \det\left( \lE{\lambda'_i - \mu'_j + j - i}{r + \mu'_j - j + 1} \right)_{i,j=1}^{\ell}.
\]
\end{prop}

\begin{ex}
\label{ex:loop Schur}
We compute the loop Schur function $\ls{(4,2)}{1}$ in the case $m=2, n=4$. The colors $c(s) + 1$ of the cells of $(4,2)$ are shown below, where Orange is $1$, Blue is $2$, Red is $3$, and Green is $4$:
\[
\begin{tikzpicture}[baseline=-40,scale=0.5]
\foreach \x/\y in {3/1} {
  \fill[green!40] (\x,-\y) rectangle (\x+1,-\y-1);
  \draw (\x+0.5,-\y-0.5) node {$G$};
}
\foreach \x/\y in {2/1,3/2} {
  \fill[orange!40] (\x,-\y) rectangle (\x+1,-\y-1);
  \draw (\x+0.5,-\y-0.5) node {$O$};
}
\foreach \x/\y in {4/1} {
  \fill[red!40] (\x,-\y) rectangle (\x+1,-\y-1);
  \draw (\x+0.5,-\y-0.5) node {$R$};
}
\foreach \x/\y in {2/2,5/1} {
  \fill[blue!40] (\x,-\y) rectangle (\x+1,-\y-1);
  \draw (\x+0.5,-\y-0.5) node {$B$};
}
\draw[thick] (2,-1) -- (6,-1) -- (6,-2) -- (4,-2) -- (4,-3) -- (2,-3) -- (2,-1);
\draw (3,-1) -- (3,-3);
\draw (4,-1) -- (4,-2);
\draw (5,-1) -- (5,-2);
\draw (2,-2) -- (6,-2);
\end{tikzpicture}\ .
\]
There are three semistandard tableaux of shape $\lambda$ with entries in $\{1,2\}$:
\[
\begin{tikzpicture}[baseline=-40,scale=0.5]
\foreach \x/\y in {3/1} {
  \fill[green!40] (\x,-\y) rectangle (\x+1,-\y-1);
  \draw (\x+0.5,-\y-0.5) node {$1$};
}
\foreach \x/\y in {2/1} {
  \fill[orange!40] (\x,-\y) rectangle (\x+1,-\y-1);
  \draw (\x+0.5,-\y-0.5) node {$1$};
}
\foreach \x/\y in {3/2} {
  \fill[orange!40] (\x,-\y) rectangle (\x+1,-\y-1);
  \draw (\x+0.5,-\y-0.5) node {$2$};
}
\foreach \x/\y in {4/1} {
  \fill[red!40] (\x,-\y) rectangle (\x+1,-\y-1);
  \draw (\x+0.5,-\y-0.5) node {$$1};
}
\foreach \x/\y in {2/2} {
  \fill[blue!40] (\x,-\y) rectangle (\x+1,-\y-1);
  \draw (\x+0.5,-\y-0.5) node {$2$};
}
\foreach \x/\y in {5/1} {
  \fill[blue!40] (\x,-\y) rectangle (\x+1,-\y-1);
  \draw (\x+0.5,-\y-0.5) node {$1$};
}
\draw[thick] (2,-1) -- (6,-1) -- (6,-2) -- (4,-2) -- (4,-3) -- (2,-3) -- (2,-1);
\draw (3,-1) -- (3,-3);
\draw (4,-1) -- (4,-2);
\draw (5,-1) -- (5,-2);
\draw (2,-2) -- (6,-2);

\begin{scope}[xshift=8cm]
\foreach \x/\y in {3/1} {
  \fill[green!40] (\x,-\y) rectangle (\x+1,-\y-1);
  \draw (\x+0.5,-\y-0.5) node {$1$};
}
\foreach \x/\y in {2/1} {
  \fill[orange!40] (\x,-\y) rectangle (\x+1,-\y-1);
  \draw (\x+0.5,-\y-0.5) node {$1$};
}
\foreach \x/\y in {3/2} {
  \fill[orange!40] (\x,-\y) rectangle (\x+1,-\y-1);
  \draw (\x+0.5,-\y-0.5) node {$2$};
}
\foreach \x/\y in {4/1} {
  \fill[red!40] (\x,-\y) rectangle (\x+1,-\y-1);
  \draw (\x+0.5,-\y-0.5) node {$1$};
}
\foreach \x/\y in {2/2} {
  \fill[blue!40] (\x,-\y) rectangle (\x+1,-\y-1);
  \draw (\x+0.5,-\y-0.5) node {$2$};
}
\foreach \x/\y in {5/1} {
  \fill[blue!40] (\x,-\y) rectangle (\x+1,-\y-1);
  \draw (\x+0.5,-\y-0.5) node {$2$};
}
\draw[thick] (2,-1) -- (6,-1) -- (6,-2) -- (4,-2) -- (4,-3) -- (2,-3) -- (2,-1);
\draw (3,-1) -- (3,-3);
\draw (4,-1) -- (4,-2);
\draw (5,-1) -- (5,-2);
\draw (2,-2) -- (6,-2);
\end{scope}

\begin{scope}[xshift=16cm]
\foreach \x/\y in {3/1} {
  \fill[green!40] (\x,-\y) rectangle (\x+1,-\y-1);
  \draw (\x+0.5,-\y-0.5) node {$1$};
}
\foreach \x/\y in {2/1} {
  \fill[orange!40] (\x,-\y) rectangle (\x+1,-\y-1);
  \draw (\x+0.5,-\y-0.5) node {$1$};
}
\foreach \x/\y in {3/2} {
  \fill[orange!40] (\x,-\y) rectangle (\x+1,-\y-1);
  \draw (\x+0.5,-\y-0.5) node {$2$};
}
\foreach \x/\y in {4/1} {
  \fill[red!40] (\x,-\y) rectangle (\x+1,-\y-1);
  \draw (\x+0.5,-\y-0.5) node {$2$};
}
\foreach \x/\y in {2/2} {
  \fill[blue!40] (\x,-\y) rectangle (\x+1,-\y-1);
  \draw (\x+0.5,-\y-0.5) node {$2$};
}
\foreach \x/\y in {5/1} {
  \fill[blue!40] (\x,-\y) rectangle (\x+1,-\y-1);
  \draw (\x+0.5,-\y-0.5) node {$2$};
}
\draw[thick] (2,-1) -- (6,-1) -- (6,-2) -- (4,-2) -- (4,-3) -- (2,-3) -- (2,-1);
\draw (3,-1) -- (3,-3);
\draw (4,-1) -- (4,-2);
\draw (5,-1) -- (5,-2);
\draw (2,-2) -- (6,-2);
\end{scope}

\end{tikzpicture}\ .
\]
Thus, we have
\[
\ls{(4,2)}{1} = \lx{1}{1}\lx{1}{2}\lx{1}{3}\lx{1}{4}\lx{2}{1}\lx{2}{2} + \lx{1}{1}\lx{1}{3}\lx{1}{4}\lx{2}{1}(\lx{2}{2})^2 + \lx{1}{1}\lx{1}{4}\lx{2}{1}(\lx{2}{2})^2\lx{2}{3}.
\]
The reader may verify that
\[
\ls{(4,2)}{1} = \det \begin{pmatrix}
\lE{2}{1} & 0 & 0 & 0 \\
\lE{1}{1} & \lE{2}{4} & 0 & 0\\
0 & 1 & \lE{1}{3} & \lE{2}{2} \\
0 & 0 & 1 & \lE{1}{2}
\end{pmatrix}.
\]
\end{ex}

It follows from Theorem~\ref{thm:LSym} that the ring $\CC[x_i^j] \cap \Inv_{\ov{R}}$ of polynomial $\ov{R}$-invariants is isomorphic to $\LSym_n(m)$---the ring of loop symmetric functions in $n$ variables and $m$ colors---and the field $\Inv_{\ov{R}}$ is isomorphic to $\Frac(\LSym_n(m))$. To distinguish $\ov{R}$-invariants from $R$-invariants, we express elements of $\LSym_n(m)$ in terms of variables $\bx{i}{j}$ for $j \in [n]$ and $i \in \Zmod{m}$; that is, we put the color in the subscript rather than the superscript. These variables are identified with the $x_i^j$ by
\[
\mb{x}^j = (x_1^j, \ldots, x_m^j) = (\bx{j}{j}, \ldots, \bx{j+m-1}{j}).
\]
In this context, we call $\LSym_n(m)$ the ring of \defn{barred loop symmetric functions}, and we denote it by $\ov{\LSym}$. This ring is generated by the \defn{barred loop elementary symmetric functions} $\bE{k}{r}, k \in [n], r \in [m]$, and has distinguished elements $\bH{k}{r}, \bs{\lambda/\mu}{r}$, etc.

\begin{ex}
When $m=3$ and $n=2$, we have
\[
\lE{2}{2}(\mb{x}_1, \mb{x}_2, \mb{x}_3) = \lx{1}{2}\lx{2}{1} + \lx{1}{2}\lx{3}{1} + \lx{2}{2}\lx{3}{1},
\]
\[
\bH{2}{3}(\mb{x}^1,\mb{x}^2) = \bx{3}{1}\bx{2}{1} + \bx{3}{1}\bx{2}{2} + \bx{3}{2}\bx{2}{2}.
\]
Using the identifications
\[
\begin{pmatrix}
\lx{1}{1} & \lx{1}{2} \medskip \\
\lx{2}{2} & \lx{2}{1} \medskip \\
\lx{3}{1} & \lx{3}{2}
\end{pmatrix}
=
\begin{pmatrix}
x_1^1 & x_1^2 \medskip \\
x_2^1 & x_2^2 \medskip \\
x_3^1 & x_3^2
\end{pmatrix}
=
\begin{pmatrix}
\bx{1}{1} & \bx{2}{2} \medskip \\
\bx{2}{1} & \bx{3}{2} \medskip \\
\bx{3}{1} & \bx{1}{2}
\end{pmatrix},
\]
we see that both of these polynomials are equal to $x_1^2x_2^2 + x_1^2x_3^1 + x_2^1x_3^1$.
\end{ex}

\subsection{Matrix interpretation of \texorpdfstring{$\LSym$}{LSym}}
\label{sec:LSym matrix}

We now explain how loop symmetric functions arise from a certain infinite matrix. An \defn{$n$-periodic matrix} is a $\mathbb{Z} \times \mathbb{Z}$ array $(A_{ij})_{i,j \in \mathbb{Z}}$ such that $A_{ij} = A_{i+n,j+n}$ for all $i,j$, and $A_{ij} = 0$ if $j-i$ is sufficiently large. The second hypothesis ensures that multiplication of these matrices is well-defined. Given an $n$-tuple $x = (x^1, \ldots, x^n)$, let $\widetilde{W}(x)$ be the $n$-periodic matrix with $\widetilde{W}(x)_{i+1,i} = 1$, $\widetilde{W}(x)_{i,i} = x^i$ (the superscript of $x^i$ is interpreted mod $n$), and all other entries zero. This matrix is called a \defn{whirl} in~\cite{LP12}.
For example, when $n=3$,
\begin{equation*}
\widetilde{W}(x^{(1)}, x^{(2)}, x^{(3)}) =
\begin{pmatrix}
x^{(1)} & 0 & 0 & 0 \\
1 & x^{(2)} & 0 & 0  \\
0 & 1 & x^{(3)} & 0 & \hdots \\
0 & 0 & 1 & x^{(1)}  \\
0 & 0 & 0 & 1 \\
&\vdots &&& \ddots
\end{pmatrix}.
\end{equation*}
Note that we are depicting only the quadrant of the matrix with coordinates $i,j \geq 1$.

As in the previous section, let $\mb{x}_i = (\lx{i}{i}, \lx{i}{i+1}, \ldots, \lx{i}{i+n-1})$ for $i = 1, \ldots, m$. Define $\widetilde{M}(\mb{x}_1, \ldots, \mb{x}_m) = \widetilde{W}(\mb{x}_1) \cdots \widetilde{W}(\mb{x}_m)$. It is straightforward to prove by induction that the entries of this matrix are loop elementary symmetric functions:
\begin{equation}
\label{eq:tilde M entries}
\widetilde{M}(\mb{x}_1, \ldots, \mb{x}_m)_{ij} = \lE{m+j-i}{i}(\mb{x}_1, \ldots, \mb{x}_m).
\end{equation}
An example of the matrix $\widetilde{M}(\mb{x}_1, \ldots, \mb{x}_m)$ appears in Figure~\ref{fig:M unfolded}.

\begin{figure}
\[
\left( \begin{array}{c@{\;}c|c@{\;}c|c@{\;}c|c}
\lx{1}{1}\lx{2}{2}\lx{3}{1}  = \lE{3}{1} & 0 & 0 & 0 & 0 & 0 & \hdots \\
\lx{1}{2}\lx{2}{1} + \lx{1}{2}\lx{3}{1} + \lx{2}{2}\lx{3}{1} = \lE{2}{2} & \lx{1}{2}\lx{2}{1}\lx{3}{2} = \lE{3}{2} & 0 & 0 & 0 & 0 \\ \hline
\lx{1}{1} + \lx{2}{1} + \lx{3}{1} = \lE{1}{1} & \lx{1}{1}\lx{2}{2} + \lx{1}{1}\lx{3}{2} + \lx{2}{1}\lx{3}{2} = \lE{2}{1} & \lE{3}{1} & 0 & 0 & 0 \\
1 & \lx{1}{2} + \lx{2}{2} + \lx{3}{2} = \lE{1}{2} & \lE{2}{2} & \lE{3}{2} & 0 & 0 \\ \hline
0 & 1 & \lE{1}{1} & \lE{2}{1} & \lE{3}{1} & 0 \\
0 & 0 & 1 & \lE{1}{2} & \lE{2}{2} & \lE{3}{2} \\ \hline
0 & 0 & 0 & 1 & \lE{1}{1} & \lE{2}{1} \\
0 & 0 & 0 & 0 & 1 & \lE{1}{2} \\ \hline
\vdots &  &  & & & & \ddots
\end{array} \right)
\]
\caption{The $2$-periodic matrix $\widetilde{M}(\mb{x}_1, \mb{x}_2, \mb{x}_3)$, with $2 \times 2$ blocks indicated. We depict only the quadrant of the matrix with coordinates $i,j \geq 1$.}
\label{fig:M unfolded}
\end{figure}

The Jacobi--Trudi formula (Proposition~\ref{prop:JT}) implies that loop skew Schur functions are precisely the minors of the unfolded matrix $\widetilde{M}(\mb{x}_1, \ldots, \mb{x}_m)$. For later use, we identify a specific submatrix of $\widetilde{M}(\mb{x}_1, \ldots, \mb{x}_m)$ whose determinant is $s_{\la/\mu}^{(r)}(\mb{x}_1, \ldots, \mb{x}_m)$.

\begin{prop}
\label{prop:JT matrix version}
The loop skew Schur function $s_{\la/\mu}^{(r)}(\mb{x}_1, \ldots, \mb{x}_m)$ is given by
\begin{equation}
\label{eq:JT matrix version}
s_{\la/\mu}^{(r)}(\mb{x}_1, \ldots, \mb{x}_m) = \Delta_{I(\mu,r),J(\la,r)}(\widetilde{M}(\mb{x}_1, \ldots, \mb{x}_m)),
\end{equation}
where
\begin{equation}
\label{eq:subsets JT matrix version}
\begin{array}{l}
I(\mu,r) = \{\mu'_\ell-\ell+1+r, \mu'_{\ell-1}-\ell+2 + r, \ldots, \mu'_1+r\}, \medskip \\
J(\la,r) = \{\la'_\ell-\ell+1+r-m, \la'_{\ell-1}-\ell+2 + r-m, \ldots, \la'_1+r-m\},
\end{array}
\end{equation}
and $\ell$ is the length of $\la'$.
\end{prop}

Note that if we replace $r$ by $r + dn$ ($d \in \ZZ$), the minor on the right-hand side of~\eqref{eq:JT matrix version} does not change due to the $n$-periodicity of $\widetilde{M}$. (The sets $I(\mu, r)$ and $J(\la, r)$ are shifted versions of the Maya diagrams of the partitions $\mu'$ and $\la'$.)

\begin{proof}
Let $A$ denote the $\ell \times \ell$ submatrix of $\widetilde{M} = \widetilde{M}(\mb{x}_1, \ldots, \mb{x}_m)$ using the rows in $I(\mu,r)$ and the columns in $J(\la,r)$. By~\eqref{eq:tilde M entries}, the $(\ell-j+1,\ell-i+1)$-entry of this submatrix is
\[
A_{\ell-j+1,\ell-i+1} = \widetilde{M}_{r+\mu'_j-j+1, r-m+\la'_i-i+1} = \lE{\la'_{i}-\mu'_{j}+j-i}{r+\mu'_{j}-j+1}.
\]
We see that $A$ is related to the matrix appearing in the Jacobi--Trudi formula by reflecting over the anti-diagonal, so it has the same determinant.
\end{proof}

There is a useful correspondence between $n$-periodic matrices and $n \times n$ matrices with entries in the field of formal Laurent series in a parameter $t$. Given an $n$-periodic matrix $A$, define an $n \times n$ matrix $\wh{A}(t)$ by
\[
\wh{A}(t)_{ij} = \sum_{d \in \ZZ} t^d A_{i+dn,j}.
\]
The map $A \mapsto \wh{A}(t)$ is an isomorphism of rings. We say that $\wh{A}(t)$ is the \defn{folded version} of $A$ and $A$ is the \defn{unfolded version} of $\wh{A}(t)$. We will denote the folded version of $\widetilde{M}(\mb{x}_1, \ldots, \mb{x}_m)$ by $\wh{M}(t)$. By~\eqref{eq:tilde M entries}, the entries of this matrix are given by
\begin{equation}
\label{eq:hat M entries}
\wh{M}(t)_{ij} = \sum_{d \geq 0} t^d \lE{m+j-i-nd}{i}(\mb{x}_1, \ldots, \mb{x}_m).
\end{equation}
For example, the folded version of the matrix $\widetilde{M}(\mb{x}_1, \mb{x}_2, \mb{x}_3)$ shown in Figure~\ref{fig:M unfolded} is
\[
\wh{M}(t) = \begin{pmatrix}
\lE{3}{1} + t\lE{1}{1} & t\lE{2}{1} + t^2 \medskip \\
\lE{2}{2} + t & \lE{3}{2} + t \lE{1}{2}
\end{pmatrix}.
\]
For each integer $d$, let $M_d$ be the $n \times n$ matrix such that $(M_d)_{ij}$ is the coefficient of $t^d$ in $\wh{M}(t)_{ij}$. As Figure~\ref{fig:M unfolded} illustrates, the unfolded matrix $\widetilde{M}(\mb{x}_1, \ldots, \mb{x}_m)$ consists of the blocks $M_d$, repeated along (block) diagonals.

Minors of the folded matrix $\wh{M}(t)$ also provide an interesting class of loop symmetric functions; these are studied in \S\ref{sec:cylindric}.

\subsection{The \texorpdfstring{$P$}{P}-pattern and \texorpdfstring{$e$}{e}-invariants}
\label{sec:P}

Recall the $n \times n$ matrix $M = M(\mb{x}_1, \ldots, \mb{x}_m) = W(\mb{x}_1) \cdots W(\mb{x}_m)$ introduced in \S\ref{sec:basic}. Each factor $W(\mb{x}_i)$ is obtained from the folded whirl $\wh{W(\mb{x}_i)}(t)$ by setting $t = 0$, so we have $M = \wh{M}(0)$, and it follows from~\eqref{eq:hat M entries} that
\begin{equation}
\label{eq:M entries}
M_{ij} = \lE{m+j-i}{i}(\mb{x}_1, \ldots, \mb{x}_m).
\end{equation}

We say that a loop elementary symmetric function $\lE{a}{b}$ is \defn{$P$-type} if $a + b \geq m+1$ and \defn{$Q$-type} otherwise. It follows from~\eqref{eq:M entries} that the $P$-type loop elementary symmetric functions (``$P$-type $E$-s'') are precisely the non-trivial entries of $M$. In the unfolded matrix $\widetilde{M}(\mb{x}_1, \ldots, \mb{x}_m)$, the $P$-type $E$-s appear in the top nonzero $n \times n$ block $M_0$, and the $Q$-type $E$-s appear in the lower $n \times n$ blocks $M_d, d > 0$ (see Figure~\ref{fig:M unfolded}). This terminology is explained by the following simple but important result.

\begin{lemma}
\label{lem:P}
The $P$-type loop elementary symmetric functions generate the same subfield of $\CC(x_i^j)$ as the entries $z_{i,j}$ of the $\gRSK$ $P$-pattern. The elements of this subfield are $e_i$-invariant for each $i \in [m-1]$.
\end{lemma}

\begin{proof}
Each $z_{i,j}$ is by definition a ratio of minors of the matrix $M(\mb{x}_1, \ldots, \mb{x}_m)$. On the other hand, it follows from Lemma~\ref{lem:Phi inverse} that the entries of this matrix are rational functions (in fact, Laurent polynomials) in the $z_{i,j}$. This proves the first assertion. The second assertion follows from Theorem~\ref{thm:gRSK isom}, which shows that the $\GL_m$-geometric crystal operators $e_i$ preserve the $P$-pattern.
\end{proof}

The converse of Lemma~\ref{lem:P} is the main result of the first paper in this series.

\begin{thm}[{\cite[\ThmInvEX]{BFPS}}]
\label{thm:e}
The $P$-type loop elementary symmetric functions (or, equivalently, the entries $z_{i,j}$ of the $\gRSK$ $P$-pattern) generate the subfield $\Inv_e \subset \CC(x_i^j)$. Moreover, the $P$-type loop elementary symmetric functions generate the ring $\CC[x_i^j] \cap \Inv_e$ of polynomial $e$-invariants.
\end{thm}

\begin{remark}
The proof of the assertion about the field of $e$-invariants is based on the geometric RSK correspondence and uses no results about loop symmetric functions or $R$-invariants. The assertion about the ring of polynomial $e$-invariants, on the other hand, relies on the fundamental theorem of loop symmetric functions (Theorem~\ref{thm:LSym}(2)).
\end{remark}

We say that a barred loop elementary symmetric function $\bE{a}{b}$ is \defn{$Q$-type} if $a + b \geq n+1$ and \defn{$P$-type} otherwise. From the symmetry of $\gRSK$, we obtain the following corollary of Lemma~\ref{lem:P} and Theorem~\ref{thm:e}.

\begin{cor}
\label{cor:ov e}
The field $\Inv_{\ov{e}} \subset \CC(x_i^j)$ is generated by the $Q$-type barred loop elementary symmetric functions or, equivalently, by the entries $z'_{j,i}$ of the $\gRSK$ $Q$-pattern.
\end{cor}

\subsection{Shape invariants}
\label{sec:shape}

We observed in \S\ref{sec:LSym matrix} that the minors of the unfolded matrix $\widetilde{M}$ are loop skew Schur functions. In particular, the entries of the $\gRSK$ $P$-pattern are ratios of loop Schur functions of rectangular shape. For $1 \leq i \leq m$ and $i \leq j \leq n$, let
\[
\Box(i, j) = \ls{(\underbrace{j-i+1, \ldots, j-i+1}_{m-i+1 \text{ times}})}{j}(\mb{x}_1, \ldots, \mb{x}_m)
\]
be the loop Schur function associated with an $(m-i+1) \times (j-i+1)$ rectangle, where the unique northwest corner has color $j$. Set $\Box(i+1,j) = 1$ if $i = j$ or $i = m$. It follows from~\eqref{eq:M entries} and Proposition~\ref{prop:JT matrix version} that
\begin{equation}
\label{eq:P loop formula}
z_{i,j} = \dfrac{\Delta_{[i,j],[1,j-i+1]}(M)}{\Delta_{[i+1,j],[1,j-i]}(M)} = \frac{\Box(i, j)}{\Box(i+1, j)}.
\end{equation}
It is clear from~\eqref{eq:P loop formula} that the $\Box(i,j)$ generate the same field as the $z_{i,j}$, and hence the same field as the $P$-type loop elementary symmetric functions.

For $i = 1, \ldots, p = \min(m,n)$, define the \defn{shape invariant} $S_i = \Box(i,n)$. It is convenient to set $S_{p+1} = \Box(p+1,n) = 1$. Note that $S_i$ is the determinant of the bottom-left justified minor of $M$ using the last $n-i+1$ rows and the first $n-i+1$ columns. We see from~\eqref{eq:P loop formula} that the shape of the $P$- and $Q$-patterns is given by
\begin{equation}
\label{eq:z shape}
(z_{1,n}, \ldots, z_{p-1,n}, z_{p,n}) = \left(\frac{S_1}{S_2}, \ldots, \frac{S_{p-1}}{S_p}, S_p\right).
\end{equation}
This shows that the shape invariants generate the same subfield of $\CC(x_i^j)$ as the $z_{i,n}$ and are therefore both $e_i$ and $\ov{e}_j$-invariant.

\begin{thm}[{\cite[\CorInvEEX]{BFPS}}]
\label{thm:shape}
The field $\Inv_{e\ov{e}} = \Inv_e \cap \Inv_{\ov{e}}$ is generated by the shape invariants.
\end{thm}

We conjecture that the shape invariants also generate the ring $\CC[x_i^j] \cap \Inv_{e\ov{e}}$.


\section{The fields \texorpdfstring{$\Inv_{R \ov{e}}$, $\Inv_{\ov{R} e}$, and $\Inv_{R \ov{R}}$}{of R-ebar, Rbar-e, and R-Rbar invariants}}
\label{sec:double}

We now turn our attention to the fields $\Inv_{R \ov{e}}$, $\Inv_{\ov{R} e}$, and $\Inv_{R \ov{R}}$. The field $\Inv_{R \ov{e}}$ can be viewed in two different ways: as the $\ov{e}$-invariant subfield of $\Inv_R$, and as the $R$-invariant subfield of $\Inv_{\ov{e}}$. In the latter perspective, $\Inv_{R \ov{e}}$ is the fixed field of a group of automorphisms of $\Inv_{\ov{e}}$ isomorphic to $S_m$, so by basic Galois theory, $\Inv_{\ov{e}}$ is a finite extension of $\Inv_{R \ov{e}}$ of degree $m!$. This means that the transcendence degree of $\Inv_{R \ov{e}}$ over $\CC$ is the same as that of $\Inv_{\ov{e}}$, which is equal to the number of entries in the $Q$-pattern by Corollary~\ref{cor:ov e}. Denote this number by $\abs{Q}$.

In this section, we consider $\Inv_{R \ov{e}}$ as a subfield of $\Inv_R = \Frac(\LSym)$. In \S\ref{sec:corner color}, we present a sufficient condition, due to Lam and the third-named author~\cite{LPaff}, for a loop skew Schur function to be $\ov{e}$-invariant. As discussed in the introduction, we call a loop skew Schur function satisfying this condition a \defn{pseudo-energy}. In \S\ref{sec:Q}, we identify a set of $\abs{Q}$ algebraically independent pseudo-energies, which consists of the shape invariants, and a collection of polynomials $\lQ{i}{j}$ that we call \defn{$Q$-invariants}. In \S\ref{sec:conj gen sets}, we discuss our conjectural sets for $\Inv_{R \ov{e}}$, $\Inv_{\ov{R} e}$, and $\Inv_{R \ov{R}}$. In \S\ref{sec:det formula}, we show that every pseudo-energy can be expressed as the determinant of a matrix whose entries are monomials in $\lQ{i}{j}$ and $S_i^{\pm 1}$, providing strong evidence that the $Q$-invariants and shape invariants generate $\Inv_{R \ov{e}}$.

\subsection{The corner color condition}
\label{sec:corner color}

Given a skew shape $\rho$ and a color $r \in \ZZ/n\ZZ$, let $\rho^{(r)}$ denote the coloring of $\rho$ in which cell $s = (i,j)$ is assigned the color $r + c(s) = r + i-j$ (considered modulo $n$). We call $\rho^{(r)}$ a \defn{colored skew shape}. Define a \defn{NW corner} of a skew shape $\rho$ to be a cell $(i,j) \in \rho$ such that $(i-1,j)$ and $(i,j-1)$ are not in $\rho$, and define a \defn{SE corner} to be a cell $(i,j) \in \rho$ such that $(i+1,j)$ and $(i,j+1)$ are not in $\rho$. These are the ``inner'' and ``outer'' removable corners, respectively; see Figure~\ref{fig:NW SE corners}. We say that a colored skew shape $\rho^{(r)}$ satisfies the \defn{corner color condition} if each NW corner has color $n$, and each SE corner has color $m$.

\begin{figure}
\[
\begin{tikzpicture}[scale=0.4]
\foreach \x/\y in {0/-5,1/-4,2/-2,4/0} {
  \fill[black!30] (\x,\y) rectangle (\x+1,\y-1);
}
\foreach \x/\y in {0/-4,1/-1,2/0} {
  \draw (\x+0.5,\y-0.5) node {$\times$};
}
\foreach \start/\end [count=\i from 0] in {2/5,1/5,1/3,1/3,0/2,0/2,0/1} {
  \draw[-] (\start,-\i) -- (\end,-\i);
}
\foreach \start/\end [count=\i from 0] in {-6/-4,-6/-1,-5/0,-3/0,-1/0,-1/0} {
  \draw[-] (\i,\start) -- (\i,\end);
}
\end{tikzpicture}
\]
\caption{The skew shape $(5,3,3,2,2,1) / (2,1,1,1)$, with NW corners marked with an $\times$ and SE corners shaded gray.}
\label{fig:NW SE corners}
\end{figure}
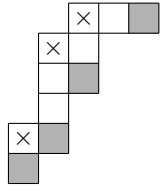

\begin{prop}[{\cite[Thm.~5.3]{LPaff}}]
\label{prop:corner color condition}
If $\rho^{(r)}$ satisfies the corner color condition, then the loop skew Schur function $s^{(r)}_{\rho}(\mb{x}_1, \ldots, \mb{x}_m)$ is $\ov{e}$-invariant.
\end{prop}

\begin{remark}
The converse of Proposition~\ref{prop:corner color condition} is false. Indeed, the colored shape $(n^m)^{(r)}$ satisfies the corner color condition if and only if $r = n$, but the loop Schur function $s_{(n^m)}^{(r)}$ is equal to $\lE{m}{1} \cdots \lE{m}{n}$ for any $r$, and is thus $\ov{e}$-invariant. We suspect that this is essentially the only way the converse can fail; that is, we suspect that if $s_{\rho}^{(r)}$ is $\ov{e}$-invariant, then $s_{\rho}^{(r)} = s_{\kappa}^{(r')}$ for some $\kappa^{(r')}$ which does satisfy the corner color condition.
\end{remark}

Proposition~\ref{prop:corner color condition}---and the ingredients needed to prove it---are of fundamental importance to this paper, so we will supply our own proof.

Let $\widetilde{M}(\mb{x}_1, \ldots, \mb{x}_m)$ be the $n$-periodic introduced in \S\ref{sec:LSym matrix}. We begin by describing the action of the barred geometric crystal operators on this matrix. For $j \in [n-1]$ and $z \in \CC$, let
\[
\wh{x}_j(z) = I + \sum_{d \in \ZZ} z E_{dn+j,dn+j+1},
\]
where $E_{ab}$ is the $n$-periodic matrix with ones in positions $(dn+a,dn+b)$ and zeroes elsewhere. The matrix $\wh{x}_j(z)$ is the unfolded version of the matrix $x_j(z)$ in the one-parameter subgroup of $\GL_n$ associated to the $j$-th simple root. The following result is a strengthening of~\eqref{eq:unipotent}, and is easily proved by induction on $m$.

\begin{lemma}
\label{lem:affine unipotent}
For $j \in [n-1]$, the action of the barred geometric crystal operator $\ov{e}_j^c$ on $\mb{x} = (\mb{x}_1, \ldots, \mb{x}_m)$ satisfies
\begin{equation}
\label{eq:affine unipotent}
\widetilde{M}(\ov{e}_j^c(\mb{x})) = \wh{x}_j\bigl( (c-1)\varphi_j(\mb{x}) \bigr) \cdot \widetilde{M}(\mb{x}) \cdot \wh{x}_j\bigl( (c^{-1}-1) \varepsilon_j(\mb{x}) \bigr).
\end{equation}
\end{lemma}

Say that a subset $I \subset \ZZ$ is \defn{$n$-initial} (resp., \defn{$n$-final}) if for every integer $d$, the intersection $I \cap [nd+1,nd+n]$ is a (possibly empty) initial interval $[nd+1,nd+k_d]$ (resp., final interval $[nd+k_d, nd+n]$).

\begin{lemma}
\label{lem:initial final condition}
Suppose $\mu \subseteq \la$ and $\la'_i - \mu'_i > 0$ for $i = 1, \ldots, \ell$, where $\ell$ is the length of $\la'$. The colored skew shape $(\la/\mu)^{(r)}$ satisfies the corner color condition if and only if $I(\mu,r)$ is $n$-final and $J(\la,r)$ is $n$-initial, where $I(\mu,r), J(\la,r)$ are defined by~\eqref{eq:subsets JT matrix version}.
\end{lemma}

\begin{remark}
\label{rem:no empty columns}
The condition $\la'_i - \mu'_i > 0$ for $i = 1, \ldots, \ell(\la')$ means that $\la/\mu$ has no empty columns. This restriction is necessary for Lemma~\ref{lem:initial final condition} to hold, and it is also necessary for some of our later results. Given a skew shape with empty columns, we can remove the empty columns without changing the associated loop skew Schur function by sliding connected components to the northwest. Thus, we do not lose any generality with this restriction.
\end{remark}

\begin{ex}
Suppose $m=5$ and $n=4$, and let $\la = (4,4,4,1)$, $\mu = (2,2)$, $r = 6$. The colored skew shape $(\la/\mu)^{(r)}$ satisfies the corner color condition (this can be seen in Example~\ref{ex:Q invariant}), and the sets
\[
I(\mu,r) = \{3,4,7,8\}, \qquad J(\la,r) = \{1,2,3,5\}
\]
are 4-final and 4-initial, respectively.
\end{ex}

\begin{proof}[Proof of Lemma~\ref{lem:initial final condition}]
Suppose $I(\mu,r) = \{i_\ell < \cdots < i_1\}$ and $J(\la,r) = \{j_\ell < \cdots < j_1\}$. The $k$th column of $(\la/\mu)^{(r)}$ has height $\la'_k - \mu'_k = m+j_k-i_k$, top cell of color $i_k$, and bottom cell of color $m+j_k-1$. The SE corners of $\la/\mu$ occur in column $\ell$, and the columns $k<\ell$ such that $j_k - j_{k+1} > 1$ (here we use the assumption that no column is empty). Thus, all the SE corners have color $m$ if and only if $j_k \equiv 1 \mod n$ whenever $k=\ell$ or $j_k$ and $j_{k+1}$ are not consecutive, i.e., if and only if $J(\la,r)$ is $n$-initial. Similarly, the NW corners occur in column 1, and the columns $k > 1$ such that $i_k-i_{k-1} > 1$. These corners all have color $n$ if and only if $i_k \equiv n \mod n$ whenever $k = 1$ or $i_k$ and $i_{k-1}$ are not consecutive, i.e., if and only if $I(\mu,r)$ is $n$-final.
\end{proof}

\begin{proof}[Proof of Proposition~\ref{prop:corner color condition}]
Suppose $\rho^{(r)}$ satisfies the corner color condition. As explained in Remark \ref{rem:no empty columns}, we may assume that $\rho = \la/\mu$, where $\mu_i' < \la_i'$ for $i = 1, \ldots, \ell(\la')$, so that $I(\mu,r)$ is $n$-final and $J(\la,r)$ is $n$-initial by Lemma~\ref{lem:initial final condition}. Lemma~\ref{lem:affine unipotent} implies that $\widetilde{M}(\ov{e}_j^c(\mb{x}_1, \ldots, \mb{x}_m))$ is obtained from $\widetilde{M}(\mb{x}_1, \ldots, \mb{x}_m)$ by adding a multiple of row $dn+j+1$ to row $dn+j$, and a multiple of column $dn+j$ to column $dn+j+1$, for all $d$. These operations do not affect minors consisting of an $n$-final set of rows and an $n$-initial set of columns, so by Proposition~\ref{prop:JT matrix version}, $s_{\rho}^{(r)}$ is $\ov{e}_j$-invariant for each $j \in [n-1]$.
\end{proof}

\begin{remark}
\label{rem:affine geometric crystal}
The basic geometric crystal of type $A_{n-1}$ introduced in \S\ref{sec:basic} extends to an affine geometric crystal of type $A_{n-1}^{(1)}$ by defining
\[
\ov{\varepsilon}_0(\mb{x}) = x_1,
\qquad\quad
\ov{\varphi}_0(\mb{x}) = x_n,
\qquad\quad
\ov{e}_0^c(\mb{x}) = (c^{-1}x_1, x_2, \dotsc, x_{n-1}, cx_n).
\]
The affine operator $\ov{e}_0^c$ also satisfies~\eqref{eq:affine unipotent}, where we extend the definition of $\wh{x}_j(z)$ by allowing $j = 0$. Thus, the matrix $\widetilde{M}(\mb{x}_1, \ldots, \mb{x}_m)$ and the ring $\LSym$ can be viewed as ``affine objects.'' However, since $\gRSK$ respects only the classical $(j \neq 0)$ part of the crystal structure, we do not consider affine crystal operators in this paper (aside from a brief discussion of affine $R$-matrix invariants in \S\ref{sec:affine actions}).
\end{remark}

\begin{remark}
The \defn{(intrinsic) energy function} is an important function on the tensor product of one-row crystals of type $A_{n-1}^{(1)}$. It was shown in~\cite{LP13} that the loop Schur function $s_{(n-1)\delta_{m-1}}^{(n)}(\mb{x}_1, \ldots, \mb{x}_m)$ tropicalizes to the intrinsic energy function, where $a\delta_b$ is the \defn{stretched staircase partition} $(ab,a(b-1), \ldots, 2a, a)$. The energy function is known to be invariant under the classical crystal operators $\widetilde{e}_j, j \in [n-1]$, and under permutations of tensor factors by combinatorial $R$-matrices, so it should come as no surprise that the colored shape $((n-1)\delta_{m-1})^{(n)}$ satisfies the corner color condition.
\end{remark}

In light of the previous remark, we make the following definition.

\begin{dfn}
A \defn{pseudo-energy} is a loop skew Schur function $s_{\rho}^{(r)}$ such that the colored skew shape $\rho^{(r)}$ satisfies the corner color condition.
\end{dfn}

\subsection{\texorpdfstring{$Q$}{Q}-invariants}
\label{sec:Q}

The number of entries in the $Q$-pattern is equal to the number of shape invariants, plus the number of $Q$-type loop elementary symmetric functions (recall that $\lE{i}{j}$ is $Q$-type if $i+j \leq m$). We have already seen that the shape invariants are $\ov{e}$-invariant. We would like to assign an $\ov{e}$-invariant to each $Q$-type $\lE{i}{j}$. The column of length $i$ with top cell of color $j$ does not in general satisfy the corner color condition; to get around this problem, we sandwich the column between two rectangles that do satisfy the corner color condition.

Suppose $i \in [m], j \in [n]$, and $i + j \leq m$. Let $0 \leq K \leq n-1$ be such that $K \equiv j+i-m-1 \mod n$. Define two partitions as follows:
\begin{gather*}
\lambda_i^j = (\underbrace{n-j+K+1, \ldots, n-j+K+1}_{m-n+K \text{ times}}, \underbrace{n-j, \ldots, n-j}_{m-i-j \text{ times}}),
\\
\mu_i^j = (\underbrace{n-j+1, \ldots, n-j+1}_{m-n+K-i \text{ times}}).
\end{gather*}
(Note that $K = j+i-m-1+dn$ for some $d > 0$, so we have $m-n+K-i \geq j-1 \geq 0$.) It is clear that $\mu_i^j \subseteq \la_i^j$, so we may consider the skew shape $\lambda_i^j/\mu_i^j$. In particular, we will be interested in the colored skew shape $(\lambda_i^j/\mu_i^j)^{(n-j+1)}$, which can be described as follows: we attach a column of length $i$, with top cell colored $j$, to the end of the first $i$ rows of a rectangle of size $(m-j) \times (n-j)$, and to the beginning of the last $i$ rows of a rectangle of size $(m-n+K) \times K$. These rectangles are the colored shapes of the shape invariants $S_{j+1}$ and $S_{n+1-K}$ (if $j = n$ or $K = 0$, then $S_{n+1} = 1$ corresponds to the empty rectangle).

\begin{lemma}
For $i,j,K$ as above, $\ls{\lambda_i^j/\mu_i^j}{n-j+1}$ is nonzero, and we have $j+1, n+1-K \in [2,\min(m,n)+1]$.
\end{lemma}

\begin{proof}
For a loop skew Schur function in $m$ variables to be nonzero, it is necessary and sufficient for each column of the skew shape to have length no larger than $m$. The lengths of the columns of $\lambda_i^j/\mu_i^j$ are $m-j, i$, and $m-n+K$, each of which is no larger than $m-1$. For the second assertion, the only inequality that is not immediately clear is $n+1-K \leq m+1$, in the case where $m < n$. In this case, we have $1 \leq m+1-j-i \leq m-1 < n-1$, so $K = n-(m+1-j-i)$, which means $n+1-K = m+2-j-i \leq m$.
\end{proof}

\begin{dfn}
\label{defn:Q invariant}
For $i \in [m]$ and $j \in [n]$ such that $i+j \leq m$, define the \defn{$Q$-invariant} $\lQ{i}{j}$ and \defn{reduced $Q$-invariant} $\lrQ{i}{j}$ by
\[
\lQ{i}{j} = \ls{\lambda_i^{(j)}/\mu_i^{(j)}}{n-j+1}, \qquad \lrQ{i}{j} = \frac{\ls{\lambda_i^{(j)}/\mu_i^{(j)}}{n-j+1}} {S_{j+1} S_{n+1-K}}.
\]
\end{dfn}

It is easy to verify that $(\lambda_i^j/\mu_i^j)^{(n-j+1)}$ satisfies the corner color condition, so $\lQ{i}{j}$ and $\lrQ{i}{j}$ are $\ov{e}$-invariant.

\begin{ex}
\label{ex:Q invariant}
Suppose $m = 5,n=4$ and $i=1,j=3$. From the above definitions, we have $K = 2$, $\lambda_1^3 = (4,4,4,1)$, $\mu_1^3 = (2,2)$, so
\[
\lQ{1}{3} = s_{(4,4,4,1)/(2,2)}^{(2)}, \qquad \lrQ{1}{3} = \frac{s_{(4,4,4,1)/(2,2)}^{(2)}}{S_4 S_3} = \frac{s_{(4,4,4,1)/(2,2)}^{(2)}}{s_{(1,1)}^{(4)}s_{(2,2,2)}^{(4)}}.
\]
The colors of the cells of the skew shape $(\lambda_1^3/\mu_1^3)^{(2)}$ are shown below, where Orange is $1$, Blue is $2$, Red is $3$, and Green is $4$:
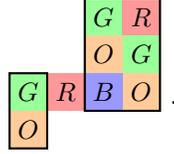
\[
\begin{tikzpicture}[baseline=-40,scale=0.5]
\foreach \x/\y in {2/0,3/1,0/2} {
  \fill[green!40] (\x,-\y) rectangle (\x+1,-\y-1);
  \draw (\x+0.5,-\y-0.5) node {$G$};
}
\foreach \x/\y in {2/1,3/2,0/3} {
  \fill[orange!40] (\x,-\y) rectangle (\x+1,-\y-1);
  \draw (\x+0.5,-\y-0.5) node {$O$};
}
\foreach \x/\y in {3/0,1/2} {
  \fill[red!40] (\x,-\y) rectangle (\x+1,-\y-1);
  \draw (\x+0.5,-\y-0.5) node {$R$};
}
\fill[blue!40] (2,-2) rectangle (3,-3);
\draw (2.5,-2.5) node {$B$};
\draw[thick] (0,-2) rectangle (1,-4);
\draw[thick] (2,-0) rectangle (4,-3);
\end{tikzpicture}\ .
\]
Note that the NW corners have color $n=4$ (Green), while the SE corners have color $m=5 \equiv 1 \mod 4$ (Orange).
\end{ex}

We saw in Lemma~\ref{lem:P} that the $P$-type loop elementary symmetric functions contain the same information as the gRSK $P$-pattern. If we combine the $Q$-invariants with the information contained in the $P$-pattern, we can recover all loop elementary symmetric functions.

\begin{prop}
\label{prop:Q}
The $Q$-type loop elementary symmetric functions can be recovered from the $Q$-invariants and the $P$-type loop elementary symmetric functions. Thus, the $Q$-invariants and $P$-type loop elementary symmetric functions (or, equivalently, the $Q$-invariants and the entries $z_{i,j}$ of the $P$-pattern) form an algebraically independent generating set for $\Frac(\LSym)$.
\end{prop}

\begin{proof}
Suppose $\lE{i}{j}$ is a $Q$-type loop elementary symmetric function. By Proposition~\ref{prop:JT matrix version}, $s_{\lambda_i^j/\mu_i^j}^{(n-j+1)}$ is equal to the determinant of the submatrix of $\widetilde{M}(\mb{x}_1, \ldots, \mb{x}_m)$ using rows $I$ and columns $J$, where
\begin{align*}
I &= [n-K+1, \ldots, n] \cup [dn + j, \ldots, dn + n], \\
J &= [1, \ldots, K+1] \cup [dn+1, dn+n-j],
\end{align*}
and $d$ is the (necessarily positive) integer such that $K = i+j-m-1+dn$. This submatrix has block form
\[
\begin{tikzpicture}
\draw (0,0) rectangle (2,2);
\draw (0,1.1) -- (2,1.1);
\draw (0.95,0) -- (0.95,2);
\draw (0.475,0.55) node{$C$};
\draw (0.475,1.55) node{$B$};
\draw (1.475,0.55) node{$A$};
\draw (1.475,1.55) node{$0$};
\draw (-0.9,0.55) node{$n-j+1 \Bigg\{$};
\draw (-0.35,1.55) node{$K \bigg\{$};
\draw (0.475,2.3) node{$\braceabove{0.85}{K+1}$};
\draw (1.475,2.3) node{$\braceabove{0.95}{n-j}$};
\end{tikzpicture}
\]
where $A$ and $B$ are bottom-left justified submatrices of the top nonzero block $M_0$ (which consists of $P$-type loop elementary symmetric functions), and $C$ is a bottom-left justified submatrix of the lower block $M_d$. Moreover, the entry in the top-right corner of $C$ is $\lE{i}{j}$, and the other entries of $C$ are $\lE{i'}{j'}$ with $i' < i$. Thus, we have
\[
\lQ{i}{j} = \lE{i}{j} S_{n+1-K}S_{j+1} + \xi,
\]
where $\xi$ is a polynomial in $P$-type loop elementary symmetric functions and $\lE{i'}{j'}$ with $i' < i$. It follows by induction on $i$ that we can solve for $\lE{i}{j}$ in terms of $Q$-invariants and $P$-type loop elementary symmetric functions.

The second assertion follows from the fact that the set of loop elementary symmetric functions is algebraically independent, and has the same cardinality as the set of $Q$-invariants and $P$-type loop elementary symmetric functions.
\end{proof}

\begin{remark}
The $Q$-invariants and $P$-type loop elementary symmetric functions do not generate $\LSym$ as a ring. Indeed, when $m=n=2$, $\lE{1}{2}, \lE{2}{1},\lE{2}{2}$ are $P$-type, $\lE{1}{1}$ is $Q$-type, and
\[
\lE{1}{1} = \dfrac{\lQ{1}{1} + \lE{1}{2}\left(\lE{2}{1} + \lE{2}{2}\right)}{\left(\lE{1}{2}\right)^2}.
\]
In this example, the denominator is equal to $(S_2)^2$. In general, it is clear from the above proof that every loop elementary symmetric function (and thus every loop symmetric function) can be expressed as a polynomial in $Q$-invariants, $P$-type loop elementary symmetric functions, and inverses of shape invariants.
\end{remark}

We will show in \S\ref{sec:back to tableau} that the $Q$-invariants are positive Laurent polynomials in the entries of the $Q$-pattern. We conjecture that these Laurent polynomials, together with the shape, determine the $Q$-pattern up to the geometric Weyl group action of $S_m$.

\subsection{Conjectural generating sets}
\label{sec:conj gen sets}

\begin{conj}
\label{conj:QS}
\
\begin{enumerate}
\item The field $\Inv_{R\ov{e}} = \Inv_{\ov{e}} \cap \Frac(\LSym)$ is generated by the $Q$-invariants and shape invariants.
\item Every element of $\CC[x_i^j] \cap \Inv_{R\ov{e}}$ is a polynomial in reduced $Q$-invariants $\lrQ{i}{j}$ and ratios of consecutive shape invariants $S_i/S_{i+1}$ (by definition, $S_{\min(m,n) + 1} = 1$).
\end{enumerate}
\end{conj}

Note that the collection of $Q$-invariants and loop Schur functions $\Box(i,j)$ is algebraically independent by Lemma~\ref{lem:P},~\eqref{eq:P loop formula}, and Proposition~\ref{prop:Q}, so in particular the collection of $Q$-invariants and shape invariants is algebraically independent. Since we know that the size of this collection is equal to the transcendence degree of $\Inv_{R\ov{e}}$ (see the discussion at the beginning of~\S \ref{sec:double}), $\Inv_{R\ov{e}}$ differs from $\CC(\lQ{i}{j}, S_k)$ by at most a finite extension. We provide strong supporting evidence for Conjecture~\ref{conj:QS} in \S\ref{sec:det formula}.

\begin{remark}
The $Q$-invariants and shape invariants do not generate the ring $\CC[x_i^j] \cap \Inv_{R\ov{e}}$. For example, when $m=n=2$, one has
\[
\lx{1}{1}\lx{1}{2} + \lx{2}{1}\lx{2}{2} = \dfrac{\lH{3}{2}}{\lE{1}{2}} = \dfrac{\lQ{1}{1}}{S_2}.
\]
On the other hand, a consequence of Conjecture~\ref{conj:QS} is that the larger ring $\CC[x_i^j,S_k^{-1}] \cap \Inv_{R\ov{e}}$ is generated by the $Q$-invariants, shape invariants, and inverses of shape invariants.
\end{remark}

Define \defn{barred $P$-invariants} $\bP{j}{i} \in \ov{\LSym}$ and \defn{reduced barred $P$-invariants} $\brP{j}{i} \in \Frac(\ov{\LSym})$ for $i \in [m], j \in [n]$ such that $i+j \leq n$ in the same way that the $Q$-invariants were defined, but with the roles of $m,n$ (and $i,j$) interchanged. Conjecture~\ref{conj:QS} is equivalent to the analogous conjecture for $\Inv_{\ov{R}e}$, with $Q$-invariants replaced by barred $P$-invariants.

Finally, we consider $\Inv_{R\ov{R}} = \Frac(\LSym) \cap \Frac(\ov{\LSym})$. The $Q$-invariants are $\ov{e}$-invariant, so in particular they are $\ov{R}$-invariant, which means they lie in $\ov{\LSym}$. Similarly, the barred $P$-invariants lie in $\LSym$.

\begin{conj}
\label{conj:QPS}
\
\begin{enumerate}
\item The field $\Inv_{R\ov{R}} = \Frac(\LSym) \cap \Frac(\ov{\LSym})$ is generated by the $Q$-invariants, barred $P$-invariants, and shape invariants.
\item Every element of $\LSym \cap \ov{\LSym}$ is a polynomial in reduced $Q$-invariants, reduced barred $P$-invariants, and ratios of consecutive shape invariants $S_i/S_{i+1}$.
\end{enumerate}
\end{conj}

\begin{prop}
\label{prop:QPS ind}
The collection of $Q$-invariants, barred $P$-invariants, and shape invariants is algebraically independent.
\end{prop}

The size of the collection of $Q$-invariants, barred $P$-invariants, and shape invariants is equal to the number of variables $x_i^j$, so Proposition~\ref{prop:QPS ind} implies that $\Inv_{R \ov{R}}$ is at most a finite extension of the field generated by this collection. The proof of Proposition~\ref{prop:QPS ind} relies on a technical lemma.

\begin{lemma}
\label{lem:alg ind}
Let $K \subset L$ be fields, and let $X = \{x_1, \ldots, x_a\}, Y = \{y_1, \ldots, y_b\}, Z = \{z_1, \ldots, z_c\}$ be disjoint subsets of $L$ such that $X \cup Y \cup Z$ is algebraically independent over $K$. Let $F = \{f_1(X,Z), \ldots, f_a(X,Z)\}$ and $G = \{g_1(Y,Z), \ldots, g_b(Y,Z)\}$, where $f_i \in K(X,Z)$, and $g_j \in K(Y,Z)$. If $F \cup Z$ and $G \cup Z$ are algebraically independent over $K$, then $F \cup Z \cup G$ is algebraically independent over $K$.
\end{lemma}

\begin{proof}
It is a standard result in field theory that if $A$ is a finite subset of $L$, then the collection of subsets of $A$ which are algebraically independent over $K$ form the independent sets of a matroid (see, e.g.,~\cite[Thm. 6.7.1]{Oxley}). (In the remainder of the proof we drop ``algebraically'' and just say ``independent'' and ``dependent.'') This means that if $I_1,I_2 \subseteq A$ are independent and $\abs{I_1} < \abs{I_2}$, then there is an element $a \in I_2 \setminus I_1$ such that $I_1 \cup \{a\}$ is independent; this the \defn{augmentation axiom} for matroids. In particular, the augmentation axiom implies that all maximal independent subsets of $A$ have the same size.

We first show that $G \cup Z \cup X$ is independent. Take $I_1 = G \cup Z$ and $I_2 = X \cup Y \cup Z$. Using the augmentation axiom $a$ times, we may extend $I_1$ to an independent set of size $a+b+c$ by adding elements of $I_2$. Since $Y \cup Z$ is a maximal independent subset of $K(Y,Z)$, $G \cup Z \cup \{y_i\}$ must be dependent for all $y_i \in Y$. Thus, the only elements of $I_2$ which are independent of $I_1$ are those in $X$, so $G \cup Z \cup X$ must be independent. Now take $I_1 = F \cup Z$ and $I_2 = G \cup Z \cup X$. By the same reasoning, the only elements of $I_2$ which are independent of $I_1$ are those of $G$, so the augmentation axiom implies that $F \cup Z \cup G$ is independent.
\end{proof}

\begin{proof}[Proof of Proposition~\ref{prop:QPS ind}]
Let $X$ be the entries of the $P$-pattern which are not part of the shape, $Y$ the entries of the $Q$-pattern which are not part of the shape, and $Z$ the entries of the common shape of the two patterns. (We may freely interchange $Z$ with the collection of shape invariants, since both sets generate the same field). The geometric RSK correspondence induces an automorphism of $\CC(x_i^j)$ which sends the variables $x_i^j$ to the elements of $X \cup Y \cup Z$, so this set is algebraically independent. Corollary~\ref{cor:ov e} implies that the $Q$-invariants are in the field generated by $Y \cup Z$, and it was observed after Conjecture~\ref{conj:QS} that the $Q$-invariants and shape invariants are algebraically independent. By the symmetry of gRSK, the barred $P$-invariants are in the field generated by $X \cup Z$, and the barred $P$-invariants and shape invariants are algebraically independent. We may therefore apply Lemma~\ref{lem:alg ind}.
\end{proof}

\subsection{Determinantal formula for pseudo-energies}
\label{sec:det formula}

Conjecture~\ref{conj:QS} predicts that if $\rho^{(r)}$ satisfies the corner color condition, then $s^{(r)}_{\rho}$ can be written as a polynomial in reduced $Q$-invariants $\lrQ{i}{j}$ and ratios of consecutive shape invariants $S_i/S_{i+1}$. In this section, we confirm that prediction by giving a new determinantal formula for these loop skew Schur functions.

Define an $n \times n$ matrix $\wh{M}'(t)$ by
\[
\wh{M}'(t)_{ij} =
M'_{ij} + \sum_{d > 0} t^d \lrQ{m+j-i-nd}{i}
\qquad \text{ where } \qquad
M'_{ij} = \begin{cases}
(-1)^{n-i} \dfrac{S_i}{S_{i+1}} & \text{ if } i \leq m \text{ and } i+j = n+1 \\
1 & \text{ if } i-j = m \\
0 & \text{ otherwise}
\end{cases}
\]
(by convention, $S_k = 1$ if $k = \min(m,n)+1$, $\lrQ{0}{i} = 1$, and $\lrQ{k}{i} = 0$ if $k < 0$). Let $\widetilde{M}'$ be the unfolded version of $\wh{M}'(t)$. In words, $\widetilde{M}'$ is obtained from $\widetilde{M}$ by replacing the ``$P$-type block'' $M = M_0$ with the matrix $M'$, and each $Q$-type loop elementary symmetric function $\lE{i}{j}$ with the reduced $Q$-invariant $\lrQ{i}{j}$. See Figure~\ref{fig:ex M'} for some examples.

\begin{figure}
\[
\left( \begin{array}{ccc|ccc|c}
0 & 0 & \frac{S_1}{S_2}     & 0 & 0 & 0 & \hdots \\
0 & -\frac{S_2}{S_3} & 0     & 0 & 0 & 0 \\
S_3 & 0 & 0                       & 0 & 0 & 0 \\ \hline
\lrQ{2}{1} & \lrQ{3}{1}   & \lrQ{4}{1} & 0 & 0 & \frac{S_1}{S_2}  \\ 
\lrQ{1}{2} & \lrQ{2}{2}   & \lrQ{3}{2} & 0 & -\frac{S_2}{S_3} & 0  \\
1              & \lrQ{1}{3}   & \lrQ{2}{3} & S_3 & 0 & 0   \\ \hline
0              & 1                & \lrQ{1}{1} & \lrQ{2}{1} & \lrQ{3}{1} & \lrQ{4}{1} \\
0              & 0                & 1              & \lrQ{1}{2} & \lrQ{2}{2} & \lrQ{3}{2} \\
0              & 0                & 0              & 1              & \lrQ{1}{3} & \lrQ{2}{3} \\ \hline
\vdots &  &  & & & & \ddots
\end{array} \right)
\qquad
\left( \begin{array}{ccc@{\;}c@{\;}c|ccc@{\;}c@{\;}c|c}
0 & 0 & 0     & 0 & \frac{S_1}{S_2}  & 0 & 0 & 0 & 0 & 0 & \hdots \\
0 & 0 & 0     & -\frac{S_2}{S_3} & 0 & 0 & 0 & 0 & 0 & 0 \\
0 & 0 & S_3 & 0                        & 0 & 0 & 0 & 0 & 0 & 0 \\
1 & 0 & 0 & 0 & 0                             & 0 & 0 & 0 & 0 & 0 \\
0 & 1 & 0 & 0 & 0                             & 0 & 0 & 0 & 0 & 0\\ \hline
0 & 0 & 1 & \lrQ{1}{1} & \lrQ{2}{1} & 0 & 0 & 0 & 0 & \frac{S_1}{S_2} \\
0 & 0 & 0 & 1 & \lrQ{1}{2}          & 0            & 0 & 0 & -\frac{S_2}{S_3} & 0 \\
0 & 0 & 0 & 0 & 1           & 0 & 0 & S_3 & 0 & 0 \\
0 & 0 & 0 & 0 & 0 & 1 & 0 & 0 & 0 & 0 \\
0 & 0 & 0 & 0 & 0 & 0 & 1 & 0 & 0 & 0 \\ \hline
\vdots &  &  & & & & & & & & \ddots
\end{array} \right)
\]
\caption{The matrix $\widetilde{M}'$ for $m=5,n=3$ (left) and $m=3,n=5$ (right).}
\label{fig:ex M'}
\end{figure}

\begin{thm}
\label{thm:new det formula}
Suppose $\mu \subseteq \la$ and $\la'_i - \mu'_i > 0$ for $i = 1, \ldots, \ell$, where $\ell$ is the length of $\la'$. If $(\la/\mu)^{(r)}$ satisfies the corner color condition, then
\begin{equation}
\label{eq:new det formula}
s_{\la/\mu}^{(r)} = \Delta_{I(\mu,r), J(\la,r)}(\widetilde{M}'),
\end{equation}
where $I(\mu,r), J(\la,r)$ are defined by~\eqref{eq:subsets JT matrix version}.
\end{thm}

As explained in Remark~\ref{rem:no empty columns}, all pseudo-energies can be written in the form required by the theorem.

\begin{ex}
\
\begin{enumerate}
\ytableausetup{smalltableaux}
\item[(a)]
Suppose $m = 5, n=3, \la = (4,3,3,1), \mu = (2), r = 2$. The colored skew shape $(\la/\mu)^{(r)}$ satisfies the corner color condition, so according to Theorem~\ref{thm:new det formula}, we have
\[
s^{(2)}_{\ytableaushort{\none\none\empty\empty,\empty\empty\empty,\empty\empty\empty,\empty}} = \Delta_{5689, 1457}(\widetilde{M}') = \det
\begin{pmatrix}
\lrQ{1}{2} & 0 & -\frac{S_2}{S_3} & 0 \\
1 & S_3 & 0 & 0 \\
0 & \lrQ{1}{2} & \lrQ{2}{2} & 0 \\
0 & 1 & \lrQ{1}{3} & S_3
\end{pmatrix}
= \lrQ{1}{2} \lrQ{2}{2} S_3^2 - \lrQ{1}{2} S_2.
\]
Note that this is not a positive formula in terms of the reduced $Q$-invariants and shape invariants.
For comparison, the Jacobi--Trudi formula (in the form given by Proposition~\ref{prop:JT matrix version}) says that
\[
s^{(2)}_{\ytableaushort{\none\none\empty\empty,\empty\empty\empty,\empty\empty\empty,\empty}} = \Delta_{5689, 1457}(\widetilde{M}) = \det
\begin{pmatrix}
\lE{1}{2} & \lE{4}{2} & \lE{5}{2} & 0 \\
1 & \lE{3}{3} & \lE{4}{3} & 0 \\
0 & \lE{1}{2} & \lE{2}{2} & \lE{4}{2} \\
0 & 1 & \lE{1}{3} & \lE{3}{3}
\end{pmatrix}.
\]

\item[(b)]
It is easy to see that Theorem~\ref{thm:new det formula} holds for $Q$-invariants. For example, consider the $Q$-invariant $\lQ{1}{2}$ in the case $m = 3, n=5$. By definition,
\[
\lQ{1}{2} = s^{(4)}_{(8,8)/(4)} \qquad \text{ and } \qquad \lrQ{1}{2} = \dfrac{s^{(4)}_{(8,8)/(4)}}{S_3 S_2}.
\]
On the other hand, the colored skew shape $((8,8)/(4))^{(4)}$ satisfies the corner color condition, so Theorem~\ref{thm:new det formula} asserts that
\[
s^{(4)}_{(8,8)/(4)} = \Delta_{\{2,3,4,5,7,8,9,10\},\{1,2,3,4,5,6,7,8\}}(\widetilde{M}').
\]
The reader may confirm this by inspecting the matrix $\widetilde{M}'$, which is shown on the right in Figure~\ref{fig:ex M'}.
\end{enumerate}
\end{ex}

To prove Theorem~\ref{thm:new det formula}, we will write down explicit upper uni-triangular matrices $U$ and $V$ such that $UMV = M'$, and then we will show that multiplying the lower blocks $M_d$ by $U$ and $V$ has the effect of turning $Q$-type loop elementary symmetric functions into reduced $Q$-invariants. The following technical lemma contains the tools needed to carry out this argument.

\begin{lemma}
\label{lem:det stuff}
\
\begin{enumerate}
\item
\label{item:UAU}
Let $N$ be an $n \times n$ matrix such that the principal anti-diagonal minors $\Delta_i = \Delta_{[i,n],[1,n+1-i]}(N)$ are nonzero for $i = 2, \ldots, n$. Define upper uni-triangular $n \times n$ matrices $U = U(N), V = V(N)$ by
\[
U_{ij} = \begin{cases}
(-1)^{i+j} \dfrac{\Delta_{[i,n] \setminus \{j\}, [1,n-i]}(N)}{\Delta_{[i+1,n],[1,n-i]}(N)} & \text{ if } i \leq j \medskip \\
0 & \text{ if } i > j,
\end{cases}
\qquad
V_{ij} = \begin{cases}
(-1)^{i+j} \dfrac{\Delta_{[n-j+2,n],[1,j] \setminus \{i\}}(N)}{\Delta_{[n-j+2,n],[1,j-1]}(N)} & \text{ if } i \leq j \medskip \\
0 & \text{ if } i > j.
\end{cases}
\]
Then $UNV$ is the anti-diagonal matrix whose entry in position $(i,n+1-i)$ is $(-1)^{n-i} \frac{\Delta_i}{\Delta_{i+1}}$ (with $\Delta_{n+1} = 1$).

\item
\label{item:UAU m<n}
Suppose $1 \leq m < n$, and let $N$ be an $n \times n$ matrix of block form
\[
\begin{tikzpicture}
\draw (-2.1,1) node{$N =$};
\draw (0,0) rectangle (2,2);
\draw (0,1.2) -- (2,1.2);
\draw (1.2,0) -- (1.2,2);
\draw (0.6,0.6) node{$N_3$};
\draw (0.6,1.6) node{$N_1$};
\draw (1.6,0.6) node{$N_4$};
\draw (1.6,1.6) node{$N_2$};
\draw (-0.65,0.6) node{$n-m \Bigg\{$};
\draw (-0.35,1.6) node{$m \Bigl\{$};
\draw (0.6,-0.3) node{$\bracebelow{1.1}{n-m}$};
\draw (1.6,-0.3) node{$\bracebelow{0.7}{m}$};
\end{tikzpicture}
\]
where $N_3$ is upper uni-triangular, and the principal anti-diagonal minors $\Delta_i = \Delta_{[i,n],[1,n+1-i]}(N)$ are nonzero for $i = 2, \ldots, m$ (note that $\Delta_{m+1} = \det(N_3) = 1$ is also nonzero). Define upper uni-triangular $n \times n$ matrices $U = U(N), V = V(N)$ by
\[
\begin{tikzpicture}
\draw (-2.1,1) node{$U =$};
\draw (0,0) rectangle (2,2);
\draw (0,1.2) -- (2,1.2);
\draw (0.8,0) -- (0.8,1.2);
\draw (0.4,0.6) node{$0$};
\draw (1,1.6) node{$U'$};
\draw (1.4,0.6) node{$N_3^{-1}$};
\draw (-0.65,0.6) node{$n-m \Bigg\{$};
\draw (-0.35,1.6) node{$m \Bigl\{$};
\draw (0.4,-0.3) node{$\bracebelow{0.7}{m}$};
\draw (1.4,-0.3) node{$\bracebelow{1.1}{n-m}$};

\begin{scope}[xshift=6cm]
\draw (-2.1,1) node{$V =$};
\draw (0,0) rectangle (2,2);
\draw (0,0.8) -- (1.2,0.8);
\draw (1.2,0) -- (1.2,2);
\draw (0.6,0.4) node{$0$};
\draw (0.6,1.4) node{$I$};
\draw (1.6,1) node{$V'$};
\draw (0.6,-0.3) node{$\bracebelow{1.1}{n-m}$};
\draw (1.6,-0.3) node{$\bracebelow{0.7}{m}$};
\draw (-0.65,1.4) node{$n-m \Bigg\{$};
\draw (-0.35,0.4) node{$m \Bigl\{$};
\end{scope}
\end{tikzpicture}
\]
where
\begin{align*}
U_{ij} = U'_{ij} &= \begin{cases}
(-1)^{i+j} \dfrac{\Delta_{[i,n] \setminus \{j\}, [1,n-i]}(N)}{\Delta_{[i+1,n],[1,n-i]}(N)} & \text{ if } i \in [m], j \in [n], \text{ and } i \leq j \medskip \\
0 & \text{ if } i \in [m], j \in [n], \text{ and } i > j,
\end{cases}
\\
V_{ij} = V'_{i,j-(n-m)} &= \begin{cases}
(-1)^{i+j} \dfrac{\Delta_{[n-j+2,n],[1,j] \setminus \{i\}}(N)}{\Delta_{[n-j+2,n],[1,j-1]}(N)} & \text{ if } i \in [n], j \in [n-m+1,n], \text{ and } i \leq j \medskip \\
0 & \text{ if } i \in [n], j \in [n-m+1,n], \text{ and } i > j.
\end{cases}
\end{align*}
Then $UNV$ has block form
\[
\begin{tikzpicture}
\draw (-2.3,1) node{$UNV =$};
\draw (0,0) rectangle (2,2);
\draw (0,1.2) -- (2,1.2);
\draw (1.2,0) -- (1.2,2);
\draw (0.6,0.6) node{$I$};
\draw (0.6,1.6) node{$0$};
\draw (1.6,0.6) node{$0$};
\draw (1.6,1.6) node{$A$};
\draw (-0.65,0.6) node{$n-m \Bigg\{$};
\draw (-0.35,1.6) node{$m \Bigl\{$};
\draw (0.6,-0.3) node{$\bracebelow{1.1}{n-m}$};
\draw (1.6,-0.3) node{$\bracebelow{0.7}{m}$};
\end{tikzpicture}
\]
where $A$ is the anti-diagonal $m \times m$ matrix whose entry in position $(i,m+1-i)$ is $(-1)^{n-i} \frac{\Delta_i}{\Delta_{i+1}}$.

\item
\label{item:ABC}
Let $N$ be a $(p+q+1) \times (p+q+1)$ matrix of block form
\[
\begin{tikzpicture}
\draw (-2.1,1) node{$N =$};
\draw (0,0) rectangle (2,2);
\draw (0,1.1) -- (2,1.1);
\draw (0.95,0) -- (0.95,2);
\draw (0.475,0.55) node{$C$};
\draw (0.475,1.55) node{$B$};
\draw (1.475,0.55) node{$A$};
\draw (1.475,1.55) node{$0$};
\draw (-0.55,0.55) node{$p+1 \Bigg\{$};
\draw (-0.25,1.55) node{$q \bigg\{$};
\draw (0.475,2.3) node{$\braceabove{0.85}{q+1}$};
\draw (1.475,2.3) node{$\braceabove{0.95}{p}$};
\draw (2.3,0.9) node{.};
\end{tikzpicture}
\]
We have
\[
\det(N) = \sum_{a = 1}^{p+1} \sum_{b=1}^{q+1} (-1)^{q+a+b} C_{ab} \det(A_{\wh{a}}) \det(B^{\wh{b}}),
\]
where $A_{\wh{a}}$ is the matrix $A$ with row $a$ removed, and $B^{\wh{b}}$ is the matrix $B$ with column $b$ removed.
\end{enumerate}
\end{lemma}

\begin{proof}
We first prove~\eqref{item:UAU}. By definition,
\[
(UN)_{ij} = \dfrac{1}{\Delta_{[i+1,n], [1,n-i]}(N)} \sum_{a=i}^n (-1)^{i+a} \Delta_{[i,n] \setminus \{a\}, [1,n-i]}(N) N_{aj}.
\]
The sum on the right-hand side is (up to sign) the determinant of the $(n-i+1) \times (n-i+1)$ matrix $A$ obtained by appending a column consisting of entries $(N_{ij}, \ldots, N_{nj})$ to the submatrix $N_{[i,n], [1,n-i]}$. If $j \leq n-i$, then $A$ has a repeated column, so $(UN)_{ij} = 0$. If $j = n+1-i$, then $A$ is the principal anti-diagonal submatrix $N_{[i,n],[1,n+1-i]}$, so we have
\[
(UN)_{i,n+1-i} = (-1)^{n-i} \frac{\Delta_i}{\Delta_{i+1}}.
\]
Multiplying $UN$ by $V$ on the right adds scalar multiples of each column of $UN$ to the columns strictly to their right, so $(UNV)_{ij} = (UN)_{ij}$ for $i+j \leq n+1$. A similar argument shows that if $i+j \geq n+2$, then $(UNV)_{ij} = (NV)_{ij} = 0$. This proves~\eqref{item:UAU}.

The proof of~\eqref{item:UAU m<n} is similar. Using block-multiplication, we see that the bottom-left $(n-m) \times (n-m)$ block of $UNV$ is the identity matrix. For the entries outside of this block, the argument used to prove~\eqref{item:UAU} goes through without modification.

To prove~\eqref{item:ABC}, we expand the determinant of $N$ along the first $q$ rows:
\begin{align*}
\det(N) &= \sum_{J \in \binom{[p+q+1]}{q}} (-1)^{\binom{q+1}{2} + \sum_{j \in J} j} \det(N_{[q],J}) \det(N_{[p+q+1] \setminus [q], [p+q+1] \setminus J}) \\
&= \sum_{b=1}^{q+1} (-1)^{q+1+b} \det(B^{\wh{b}}) \det(C^b \sqcup A) \\
& = \sum_{b=1}^{q+1} (-1)^{q+1+b} \det(B^{\wh{b}}) \sum_{a=1}^{p+1} (-1)^{a+1} C_{ab} \det(A_{\wh{a}}).
\end{align*}
In the second line, $C^b \sqcup A$ is the matrix consisting of the $b$-th column of $C$ followed by $A$.
\end{proof}

\begin{proof}[Proof of Theorem~\ref{thm:new det formula}]
Define upper uni-triangular $n \times n$ matrices $U = U(M)$ and $V = V(M)$ as in Lemma~\ref{lem:det stuff}\eqref{item:UAU} (if $m \geq n$) or Lemma~\ref{lem:det stuff}\eqref{item:UAU m<n} (if $m < n$), and let $\widetilde{U}, \widetilde{V}$ be the unfolded versions of $U,V$. We will show that
\begin{equation}
\label{eq:UAU tilde decomp}
\widetilde{M}' = \widetilde{U} \widetilde{M} \widetilde{V}.
\end{equation}
Multiplying $\widetilde{M}$ by $\widetilde{U}$ on the left (resp., $\widetilde{V}$ on the right) has the effect of adding a multiple of row $nd+j$ (resp., column $nd+i$) to row $nd+i$ (resp., column $nd+j$) for each integer $d$, and each pair $1 \leq i \leq j \leq n$. These operations do not affect minors of the form $\Delta_{I,J}(\widetilde{M})$ where $I$ is $n$-final and $J$ is $n$-initial, so the assertion of the theorem follows from Proposition~\ref{prop:JT matrix version}, Lemma~\ref{lem:initial final condition}, and~\eqref{eq:UAU tilde decomp}. Furthermore, we see from this description that~\eqref{eq:UAU tilde decomp} is equivalent to the equations
\begin{equation}
\label{eq:UAU decomp}
M'_d = U M_d V
\end{equation}
for all $d \geq 0$, where $M'_d$ is the $n \times n$ matrix consisting of the coefficients of $t^d$ in the entries of $\wh{M}'(t)$.

For $d = 0$,~\eqref{eq:UAU decomp} follows immediately from Lemma~\ref{lem:det stuff}\eqref{item:UAU} or~\eqref{item:UAU m<n}. Now suppose $d > 0$. By definition,
\[
(M'_d)_{ij} = \lrQ{m+j-i-nd}{i} = \begin{cases}
\dfrac{\lQ{m+j-i-nd}{i}}{S_{i+1}S_{n+1-K}} & \text{ if } m+j-i-nd > 0 \\
1 & \text{ if } m+j-i-nd = 0 \\
0 & \text{ if } m+j-i-nd < 0,
\end{cases}
\]
where $K = j-1$. Since $U$ and $V$ are upper triangular, we have
\begin{equation}
\label{eq:UM_dV}
(UM_dV)_{ij} = \sum_{a = i}^n \sum_{b = 1}^j U_{ia} (M_d)_{ab} V_{bj}.
\end{equation}
If $m+j-i-nd = 0$, then $M_{ij} = 1$ and $M_{ab} = 0$ for $a > i$ or $b < j$, so $(UM_dV)_{ij} = U_{ii}M_{ij}V_{jj} = 1$. If $m+j-i-nd < 0$, then all entries of $M_d$ appearing in~\eqref{eq:UM_dV} are zero, so $(UM_dV)_{ij} = 0$.

If $m+j-i-nd > 0$, then we have
\begin{equation}
\label{eq:M entry}
\sum_{a = i}^n \sum_{b = 1}^j U_{ia} (M_d)_{ab} V_{bj} = \sum_{a = i}^n \sum_{b = 1}^j (-1)^{i+j+a+b} (M_d)_{ab} \dfrac{\Delta_{[i,n] \setminus \{a\}, [1,n-i]}(M)}{S_{i+1}} \dfrac{\Delta_{[n-j+2,n],[1,j] \setminus \{b\}}(M)}{S_{n-j+2}}.
\end{equation}
(This is clear if $m \geq n$. If $m < n$, then the inequality $m+j-i-nd > 0$ implies that $d = 1, i \in [m],$ and $j \in [n-m+1,n]$, so the entries $U_{ia}$ and $V_{bj}$ are in the first $m$ rows of $U$ and the last $m$ columns of $V$, which are defined in the same way in both the $m \geq n$ and $m < n$ cases.) We saw in the proof of Proposition~\ref{prop:Q} that $\lQ{m+j-i-nd}{i}$ is the determinant of the matrix
\[
\begin{tikzpicture}
\draw (0,0) rectangle (2,2);
\draw (0,1.1) -- (2,1.1);
\draw (0.95,0) -- (0.95,2);
\draw (0.475,0.55) node{$C$};
\draw (0.475,1.55) node{$B$};
\draw (1.475,0.55) node{$A$};
\draw (1.475,1.55) node{$0$};
\draw (-0.85,0.55) node{$n-i+1 \Bigg\{$};
\draw (-0.55,1.55) node{$j-1 \bigg\{$};
\draw (0.475,2.3) node{$\braceabove{0.85}{j}$};
\draw (1.475,2.3) node{$\braceabove{0.95}{n-i}$};
\end{tikzpicture}
\]
where $A,B,C$ are the following bottom-left justified submatrices of $M$ and $M_d$:
\[
A = M_{[i,n],[1,n-i]}, \qquad B = M_{[n-j+2,n],[1,j]}, \qquad C = (M_d)_{[i,n],[1,j]}.
\]
It follows from Lemma~\ref{lem:det stuff}\eqref{item:ABC} that the right-hand side of~\eqref{eq:M entry} is equal to $(M'_d)_{ij}$.
\end{proof}


\section{Pseudo-energies and the \texorpdfstring{$Q$}{Q}-pattern}
\label{sec:Q tab}

Corollary~\ref{cor:ov e} implies that every pseudo-energy is a polynomial in the $Q$-type barred loop elementary symmetric functions, which are the entries of the $m \times m$ matrix $\ov{M} = M(\mb{x}^1, \ldots, \mb{x}^n)$. In \S\ref{sec:unfolded sum of minors}, we give an explicit formula for each pseudo-energy as a positive polynomial in minors of the matrix $\ov{M}$. Since minors of $\ov{M}$ are barred loop skew Schur functions, and the product of loop skew Schur functions is again a loop skew Schur function, this result says that every $\ov{e}$-invariant (unbarred) loop skew Schur function is a sum of barred loop skew Schur functions.

By applying the Lindstr\"om/Gessel--Viennot Lemma to the network $\Gamma_m^{\leq n}$ introduced in \S\ref{sec:GT}, one sees that every nonzero minor of $\ov{M}$ is a positive Laurent polynomial in the entries $z'_{j,i}$ of the gRSK $Q$-pattern. Thus, the above-mentioned result shows that every pseudo-energy can be viewed as a positive function on the $Q$-pattern that is invariant under the birational $S_m$-action on the Gelfand--Tsetlin geometric crystal $\GT_m^{\leq n}$. We discuss these functions and their tropicalizations in \S\ref{sec:back to tableau}.

\subsection{A positive formula for pseudo-energies}
\label{sec:unfolded sum of minors}

Our formula for psuedo-energies comes from a bijection between two classes of lattice paths in a vertex-weighted planar network. Following~\cite{LP13II}, we interpret loop skew Schur functions as generating functions for non-intersecting families of ``highway paths,'' and minors of the matrix $\ov{M}$ as generating functions for non-intersecting families of ``underway paths." In the case of a pseudo-energy, these two types of families are complementary to each other, yielding the desired bijection.

We begin by reviewing several notions from~\cite{LP13II}. Let $\Net_{m,n}$ be the infinite directed network with vertex set $\ZZ \times [0,m+1]$ and edge set
\[
\{(i,j) \rightarrow (i-1,j) \mid i \in \ZZ, j \in [1,m]\} \cup \{(i,j) \rightarrow (i,j+1) \mid i \in \ZZ, j \in [0,m]\}.
\]
The source vertex $(i,0)$ is labeled $i$, the sink vertex $(i,m+1)$ is labeled $i'$, and for $j \in [1,m]$, the interior vertex $(i,j)$ has weight $\lx{i}{j+i-1}$. An example of $\Net_{m,n}$ is shown in Figure~\ref{fig:Network ex}. We view $(i,j)$ as matrix coordinates (i.e., the first coordinate increases downward, the second from left to right), and as a visual reminder of the $n$-periodicity of the vertex weights, we draw a dashed red line between rows $nd$ and $nd+1$ for each $d$.

Each interior vertex of $\Net_{m,n}$ is a simple crossing; that is, there are two incoming edges and two outgoing edges, with the incoming edges adjacent in cyclic order around the vertex. A \defn{highway path} (resp., \defn{underway path} in $\Net_{m,n}$ is a directed path which turns \emph{right} (resp., \emph{left}) at each interior vertex for which the other incoming edge is directly to its \emph{left} (resp., \emph{right}). The three ways a highway path (resp., underway path) can enter and exit an interior vertex are shown below as solid blue (resp., dashed purple) lines:

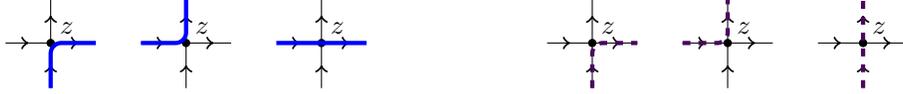
\begin{center}
\begin{tikzpicture}[scale=0.6]
\filldraw (0,0) circle [radius=0.08] node[above right] {$z$};
\draw (-1,0) -- (0,0) -- (1,0);
\draw (0,-1) -- (0,0) -- (0,1);
\draw[thick,->] (-0.501,0) -- (-0.499,0);
\draw[thick,<-] (0.601,0) -- (0.599,0);
\draw[thick,<-] (0,0.601) -- (0,0.599);
\draw[thick,->] (0,-0.501) -- (0,-0.499);
\draw[ultra thick, rounded corners, blue] (0,-1) -- (0,0) -- (1,0);

\begin{scope}[xshift=3cm]
\filldraw (0,0) circle [radius=0.08] node[above right] {$z$};
\draw (-1,0) -- (0,0) -- (1,0);
\draw (0,-1) -- (0,0) -- (0,1);
\draw[thick,->] (-0.501,0) -- (-0.499,0);
\draw[thick,<-] (0.601,0) -- (0.599,0);
\draw[thick,<-] (0,0.601) -- (0,0.599);
\draw[thick,->] (0,-0.501) -- (0,-0.499);
\draw[ultra thick, rounded corners, blue] (-1,0) -- (0,0) -- (0,1);
\end{scope}

\begin{scope}[xshift=6cm]
\filldraw (0,0) circle [radius=0.08] node[above right] {$z$};
\draw (0,-1) -- (0,0) -- (0,1);
\draw[thick,->] (-0.501,0) -- (-0.499,0);
\draw[thick,<-] (0.601,0) -- (0.599,0);
\draw[thick,<-] (0,0.601) -- (0,0.599);
\draw[thick,->] (0,-0.501) -- (0,-0.499);
\draw[ultra thick, rounded corners, blue] (-1,0) -- (0,0) -- (1,0);
\end{scope}

\begin{scope}[xshift=12cm]
\filldraw (0,0) circle [radius=0.08] node[above right] {$z$};
\draw (-1,0) -- (0,0) -- (1,0);
\draw (0,-1) -- (0,0) -- (0,1);
\draw[thick,->] (-0.501,0) -- (-0.499,0);
\draw[thick,<-] (0.601,0) -- (0.599,0);
\draw[thick,<-] (0,0.601) -- (0,0.599);
\draw[thick,->] (0,-0.501) -- (0,-0.499);
\draw[ultra thick, rounded corners, dashed, UQpurple] (0,-1) -- (0,0) -- (1,0);

\begin{scope}[xshift=3cm]
\filldraw (0,0) circle [radius=0.08] node[above right] {$z$};
\draw (-1,0) -- (0,0) -- (1,0);
\draw (0,-1) -- (0,0) -- (0,1);
\draw[thick,->] (-0.501,0) -- (-0.499,0);
\draw[thick,<-] (0.601,0) -- (0.599,0);
\draw[thick,<-] (0,0.601) -- (0,0.599);
\draw[thick,->] (0,-0.501) -- (0,-0.499);
\draw[ultra thick, rounded corners, dashed, UQpurple] (-1,0) -- (0,0) -- (0,1);
\end{scope}

\begin{scope}[xshift=6cm]
\filldraw (0,0) circle [radius=0.08] node[above right] {$z$};
\draw (-1,0) -- (0,0) -- (1,0);
\draw[thick,->] (-0.501,0) -- (-0.499,0);
\draw[thick,<-] (0.601,0) -- (0.599,0);
\draw[thick,<-] (0,0.601) -- (0,0.599);
\draw[thick,->] (0,-0.501) -- (0,-0.499);
\draw[ultra thick, rounded corners, dashed, UQpurple] (0,-1) -- (0,0) -- (0,1);
\end{scope}
\end{scope}
\end{tikzpicture}
\end{center}

The \defn{weight} of a highway or underway path $p$, denoted $\wt(p)$, is the product of the weights of the interior vertices that $p$ passes through \emph{without turning}. In each of the sets of three paths above, only the right-most path picks up the weight $z$. For a family $\mc{F} = (p_1, \ldots, p_a)$ of highway or underway paths, the weight $\wt(\mc{F})$ is the product of the weights of $p_1, \ldots, p_a$. We say that $\mc{F}$ is \defn{non-intersecting} if no two of its paths share an edge (non-intersecting paths are allowed to ``kiss'' at a vertex).

Given a non-intersecting family of highway paths $\mc{F}$, define $\ov{\mc{F}}$ to be the set of edges not in $\mc{F}$. If we draw the edges of $\mc{F}$ and $\ov{\mc{F}}$ simultaneously, each interior vertex has one of the following five configurations:
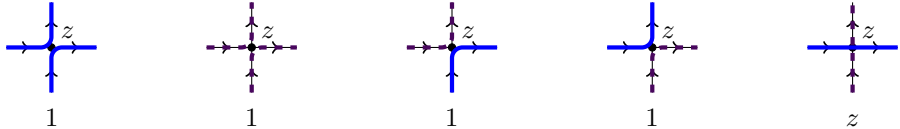
\[
\begin{array}{c@{\qquad\qquad}c@{\qquad\qquad}c@{\qquad\qquad}c@{\qquad\qquad}c}
\begin{tikzpicture}[scale=0.6]
\filldraw (0,0) circle [radius=0.08] node[above right] {$z$};
\draw (-1,0) -- (0,0) -- (1,0);
\draw (0,-1) -- (0,0) -- (0,1);
\draw[thick,->] (-0.501,0) -- (-0.499,0);
\draw[thick,<-] (0.601,0) -- (0.599,0);
\draw[thick,<-] (0,0.601) -- (0,0.599);
\draw[thick,->] (0,-0.501) -- (0,-0.499);
\draw[blue,ultra thick,rounded corners] (-1,0) -- (0,0) -- (0,1);
\draw[blue,ultra thick,rounded corners] (0,-1) -- (0,0) -- (1,0);
\end{tikzpicture}
&
\begin{tikzpicture}[scale=0.6]
\filldraw (0,0) circle [radius=0.08] node[above right] {$z$};
\draw (-1,0) -- (0,0) -- (1,0);
\draw (0,-1) -- (0,0) -- (0,1);
\draw[thick,->] (-0.501,0) -- (-0.499,0);
\draw[thick,<-] (0.601,0) -- (0.599,0);
\draw[thick,<-] (0,0.601) -- (0,0.599);
\draw[thick,->] (0,-0.501) -- (0,-0.499);
\draw[UQpurple,dashed,ultra thick,rounded corners] (-1,0) -- (0,0) -- (0,1);
\draw[UQpurple,dashed,ultra thick,rounded corners] (0,-1) -- (0,0) -- (1,0);
\end{tikzpicture}
&
\begin{tikzpicture}[scale=0.6]
\filldraw (0,0) circle [radius=0.08] node[above right] {$z$};
\draw (-1,0) -- (0,0) -- (1,0);
\draw (0,-1) -- (0,0) -- (0,1);
\draw[thick,->] (-0.501,0) -- (-0.499,0);
\draw[thick,<-] (0.601,0) -- (0.599,0);
\draw[thick,<-] (0,0.601) -- (0,0.599);
\draw[thick,->] (0,-0.501) -- (0,-0.499);
\draw[UQpurple,dashed,ultra thick,rounded corners] (-1,0) -- (0,0) -- (0,1);
\draw[blue,ultra thick,rounded corners] (0,-1) -- (0,0) -- (1,0);
\end{tikzpicture}
&
\begin{tikzpicture}[scale=0.6]
\filldraw (0,0) circle [radius=0.08] node[above right] {$z$};
\draw (-1,0) -- (0,0) -- (1,0);
\draw (0,-1) -- (0,0) -- (0,1);
\draw[thick,->] (-0.501,0) -- (-0.499,0);
\draw[thick,<-] (0.601,0) -- (0.599,0);
\draw[thick,<-] (0,0.601) -- (0,0.599);
\draw[thick,->] (0,-0.501) -- (0,-0.499);
\draw[blue,ultra thick,rounded corners] (-1,0) -- (0,0) -- (0,1);
\draw[UQpurple,dashed,ultra thick,rounded corners] (0,-1) -- (0,0) -- (1,0);
\end{tikzpicture}
&
\begin{tikzpicture}[scale=0.6]
\filldraw (0,0) circle [radius=0.08] node[above right] {$z$};
\draw (-1,0) -- (0,0) -- (1,0);
\draw (0,-1) -- (0,0) -- (0,1);
\draw[thick,->] (-0.501,0) -- (-0.499,0);
\draw[thick,<-] (0.601,0) -- (0.599,0);
\draw[thick,<-] (0,0.601) -- (0,0.599);
\draw[thick,->] (0,-0.501) -- (0,-0.499);
\draw[blue,ultra thick,rounded corners] (-1,0) -- (0,0) -- (1,0);
\draw[UQpurple,dashed,ultra thick,rounded corners] (0,-1) -- (0,0) -- (0,1);
\end{tikzpicture}
\\
1 & 1 & 1 & 1 & z
\end{array}
\]
It is clear that the edges in $\ov{\mc{F}}$ determine a unique non-intersecting family of underway paths. Moreover, each of the five configurations contributes the same weight (shown beneath the configuration) to both $\mc{F}$ and $\ov{\mc{F}}$. Thus, the correspondence $\mc{F} \longleftrightarrow \ov{\mc{F}}$ is a weight-preserving bijection between non-intersecting families of highway paths and non-intersecting families of underway paths.

\begin{remark}
If we view highway (resp., underway) paths as the positions of particles (resp., holes) in a five-vertex model, the correspondence just described is the now classical particle-hole duality.
\end{remark}

We now explain the connection between highway paths in $\Lambda_{m,n}$ and loop skew Schur functions. Associate to $\Net_{m,n}$ the $\ZZ \times \ZZ$ matrix $A(\Net_{m,n})$ with entries
\[
A(\Net_{m,n})_{ij} = \sum_{p \colon i \rightarrow j'} \wt(p),
\]
where the sum is over highway paths in $\Net_{m,n}$ from source $i$ to sink $j'$.

\begin{ex}
In the network $\Net_{5,3}$, there are ten highway paths from source $5$ to sink $2'$. These paths have weights $\lx{a}{2}\lx{b}{3}$ for $1 \leq a < b \leq 5$; the path of weight $\lx{2}{2}\lx{4}{3}$ is shown in Figure~\ref{fig:Network ex}. Thus, we have
\[
A(\Net_{5,3})_{5,2} = \lE{2}{2}(\mb{x}_1,\mb{x}_2,\mb{x}_3,\mb{x}_4,\mb{x}_5).
\]
\end{ex}

Recall the $n$-periodic matrix $\widetilde{M} = \widetilde{M}(\mb{x}_1, \ldots, \mb{x}_m)$ introduced in~\S \ref{sec:LSym matrix}. The $(i,j)$-entry of $\widetilde{M}$ is the loop elementary symmetric function $E_{m+j-i}^{(i)}(\mb{x}_1, \ldots, \mb{x}_m)$, and the minors of $\widetilde{M}$ are the loop skew Schur functions.

\begin{lemma}
\label{lem:A=M}
$A(\Net_{m,n})$ is equal to $\widetilde{M}$.
\end{lemma}

\begin{proof}
Suppose $p$ is a highway path from source $i$ to sink $j'$. When $p$ enters an interior vertex from the bottom, it must turn right, and when $p$ enters an interior vertex from the left, it may either turn left or pass straight through. To get from $i$ to $j'$, $p$ must enter $m$ interior vertices from the left, and turn left at exactly $i-j$ of these vertices. Thus, $p$ passes straight through $m+j-i$ interior vertices, so the weight of $p$ is a monomial of degree $m+j-i$. It is easy to see that the possible weights of such a path are in bijection with the monomials of $\lE{m+j-i}{i}(\mb{x}_1, \ldots, \mb{x}_m)$.
\end{proof}

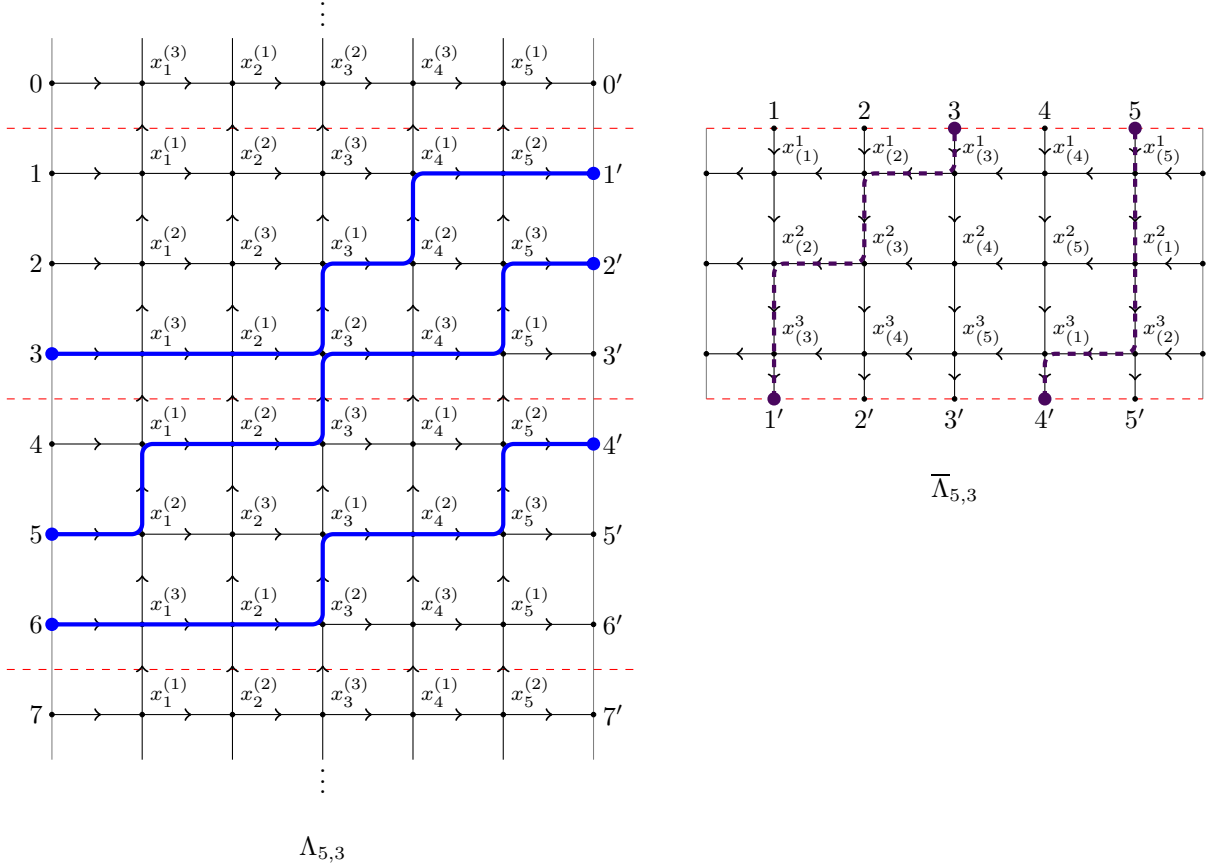
\begin{figure}
\begin{tikzpicture}[scale=0.6]
\pgfmathtruncatemacro{\m}{5};
\pgfmathtruncatemacro{\n}{3};
\pgfmathtruncatemacro{\toprow}{0};
\pgfmathtruncatemacro{\bottomrow}{2*\n+1};
\pgfmathtruncatemacro{\numrows}{\bottomrow -\toprow+1};

\draw (\m+1,-2) node {$\Net_{5,3}$};

\draw[gray] (0,2*\numrows) -- (0,0);
\draw[gray] (2*\m+2,2*\numrows) -- (2*\m+2,0);
\foreach \y in {1,1+2*\n,1+4*\n} {
\draw[dashed,red] (-1,2*\y) -- (2*\m+3,2*\y);}
\draw (\m+1,-0.3) node{$\vdots$};
\draw (\m+1,2*\numrows+0.7) node{$\vdots$};

\foreach \y in {1,...,\numrows} {
\pgfmathtruncatemacro{\row}{\bottomrow-\y+1};
\filldraw (0,2*\y-1) circle [radius = .05] node[left] {$\row$};
\filldraw (2*\m+2,2*\y-1) circle [radius = .05] node[right] {$\row'$};
\draw (2*\m+2,2*\y-1) -- (0,2*\y-1);
\draw[thick,<-] (2*\m+1.09,2*\y-1) -- (2*\m+1.07,2*\y-1);}

\foreach \x in {1,...,\m} {
\draw (2*\x,0) -- (2*\x,2*\numrows);
\foreach \y in {1,...,\numrows} {
\draw[thick,<-] (2*\x-0.91,2*\y-1) -- (2*\x-0.93,2*\y-1);
\ifnum \y < \numrows \draw[thick,<-] (2*\x,2*\y+0.09) -- (2*\x,2*\y+0.07); \fi
\filldraw (2*\x,2*\y-1) circle [radius = .05];
\pgfmathtruncatemacro{\var}{\x};
\pgfmathtruncatemacro{\color}{mod(\x-\y+\bottomrow,\n)};
\ifnum \color < 1 \pgfmathtruncatemacro{\color}{\color+\n}; \fi
\footnotesize
\draw (2*\x,2*\y-1) node[above right]{$\lx{\var}{\color}$};}}
\normalsize

\foreach \x/\y in {5/3,9/5,    1/5,5/7,9/9,    5/9,7/11} {\draw[ultra thick,rounded corners, blue] (\x,\y) -- (\x+1,\y) -- (\x+1,\y+1) -- (\x+1,\y+2) -- (\x+2,\y+2);}
\foreach \x/\y in {1/3,3/3,7/5,    3/7,7/9,    1/9,3/9,9/13} {\draw[ultra thick, blue] (\x,\y) -- (\x+2,\y);}
\foreach \x/\y in {0/3,    0/5,    0/9} {\filldraw[ultra thick, blue] (\x,\y) circle [radius = 0.1] -- (\x+1,\y);}
\foreach \x/\y in {11/7,    11/11,    11/13} {\filldraw[ultra thick, blue] (\x,\y) -- (\x+1,\y) circle [radius = 0.1];}

\begin{scope}[yshift = 8cm,xshift=14cm]
\draw (\m+1,-2) node {$\ov{\Net}_{5,3}$};

\draw[gray] (1/2,0) -- (1/2,2*\n);
\draw[gray] (2*\m+3/2,0) -- (2*\m+3/2,2*\n);
\draw[dashed,red] (1/2,0) -- (2*\m+3/2,0);
\draw[dashed,red] (1/2,2*\n) -- (2*\m+3/2,2*\n);

\foreach \y in {1,...,\n} {
\draw (2*\m+3/2,2*\y-1) -- (1/2,2*\y-1);
\filldraw (1/2,2*\y-1) circle [radius = .05];
\filldraw (2*\m+3/2,2*\y-1) circle [radius = .05];
\draw[thick,->] (2*\m+3/4-0.07,2*\y-1) -- (2*\m+3/4-0.09,2*\y-1);}

\foreach \x in {1,...,\m} {
\filldraw (2*\x,2*\n) circle [radius = .05] node[above] {$\x$};
\filldraw (2*\x,0) circle [radius = .05] node[below] {$\x'$};
\draw[thick,->] (2*\x,2*\n-0.5-0.07) -- (2*\x,2*\n-0.5-0.09);
\draw[thick,->] (2*\x,0.5-0.07) -- (2*\x,0.5-0.09);
\draw (2*\x,0) -- (2*\x,2*\n);
\foreach \y in {1,...,\n} {
\ifnum \x > 1 \draw[thick,->] (2*\x-1.07,2*\y-1) -- (2*\x-1.09,2*\y-1); \fi
\ifnum \x = 1 \draw[thick,->] (2*\x-3/4-0.07,2*\y-1) -- (2*\x-3/4-0.09,2*\y-1); \fi
\ifnum \y < \n \draw[thick,->] (2*\x,2*\y-0.07) -- (2*\x,2*\y-0.09); \fi
\filldraw (2*\x,2*\y-1) circle [radius = .05];
\pgfmathtruncatemacro{\var}{\n-\y+1};
\pgfmathtruncatemacro{\color}{mod(\n-\y+\x,\m)};
\ifnum \color < 1 \pgfmathtruncatemacro{\color}{\color+\m}; \fi
\footnotesize
\draw (2*\x,2*\y-1) node[above right]{$\bx{\color}{\var}$};}}
\normalsize

\foreach \x\y in {2/2,4/4,    8/0} {\draw[ultra thick, UQpurple, rounded corners, dashed] (\x,\y) -- (\x,\y+1) -- (\x+2,\y+1) -- (\x+2,\y+2);}
\foreach \x\y in {2/0,    10/2,10/4} {\draw[ultra thick, UQpurple, dashed] (\x,\y) -- (\x,\y+2);}
\foreach \x\y in {2/0,8/0,10/6,6/6} {\filldraw[ultra thick, UQpurple] (\x,\y) circle [radius = 0.1];}
\end{scope}

\end{tikzpicture}
\caption{(Left) The network $\Net_{5,3}$, which gives rise to the $3$-periodic matrix $\widetilde{M} = \widetilde{M}(\mb{x}_1, \ldots, \mb{x}_5)$. A non-intersecting family of highway paths is shown in blue. (Right) The network $\ov{\Net}_{5,3}$, which gives rise to the $5 \times 5$ matrix $\ov{M} = M(\mb{x}^1, \mb{x}^2, \mb{x}^3)$. The dashed purple lines show a non-intersecting family of underway paths.}
\label{fig:Network ex}
\end{figure}

For the next result, we assume the reader is familiar with the Lindstr\"om/Gessel--Viennot Lemma \cite{Lindstrom, GesselViennot}, which expresses minors of matrices associated to planar networks as weighted sums of non-intersecting families of paths in the network.

\begin{lemma}
\label{lem:highway Lind}
For $k$-element subsets $I,J \subset \ZZ$, we have
\[
\Delta_{I,J}(\widetilde{M}) = \sum_{\mc{F} \in N^h_{I,J}} \wt(\mc{F}),
\]
where $N^h_{I,J}$ is the set of non-intersecting families of highway paths in $\Net_{m,n}$ from sources $I$ to sinks $J$.
\end{lemma}

\begin{proof}
As in~\cite{LP13II}, convert $\Net_{m,n}$ into an edge-weighted network $\Gamma$ by replacing the configuration around each interior vertex as follows:
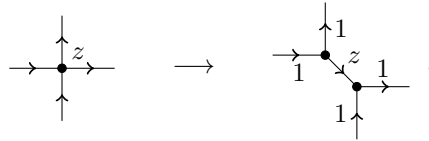
\begin{center}
\begin{tikzpicture}[scale=0.7]
\filldraw (0,0) circle [radius=0.08] node[above right] {$z$};
\draw (-1,0) -- (0,0) -- (1,0);
\draw (0,-1) -- (0,0) -- (0,1);
\draw[thick,->] (-0.501,0) -- (-0.499,0);
\draw[thick,<-] (0.601,0) -- (0.599,0);
\draw[thick,<-] (0,0.601) -- (0,0.599);
\draw[thick,->] (0,-0.501) -- (0,-0.499);

\draw (2.5,0) node {$\longrightarrow$};
\draw (7,0) node{.};

\begin{scope}[xshift=5cm,yshift=0.25cm]
\filldraw (0,0) circle [radius=0.08];
\filldraw (0.6,-0.6) circle [radius=0.08];
\draw (-1,0) -- node[below] {$1$} (0,0) -- node[right] {$1$} (0,1);
\draw (0.6,-1.6) -- node[left] {$1$} (0.6,-0.6) -- node[above] {$1$} (1.6,-0.6);
\draw (0,0) -- (0.6,-0.6);
\draw (0.55,-0.05) node {$z$};
\draw[thick,->] (-0.501,0) -- (-0.499,0);
\draw[thick,<-] (1.201,-0.6) -- (1.199,-0.6);
\draw[thick,<-] (0,0.601) -- (0,0.599);
\draw[thick,->] (0.6,-1.101) -- (0.6,-1.099);
\draw[thick,->] (0.349,-0.349) -- (0.351,-0.351);
\end{scope}
\end{tikzpicture}
\end{center}
It is clear that there is a weight-preserving correspondence between non-intersecting families of highway paths in $\Net_{m,n}$ and non-intersecting families of directed paths in $\Gamma$, so the result follows from the Lindstr\"om/Gessel--Viennot Lemma applied to $\Gamma$.
\end{proof}

Now consider a pseudo-energy $s_{\la/\mu}^{(r)}$. As explained in Remark~\ref{rem:no empty columns}, we may assume that $\mu_i' < \la_i'$ for $i = 1, \ldots, \ell(\la')$. Let $I = I(\mu,r), J = J(\la,r)$ be the subsets associated to $(\la/\mu)^{(r)}$ by~\eqref{eq:subsets JT matrix version}, translated by a multiple of $n$ so that they consist of positive integers. By Lemma~\ref{lem:initial final condition}, there exist $d \geq 0$ and $a_0, \ldots, a_d, b_0, \ldots, b_d \in [0,n]$ such that
\begin{equation}
\label{eq_def_I_J}
\begin{matrix}
I = & [n+1-a_0,n] & \cup & [2n+1-a_1,2n] & \cup & \cdots & \cup & [(d+1)n+1-a_d, (d+1)n], \\
J = & [1,b_0] & \cup & [n+1,n+b_1] & \cup & \cdots & \cup & [dn + 1, dn+b_d].
\end{matrix}
\end{equation}
Note that $\sum a_i = \abs{I} = \abs{J} = \sum b_i$. Proposition~\ref{prop:JT matrix version} and Lemma~\ref{lem:highway Lind} imply that $s_{\la/\mu}^{(r)}$ is the generating function for non-intersecting families of highway paths in $\Lambda_{m,n}$ from sources $I$ to sinks $J$.

Let $\mc{F}$ be a non-intersecting family of highway paths from $I$ to $J$, and let $\ov{\mc{F}}$ be the complementary family of underway paths. Let $\Upsilon_k \subseteq \Net_{m,n}$ denote the \defn{band} between the dashed red lines at heights $kn+1/2$ and $(k+1)n+1/2$. For $k \in [0,d]$, $a_k$ of the sources in $\Upsilon_k$ are in $I$, and $b_k$ of the sinks in $\Upsilon_k$ are in $J$. This means there are $n-a_k$ underway paths in $\ov{\mc{F}}$ starting at the left boundary of $\Upsilon_k$, and $n-b_k$ underway paths ending at the right boundary of $\Upsilon_k$. Since these paths cannot pass horizontally through an interior vertex, they are constrained to move by alternating turns within $\Upsilon_k$, so the portions of these paths in $\Upsilon_k$ do not contribute to the weight of $\ov{\mc{F}}$. Hence, we can ignore the edges of $\ov{\mc{F}}$ contained within the triangular regions of $\Upsilon_k$ above (resp., below) the slope $1$ lines that intersect source $(k+1)n+1-a_k$ (resp., sink $(kn+b_k)'$). We call these triangular regions \defn{fixed}. (The edges of $\ov{\mc{F}}$ outside of $\Upsilon_0 \cup \cdots \cup \Upsilon_d$ also do not contribute to the weight, so we ignore these as well.) The remaining edges of $\ov{\mc{F}}$ form families of non-intersecting paths between the bottom and top boundaries of each $\Upsilon_k$ (see Figure~\ref{fig:unfolded complementary paths} for an example). These families contribute to minors of $\ov{M}$. This observation leads to the main result of this section.

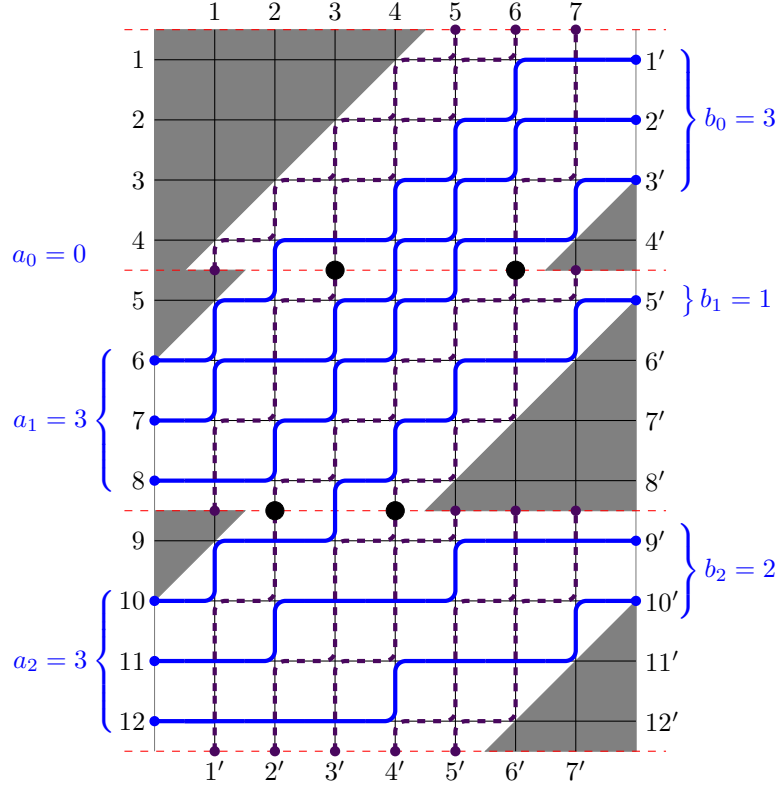
\begin{figure}
\[
\begin{tikzpicture}[scale=0.4]
\pgfmathtruncatemacro{\m}{7};
\pgfmathtruncatemacro{\n}{4};
\pgfmathtruncatemacro{\toprow}{1};
\pgfmathtruncatemacro{\bottomrow}{3*\n};
\pgfmathtruncatemacro{\numrows}{\bottomrow -\toprow+1};

\filldraw[gray] (0,15) -- (9,24) -- (0,24) -- (0,15);
\filldraw[gray] (0,13) -- (3,16) -- (0,16) -- (0,13);
\filldraw[gray] (0,5) -- (3,8) -- (0,8) -- (0,5);
\filldraw[gray] (16,5) -- (11,0) -- (16,0) -- (16,5);
\filldraw[gray] (16,15) -- (9,8) -- (16,8) -- (16,15);
\filldraw[gray] (16,19) -- (13,16) -- (16,16) -- (16,19);

\draw[gray] (0,2*\numrows) -- (0,0);
\draw[gray] (2*\m+2,2*\numrows) -- (2*\m+2,0);
\foreach \y in {0,4,8,12} {\draw[red,dashed] (-1,2*\y) -- (2*\m+3,2*\y);}

\foreach \y in {1,...,\numrows} {
\pgfmathtruncatemacro{\row}{\bottomrow-\y+1};
\draw (0,2*\y-1) node[left] {$\row$} -- (2*\m+2,2*\y-1) node[right] {$\row'$};}

\foreach \x in {1,...,\m} {
\draw (2*\x,0) -- (2*\x,2*\numrows);
\draw (2*\x,2*\numrows) node[above] {$\x$};
\draw (2*\x,0) node[below] {$\x'$};}

\foreach \x/\y in {7/1,13/3,    3/3,9/5,    1/5,5/7,7/9,9/11,13/13,    3/9,5/11,7/13,9/15,13/17,    1/11,5/13,7/15,9/17,11/19,    1/13,3/15,7/17,9/19,11/21} {
	\draw[ultra thick,rounded corners, blue] (\x,\y) -- (\x+1,\y) -- (\x+1,\y+1) -- (\x+1,\y+2) -- (\x+2,\y+2);}
\foreach \x/\y in {1/1,3/1,5/1,9/3,11/3,    1/3,5/5,7/5,11/7,13/7,    3/7,11/13,    1/9,11/17,    3/13,13/21,    5/17,13/23,} {\draw[ultra thick, blue] (\x,\y) -- (\x+2,\y);}
\foreach \x/\y in {0/1,    0/3,    0/5,    0/9,    0/11,    0/13} {\filldraw[ultra thick, blue] (\x,\y) circle [radius = 0.1] -- (\x+1,\y);}
\foreach \x/\y in {15/5,    15/7,    15/15,    15/19,    15/21,    15/23} {\filldraw[ultra thick, blue] (\x,\y) -- (\x+1,\y) circle [radius = 0.1];}

\foreach \x/\y in {2/10,4/14,6/18,8/20,10/22,    2/16,4/18,6/20,8/22,    2/4,4/8,6/10,8/12,10/14,12/18,    4/2,6/6,8/8,10/10,12/14,    6/2,8/6,    8/0,10/4,    10/0,12/4} {
	\draw[ultra thick,rounded corners, dashed,UQpurple] (\x,\y) -- (\x,\y+1) -- (\x+1,\y+1) -- (\x+2,\y+1) -- (\x+2,\y+2);}
\foreach \x/\y in {2/8,4/12,6/16,    2/0,2/2,4/6,    4/0,6/4,12/12,12/16,14/20,14/22,    6/0,8/4,    10/2,12/6,    12/2,14/6} {\draw[ultra thick, dashed,UQpurple] (\x,\y) -- (\x,\y+2);}
\foreach \x/\y in {2/0,4/0,6/0,8/0,10/0,    2/8,4/8,8/8,10/8,12/8,14/8,    2/16,6/16,12/16,14/16,    10/24,12/24,14/24} {\filldraw[ultra thick, UQpurple] (\x,\y) circle [radius = 0.1];}

\foreach \x\y in {4/8,8/8,6/16,12/16} {\filldraw (\x,\y) circle [radius=0.3];}

\draw[blue] (-3.5,16.5) node {$a_0 = 0$};
\draw[blue] (-3,11) node {$a_1 = 3 \, \Vast\{$};
\draw[blue] (-3,3) node {$a_2 = 3 \, \Vast\{$};
\draw[blue] (19,21) node {$\Vast\} \, b_0 = 3$};
\draw[blue] (19,15) node {$\big\} \, b_1 = 1$};
\draw[blue] (19,6) node {$\vast\} \, b_2 = 2$};
\end{tikzpicture}
\]
\caption{An example of the bijection between complementary families of lattice paths used to prove Theorem~\ref{thm:unfolded sum of minors}. The parameters for this example are $m = 7, n = 4, d = 2$, and the indicated $a_i,b_i$. The relevant bands $\Upsilon_0, \Upsilon_1, \Upsilon_2$ are shown, with fixed regions shaded in gray. The highway paths are drawn as solid blue lines, and the portions of the underway paths outside the fixed regions as dashed purple lines. The sets $X_1 = \{3,6\}, X_2 = \{2,4\}$ are represented by the large black vertices.}
\label{fig:unfolded complementary paths}
\end{figure}

\begin{thm}
\label{thm:unfolded sum of minors}
Suppose $d \geq 0$, and $a_0, \ldots, a_d, b_0, \ldots, b_d \in [0,n]$ satisfy
\[
\sum_{k=0}^d a_k = \sum_{k=0}^d b_k
\qquad \text{ and } \qquad
a_k, b_k \geq n-m
\]
for $k = 0, \ldots, d$. Define $I,J$ by~\eqref{eq_def_I_J}. Then
\begin{equation}
\label{eq:unfolded sum of minors}
\Delta_{I,J}(\widetilde{M}) = \sum_{\substack{X = (X_1, \ldots, X_d) \colon \\ X_k \subseteq [n-a_k+1,m-(n-b_{k-1})]}}
  \prod_{k=0}^d \Delta_{X_k \cup [m-(n-b_{k-1})+1,m], [1,n-a_{k+1}] \cup X_{k+1}}(\ov{M}),
\end{equation}
where $a_{d+1}, b_{-1} = n$, $X_0 = [n-a_0+1,m]$, and $X_{d+1} = [1,m-(n-b_d)]$.
\end{thm}

Let us decipher the theorem statement briefly. For $k = 1, \ldots, d$, the set $X_k$ represents the places where an underway path has chosen to cross over from $\Upsilon_k$ to $\Upsilon_{k-1}$. The interval $[1,n-a_k]$ corresponds to the vertices at the top of the fixed region in the northwest corner of $\Upsilon_k$, and the interval $[m-(n-b_{k-1})+1,m]$ corresponds to the vertices at the bottom of the fixed region in the southeast corner of $\Upsilon_{k-1}$. By construction, $X_k$ is disjoint from these intervals. Note that the $X_k$ must satisfy $\abs{X_k} + (n - b_{k-1}) = (n - a_{k+1}) + \abs{X_{k+1}}$ for all $0 \leq k \leq d$; thus, the sizes of the $X_k$ are determined by the numbers $a_i$ and $b_i$.

\begin{remark}
If $a_k < n-m$ (resp., $b_k < n-m$) for some $k \in [0,d]$, then the interval $[1,n-a_k]$ (resp., $[m-(n-b_k)+1,m]$) is not contained in $[1,m]$, so the right-hand side of~\eqref{eq:unfolded sum of minors} does not make sense. One can show, however, that every non-zero pseudo-energy can be expressed as $\Delta_{I,J}(\widetilde{M})$ for some choice of parameters $a_k, b_k$ which satisfy $a_k,b_k \geq n-m$ for all $k$.
\end{remark}

\begin{proof}[Proof of Theorem~\ref{thm:unfolded sum of minors}]
Let $\ov{\Net}_{m,n}$ be the network obtained from $\Net_{m,n}$ by restricting to a single band $\Upsilon_k$, reversing all arrows, relabeling the vertex weights using the identifications $\lx{i}{j+i-1} = x_i^j = \bx{j+i-1}{j}$, removing the source and sink labels, and placing $m$ new source vertices (resp., sink vertices) labeled $1, \ldots, m$ (resp., $1', \ldots, m'$) along the top (resp., bottom) boundary, as in Figure~\ref{fig:Network ex}. Arguing as in the proofs of Lemmas~\ref{lem:A=M} and~\ref{lem:highway Lind}, one shows that for $A,B \subset [m]$,
\begin{equation}
\label{eq:underway Lind}
\Delta_{A,B}(\ov{M}) = \sum_{\ov{\mc{F}} \in N^u_{A,B}} \wt(\ov{\mc{F}}),
\end{equation}
where $N^u_{A,B}$ is the set of non-intersecting families of underway paths in $\ov{\Net}_{m,n}$ from $A$ to $B$.

Given an underway path in $\Net_{m,n}$ from the bottom to the top edge of $\Upsilon_k$, one obtains an underway path in $\ov{\Net}_{m,n}$ of the same weight by reversing all the arrows. The theorem now follows from \eqref{eq:underway Lind} and the above discussion.
\end{proof}

\begin{ex}
\label{ex:unfolded sum of minors}
\
\begin{enumerate}
\item[(a)]
When $d = 0$, Theorem~\ref{thm:unfolded sum of minors} says that
\[
\Delta_{[n+1-a,n], [1,a]}(\widetilde{M}) = \Delta_{[n+1-a,m],[1,m-n+a]}(\ov{M}).
\]
In this case, we already know that both of these minors are equal to the shape invariant $S_{n+1-a}$.
\item[(b)]
When $m = n = 3$, we have
\[
\ov{M} =
\begin{pmatrix}
\bE{3}{1} & 0 & 0 \\
\bE{2}{2} & \bE{3}{2} & 0 \\
\bE{1}{3} & \bE{2}{3} & \bE{3}{3}
\end{pmatrix},
\]
\ytableausetup{smalltableaux}
and Theorem~\ref{thm:unfolded sum of minors} gives the following formulas for the $Q$-invariants:
\begin{align*}
\lQ{1}{1} = s_{\ytableaushort{\empty\empty\empty\empty, \empty\empty}}^{(3)} &= \Delta_{3456,1245}(\widetilde{M}) = \sum_{x \in \{1,2\}} \Delta_{3,x}(\ov{M}) \Delta_{x3,12}(\ov{M}) \\
&= \Delta_{3,1}(\ov{M}) \Delta_{13,12}(\ov{M}) + \Delta_{3,2}(\ov{M}) \Delta_{23,12}(\ov{M}) \\
&= \bs{\ytableaushort{\empty}}{3} \bs{\ytableaushort{\none\empty, \empty\empty, \empty\empty}}{2} + \bs{\ytableaushort{\empty, \empty}}{3} \bs{\ytableaushort{\empty\empty, \empty\empty}}{3} \,,
\allowdisplaybreaks \\
\lQ{1}{2} = s_{\ytableaushort{\none\none\empty\empty,\empty\empty\empty\empty}}^{(2)} &= \Delta_{2356,1234}(\widetilde{M}) = \sum_{x \in \{2,3\}} \Delta_{23,1x}(\ov{M}) \Delta_{x,1}(\ov{M}) \\
&= \Delta_{23,12}(\ov{M}) \Delta_{2,1}(\ov{M}) + \Delta_{23,13}(\ov{M}) \Delta_{3,1}(\ov{M}) \\
&= \bs{\ytableaushort{\empty\empty, \empty\empty}}{3} \bs{\ytableaushort{\empty, \empty}}{2} + \bs{\ytableaushort{\empty\empty, \empty\empty, \empty}}{3} \bs{\ytableaushort{\empty}}{3} \,,
\allowdisplaybreaks \\
\lQ{2}{1} = s_{\ytableaushort{\empty\empty\empty\empty\empty,\empty\empty\empty\empty\empty}}^{(3)} & = \Delta_{23456,12345}(\widetilde{M}) = \sum_{X \in \binom{[3]}{2}} \Delta_{23,X}(\ov{M}) \Delta_{X,12}(\ov{M}) \\
&= \Delta_{23,12}(\ov{M}) \Delta_{12,12}(\ov{M}) + \Delta_{23,13}(\ov{M}) \Delta_{13,12}(\ov{M}) + \Delta_{23,23}(\ov{M}) \Delta_{23,12}(\ov{M}) \\
&= \bs{\ytableaushort{\empty\empty, \empty\empty}}{3} \bs{\ytableaushort{\empty\empty, \empty\empty, \empty\empty}}{2} + \bs{\ytableaushort{\empty\empty, \empty\empty, \empty}}{3} \bs{\ytableaushort{\none\empty, \empty\empty, \empty\empty}}{2} + \bs{\ytableaushort{\empty\empty, \empty\empty, \empty\empty}}{3} \bs{\ytableaushort{\empty\empty, \empty\empty}}{3} \,.
\end{align*}
\end{enumerate}
\end{ex}

Additional examples and consequences of Theorem~\ref{thm:unfolded sum of minors} appear in \S\ref{sec:back to tableau}.

\subsection{\texorpdfstring{$Q$}{Q}-invariants and the \texorpdfstring{$Q$}{Q}-pattern}
\label{sec:back to tableau}

Let $(z'_{j,i})_{1 \leq j \leq n, j \leq i \leq m}$ denote the entries of the gRSK $Q$-pattern. By definition, the $z'_{j,i}$ are ratios of flag minors of $\ov{M}$.

\begin{lemma}
\label{lem:back to tableau}
Every pseduo-energy is a positive Laurent polynomial in the $z'_{j,i}$. In particular, the $Q$-invariants (and reduced $Q$-invariants) are positive Laurent polynomials in the $z'_{j,i}$.
\end{lemma}

\begin{proof}
Theorem~\ref{thm:unfolded sum of minors} shows that every pseudo-energy is a positive polynomial in the minors of $\ov{M}$. Theorem~\ref{thm:gRSK props}(3) says that $Q = \Psi_m^{\leq n}(\ov{M})$, so $\ov{M} = \Phi_m^{\leq n}(Q)$ by Lemma~\ref{lem:Phi inverse}. It follows from the discussion in \S\ref{sec:GT} that $\ov{M}$ is the matrix associated to a planar network $\Gamma_m^{\leq n}$, whose edge weights are ratios of the $z'_{j,i}$. By the Lindstr\"om/Gessel--Viennot Lemma, each nonzero minor of $\ov{M}$ is a positive polynomial in the edge weights, and thus a Laurent polynomial in the $z'_{j,i}$. The statement about reduced $Q$-invariants follows from~\eqref{eq:z shape}, which says that the shape invariants are products of $z'_{j,m}$, the entries of the shape of the $Q$-pattern.
\end{proof}

\begin{remark}
The flag minors form a cluster in the coordinate ring of the space of lower triangular matrices, so the fact that all nonzero minors of a lower triangular matrix are positive Laurent polynomials in the flag minors is an example of the (positive) Laurent phenomenon in the theory of cluster algebras~\cite{FZ}. In fact, this is one of the examples that led Fomin and Zelevinsky to introduce the notion of cluster algebras.
\end{remark}

The planar network $\Gamma_m^{\leq n}$ provides a very convenient computational method for obtaining explicit formulas for reduced $Q$-invariants in terms of the $z'_{j,i}$.

\begin{figure}
\begin{center}
\begin{tikzpicture}

\foreach \a/\b/\c/\d in {0/1/1/0, 0/2/2/0, 0/3/3/0, 1/2/1/0, 2/1/2/0} {\draw (\a,\b) -- (\c,\d);}

\foreach \a/\b in {0/1,0/2,0/3} {
\filldraw (\a,\b) circle[radius=.04cm] node[left]{$\b$};
\filldraw (\b,\a) circle[radius=.04cm] node[below]{$\b'$};}

\foreach \a in {1,2,3} {\draw (0.6,\a-0.2) node{$z'_{\a\a}$};}
\foreach \a/\b in {1/2,2/3} {\draw (1.6,\a-0.1) node{$\frac{z'_{\a\b}}{z'_{\a\a}}$};}
\foreach \a/\b/\c in {1/2/3} {\draw (2.6,\a-0.1) node{$\frac{z'_{\a\c}}{z'_{\a\b}}$};}

\end{tikzpicture}
\end{center}
\caption{The network $\Gamma_3$.}
\label{fig:gamma 3}
\end{figure}
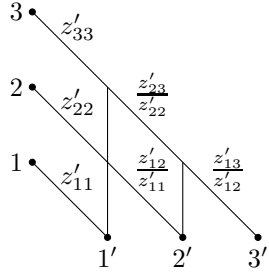

\begin{ex}
\label{ex:Q dec}
Let $m = n = 3$. The network $\Gamma_3 = \Gamma_3^{\leq 3}$ is shown in Figure~\ref{fig:gamma 3}. By Lindstr\"om/Gessel--Viennot and Example~\ref{ex:unfolded sum of minors}(b), we have
\begin{align*}
\lrQ{1}{1} = \dfrac{\lQ{1}{1}}{S_2S_3} &= \dfrac{\Delta_{13,12}(\ov{M})}{\Delta_{23,12}(\ov{M})} + \dfrac{\Delta_{3,2}(\ov{M})}{\Delta_{3,1}(\ov{M})} \bigskip\bigskip \\
&= \dfrac{z'_{11}z'_{33} \left(\frac{z'_{12}}{z'_{11}} + \frac{z'_{23}}{z'_{22}}\right)}{z'_{23}z'_{33}} + \dfrac{z'_{33} \left(\frac{z'_{12}}{z'_{11}} + \frac{z'_{23}}{z'_{22}}\right)}{z'_{33}} = \dfrac{z'_{12}}{z'_{23}} + \dfrac{z'_{11}}{z'_{22}} + \dfrac{z'_{12}}{z'_{11}} + \dfrac{z'_{23}}{z'_{22}}.
\end{align*}
In the same way, we compute
\begin{align*}
\lrQ{1}{2} & = \dfrac{\lQ{1}{2}}{S_2S_3} = \dfrac{\Delta_{2,1}(\ov{M})}{\Delta_{3,1}(\ov{M})} + \dfrac{\Delta_{23,13}(\ov{M})}{\Delta_{23,12}(\ov{M})} = \dfrac{z'_{22}}{z'_{33}} + \dfrac{z'_{13}}{z'_{12}},
\\
\lrQ{2}{1} & = \dfrac{\lQ{2}{1}}{S_2S_2} = \dfrac{\Delta_{12,12}(\ov{M})}{\Delta_{23,12}(\ov{M})} + \dfrac{\Delta_{23,13}(\ov{M}) \Delta_{13,12}(\ov{M})}{\Delta_{23,12}(\ov{M})\Delta_{23,12}(\ov{M})} + \dfrac{\Delta_{23,23}(\ov{M})}{\Delta_{23,12}(\ov{M})} = \dfrac{z'_{12}z'_{22}}{z'_{23}z'_{33}} + \dfrac{z'_{13}}{z'_{23}} + \dfrac{z'_{11}z'_{13}}{z'_{12}z'_{22}} + \dfrac{z'_{13}}{z'_{11}}.
\end{align*}
\end{ex}

\begin{remark}
In \S\ref{sec:central charge}, we will generalize Example~\ref{ex:Q dec} by showing that the sum of reduced $Q$-invariants $\lrQ{1}{j}$ is equal to the decoration of the $Q$-pattern.
\end{remark}

\begin{ex}
\label{ex:Q inv d=1}
Let $\lE{i}{j}$ be a $Q$-type loop elementary symmetric function which appears in the block $M_1$ of the unfolded matrix $\widetilde{M}$. In other words, $i \in [m]$ and $j \in [n]$ must satisfy $m-i-j \in [0,n-1]$; in particular, all $Q$-type loop elementary symmetric functions appear in this block when $n \geq m$. By Proposition~\ref{prop:JT matrix version}, the corresponding $Q$-invariant $\lQ{i}{j}$ is equal to $\Delta_{I,J}(\widetilde{M})$, where
\[
I = [n-K+1,n] \cup [n+j,2n], \qquad J = [1,K+1] \cup [n+1,2n-j],
\]
and $K = i+j-m-1+n$. Theorem~\ref{thm:unfolded sum of minors} gives the formulas
\begin{align*}
\lQ{i}{j} &= \sum_{X_1 \in \binom{[j,i+j]}{i}} \Delta_{[m-i-j+2,m], [1,j-1] \cup X_1}(\ov{M}) \Delta_{X_1 \cup [i+j+1,m], [1,m-j]}(\ov{M}) \\
&= \sum_{a=0}^{i} \Delta_{[m-i-j+2,m], [1,j+i] \setminus \{j+a\}}(\ov{M}) \Delta_{[j,m] \setminus \{j+a\}, [1,m-j]}(\ov{M})
\end{align*}
and
\[
\lrQ{i}{j} = \dfrac{\lQ{i}{j}}{S_{j+1}S_{n+1-K}} = \sum_{a=0}^{i} \dfrac{\Delta_{[m-i-j+2,m], [1,j+i] \setminus \{j+a\}}(\ov{M}) \Delta_{[j,m] \setminus \{j+a\}, [1,m-j]}(\ov{M})}{\Delta_{[m-i-j+2,m], [1,j+i-1]}(\ov{M}) \Delta_{[j+1,m], [1,m-j]}(\ov{M})}.
\]

Observe that these formulas are stable in $n$, in the sense that their only dependence on $n$ comes from the fact that the entries of $\ov{M}$ are polynomials in $n$ sets of variables. This has the following interpretation: if one fixes $m$ and increases $n$ from 1, new reduced $Q$-invariants appear until $n = m$, at which point all reduced $Q$-invariants are present, and they have stabilized to the above expression. This corresponds to the fact that the $Q$-pattern reaches its full size when $n = m$.
\end{ex}


We saw in Proposition~\ref{prop:R=Weyl} that the action of the geometric $R$-matrix $R_i$ on $\Mat_{m \times n}(\CC^*)$ agrees with the action of the $\GL_m$-geometric crystal reflection operator $s_i$. Since geometric RSK is an isomorphism of geometric crystals, Corollary~\ref{cor:ov e} implies that the $Q$-invariants and shape invariants, when viewed as functions on the $Q$-pattern, are invariants for the birational action of the Weyl group $S_m$ on the Gelfand--Tsetlin geometric crystal $\GT_m^{\leq n}$. Moreover, Conjecture~\ref{conj:QS}(1) is equivalent to the following conjecture.

\begin{conj}
\label{conj:QQ}
The subfield of $\CC(z'_{j,i})$ fixed by the birational action of $S_m$ on $\GT_m^{\leq n}$ is generated by the (reduced) $Q$-invariants and shape invariants.
\end{conj}

\begin{remark}
In Conjecture~\ref{conj:QQ}, we view the $z'_{j,i}$ as indeterminates, rather than functions of the $x_i^j$. Since $\GT_m^{\leq n}$ stabilizes to $\GT_m$ when $n \geq m$, it is enough to prove Conjecture~\ref{conj:QQ} for $n \leq m$. It is therefore also enough to prove Conjecture~\ref{conj:QS}(1) for $n \leq m$.
\end{remark}

In~\cite{KB95}, Kirillov and Berenstein posed the problem of describing all piecewise-linear functions in the entries of a Gelfand--Tsetlin pattern that are invariant under the action of the crystal reflection operators. Conjecture~\ref{conj:QQ} proposes a solution to the ``geometric lifting'' of this problem. As for the original problem, Lemma~\ref{lem:back to tableau} implies that every pseudo-energy gives rise, by tropicalization, to an $S_m$-invariant piecewise-linear function on Gelfand--Tsetlin patterns, so we have obtained a large class of such functions.

\begin{question}
\label{q:QQ trop}
Do the tropicalizations of the (reduced) $Q$-invariants and shape invariants generate the semi-field of $S_m$-invariant piecewise-linear functions on a Gelfand--Tsetlin pattern? If not, is there a finite set of pseudo-energies whose tropicalizations have this property?
\end{question}

There does not seem to be an a priori implication---in either direction---between Conjecture~\ref{conj:QQ} and the corresponding conjecture at the piecewise-linear level.


\section{Cylindric pseudo-energies}
\label{sec:cylindric}

The main actors in the previous two sections were the skew loop Schur functions, which are the minors of the infinite matrix $\widetilde{M}$. This section is devoted to cylindric loop Schur functions, a class of loop symmetric functions coming from minors of the folded matrix $\wh{M}(t)$. Cylindric Schur functions in the ring of symmetric functions were introduced independently by Postnikov and McNamara~\cite{Postnikov05,McNamara}, and generalized to the loop setting by Lam, Pylyavskyy, and Sakamoto~\cite{LPS15}.

In \S\ref{sec:cyl loop Schur}, we define cylindric loop Schur functions as generating functions for cylindric tableaux. We then prove a cylindric Jacobi--Trudi formula, which expresses each cylindric loop Schur function as the coefficient (up to sign) of a power of $t$ in a minor of $\wh{M}(t)$, and vice versa. The proof uses a network obtained by wrapping the infinite network introduced in~\S\ref{sec:unfolded sum of minors} around a cylinder. In \S\ref{sec:folded sum of minors}, we identify a class of \defn{cylindric pseudo-energies} and prove cylindric analogues of the two main theorems about pseduo-energies (Theorems~\ref{thm:new det formula} and~\ref{thm:unfolded sum of minors}). In~\S\ref{sec:applications}, we use the cylindric analogue of Theorem~\ref{thm:unfolded sum of minors} to obtain formulas for energy and cocharge.

\subsection{Cylindric loop Schur functions}
\label{sec:cyl loop Schur}

Fix $k \in [n]$. Say that a partition $\la$ is \defn{$k$-cylindric} if
\[
\la_1 \leq k, \qquad \text{ and } \qquad \la'_1 - \la'_k \leq n-k.
\]
If $\la$ is $k$-cylindric, the union of all translates of the Young diagram of $\la$ by integer multiples of the vector $(-n+k,k)$ is an infinite skew shape, which we denote by $\tw{\la}$.\footnote{We index cells of Young diagrams with matrix-style coordinates, so the translation consists of $k$ steps to the right and $n-k$ steps up.} More generally, a \defn{$k$-cylindric shape} is a skew shape of the form $\la/\mu$, where $\mu \subseteq \la$ are $k$-cylindric partitions. (When we refer to a $k$-cylindric shape $\la/\mu$, we will always assume that the partitions $\mu$ and $\la$ are themselves $k$-cylindric.) We say that a semistandard Young tableau $T$ of shape $\la/\mu$ is a \defn{k-cylindric tableau} if the filling of $\tw{\la}/\tw{\mu}$ obtained by placing $T$ in each translate of $\la/\mu$ has weakly increasing rows and strictly increasing columns. See Figure~\ref{fig:cyl tabs} for an example and a non-example. We write $\Cyl_{\leq m}^k(\la/\mu)$ for the set of $k$-cylindric tableaux of shape $\la/\mu$ with entries at most $m$.

\begin{figure}
\begin{center}
\begin{tikzpicture}[scale=1.5]

\draw (5,1/2) node{$T=$};

\draw (6-1/3,1/2) rectangle node{$2$} (6+0/3,5/6);
\draw (6-2/3,1/2) rectangle node{$1$} (6-1/3,5/6);
\draw (6-2/3,1/6) rectangle node{$3$} (6-1/3,1/2);

\begin{scope}[xshift=1.3cm, yshift=-0.5cm]
\draw (5.15,0.1) node{$\iddots$};
\draw (7.5,1.7) node{$\iddots$};

\draw (6-1/3,1/2) rectangle node{$2$} (6+0/3,5/6);
\draw (6-2/3,1/2) rectangle node{$1$} (6-1/3,5/6);
\draw (6-2/3,1/6) rectangle node{$3$} (6-1/3,1/2);

\begin{scope}[xshift=0.66667cm, yshift=0.33333cm]
\draw (6-1/3,1/2) rectangle node{$2$} (6+0/3,5/6);
\draw (6-2/3,1/2) rectangle node{$1$} (6-1/3,5/6);
\draw (6-2/3,1/6) rectangle node{$3$} (6-1/3,1/2);
\end{scope}

\begin{scope}[xshift=1.33333cm, yshift=0.66667cm]
\draw (6-1/3,1/2) rectangle node{$2$} (6+0/3,5/6);
\draw (6-2/3,1/2) rectangle node{$1$} (6-1/3,5/6);
\draw (6-2/3,1/6) rectangle node{$3$} (6-1/3,1/2);
\end{scope}
\end{scope}

\begin{scope}[xshift=5.5cm]
\draw (5,1/2) node{$U=$};

\draw (6-1/3,1/2) rectangle node{$3$} (6+0/3,5/6);
\draw (6-2/3,1/2) rectangle node{$1$} (6-1/3,5/6);
\draw (6-2/3,1/6) rectangle node{$2$} (6-1/3,1/2);

\begin{scope}[xshift=1.3cm, yshift=-0.5cm]
\draw (5.15,0.1) node{$\iddots$};
\draw (7.5,1.7) node{$\iddots$};

\draw (6-1/3,1/2) rectangle node{$3$} (6+0/3,5/6);
\draw (6-2/3,1/2) rectangle node{$1$} (6-1/3,5/6);
\draw (6-2/3,1/6) rectangle node{$2$} (6-1/3,1/2);

\begin{scope}[xshift=0.66667cm, yshift=0.33333cm]
\draw (6-1/3,1/2) rectangle node{$3$} (6+0/3,5/6);
\draw (6-2/3,1/2) rectangle node{$1$} (6-1/3,5/6);
\draw (6-2/3,1/6) rectangle node{$2$} (6-1/3,1/2);
\end{scope}

\begin{scope}[xshift=1.33333cm, yshift=0.66667cm]
\draw (6-1/3,1/2) rectangle node{$3$} (6+0/3,5/6);
\draw (6-2/3,1/2) rectangle node{$1$} (6-1/3,5/6);
\draw (6-2/3,1/6) rectangle node{$2$} (6-1/3,1/2);
\end{scope}
\end{scope}
\end{scope}

\end{tikzpicture}
\end{center}

\caption{Suppose $k=2$ and $n=3$. Two semistandard Young tableaux of shape $\la = (2,1)$ are shown, together with the corresponding fillings of the infinite skew shape $\tw{\la}$. $T$ is $2$-cylindric, but $U$ is not.}
\label{fig:cyl tabs}
\end{figure}
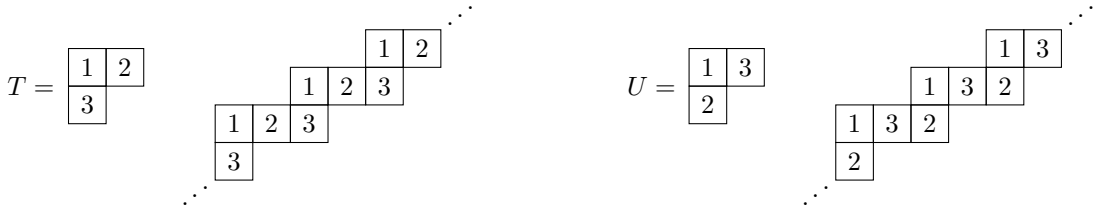

Given a $k$-cylindric shape $\la/\mu$ and a color $r \in \ZZ/n\ZZ$, define the \defn{cylindric loop Schur function}
\[
cs_{\la/\mu;k}^{(r)} = cs_{\la/\mu;k}^{(r)}(\mb{x}_1, \ldots, \mb{x}_m) = \sum_{T \in \Cyl^k_{\leq m}(\la/\mu)} \xx^T,
\]
where $\xx^T = \prod_{s \in \lambda/\mu} \lx{T(s)}{c(s)+r}$ as in the definition of loop Schur functions. Note that translation by $(-n+k,k)$ does not change the value of $c(s)$ modulo $n$, so we may use any fundamental domain of $\tw{\la}/\tw{\mu}$ to compute the weight. Also, different choices of $\la$ and $\mu$ may give rise to translations of the same infinite skew shape $\tw{\la}/\tw{\mu}$, leading to multiple expressions for the same cylindric loop Schur function. For example, when $n=3$, we have
\[
cs^{(r)}_{(2,1);2} = cs^{(r+1)}_{(2,2)/(1);2}.
\]

\begin{ex}
\label{ex:special cyl loop Schurs}
\
\begin{enumerate}
\item[(a)] A $1$-cylindric shape is a single column, and the corresponding cylindric loop Schur function is just a loop elementary symmetric function.

\item[(b)] An $(n-1)$-cylindric partition is uniquely determined by its number of boxes $N$, since each of the first $\lfloor N / (n-1) \rfloor$ rows must have length $n-1$, and the remaining boxes lie in the last row. We denote this unique partition by $\la(N)$. By \cite[Lem.~6.5]{LPS15}, we have
\begin{equation}
\label{eq_tau_def}
cs^{(r)}_{\la(N); n-1}(\mb{x}_1, \ldots, \mb{x}_m) = \sum_{\substack{1 \leq i_1 \leq i_2 \leq \cdots \leq i_N \leq m, \\ \#\{i_j = a\} \leq n-1}} \lx{i_1}{r} \lx{i_2}{r-1} \cdots \lx{i_N}{r-N+1} =: \tau^{(r)}_N(\mb{x}_1, \ldots, \mb{x}_m),
\end{equation}
where the sum is over weakly increasing sequences in which no number appears more than $n-1$ times.\footnote{For $k=n-1$, no further generality is gained by considering $s^{(r)}_{{\la/\mu}; n-1}$, since in this case every infinite skew shape $\tw{\la}/\tw{\mu}$ is a translation of $\tw{\nu}$ for some $(n-1)$-cylindric partition $\nu$.} The cylindric loop Schur functions $\tau^{(r)}_N$ play an important role in \S\ref{sec:energy}.

\item[(c)] The case $k = n$ is rather degenerate, since there is no vertical shift. The only $n$-cylindric shapes are the rectangles with $n$ columns, and a semistandard tableau is $n$-cylindric if and only if each row is constant. Thus, the cylindric loop Schur functions for $k=n$ are the (ordinary) elementary symmetric polynomials in $\pi_1, \ldots, \pi_m$, where $\pi_i = \lx{i}{1} \cdots \lx{i}{n}$. These polynomials were not considered to be cylindric loop Schur functions in~\cite{LPS15}, but we feel it is natural to include them in the definition.
\end{enumerate}
\end{ex}

\begin{dfn}
Suppose $\la$ is a $k$-cylindric partition with conjugate $\la' = (\la_1', \ldots, \la_k')$. If $\la_1' \geq n-k+1$, let $\la^\flat$ be the $k$-cylindric partition whose conjugate is given by
\begin{equation}
\label{eq_remove_border_strip}
(\la^\flat)' = (\la_2' - 1, \la_3'-1, \ldots, \la_k'-1,\la_1'-n+k-1).
\end{equation}
This is the partition obtained from $\la$ by removing the border strip of size $n$ which contains all boxes in the bottom row. If $\la_1' \leq n-k$, then there is no such border strip, and $\la^\flat$ is undefined.

Let $\la/\mu$ be a $k$-cylindric shape. If $\la^\flat$ is defined and $\mu \subseteq \la^\flat$, let $R(\la/\mu) = \la^\flat/\mu$; otherwise $R(\la/\mu)$ is undefined (this is not the same as saying that $R(\la/\mu)$ is the empty shape).
\end{dfn}

\begin{ex}
\label{ex:cyl shapes R}
Let $\la = (5,5,5,5,2,1), \mu = (2), k = 5, n = 7$. The $k$-cylindric shapes $\la/\mu, R(\la/\mu), R^2(\la/\mu)$ are shown below; $R^3(\la/\mu)$ is undefined.
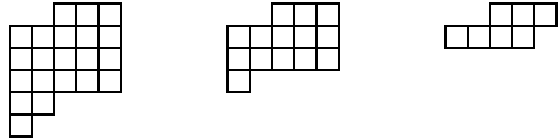
\ytableausetup{smalltableaux, nocentertableaux}
\[
\ydiagram{2+3,5,5,5,2,1}
\qquad\qquad
\ydiagram{2+3,5,5,1}
\qquad\qquad
\ydiagram{2+3,4}
\]
\end{ex}

\begin{remark}
\label{rem:Postnikov}
The operation $R(\la/\mu)$ is defined if and only if the infinite skew shape $\tw{\la}/\tw{\mu}$ is connected. In this case, the infinite skew shape associated to $R(\la/\mu)$ is obtained from $\tw{\la}/\tw{\mu}$ by removing the infinite border strip consisting of all boxes $(i,j) \in \la/\mu$ such that $(i+1,j+1) \not \in \la/\mu$. This corresponds to decreasing the parameter $d$ by 1 in Postnikov's notation $\kappa/d/\rho$ for cylindric shapes~\cite{Postnikov05}.
\end{remark}

Given $k$-cylindric partitions $\mu \subseteq \la$ with conjugates $\mu' = (\mu_1', \ldots, \mu_k')$ and $\la' = (\la_1', \ldots, \la_k')$, and a color $r \in [n]$, we define $k$-element subsets
\begin{align*}
I(\mu,r) &= \{\mu'_k-k+1+r, \mu'_{k-1}-k+2 + r, \ldots, \mu'_1+r\}, \\
J(\la,r) &= \{\la'_k-k+1+r-m, \la'_{k-1}-k+2 + r-m, \ldots, \la'_1+r-m\}
\end{align*}
as in~\eqref{eq:subsets JT matrix version}. The $k$-cylindric condition guarantees that $I(\mu,r)$ and $J(\la,r)$ are each contained in an interval of length $n$, so by reducing modulo $n$, we obtain $k$-element subsets $\wh{I}(\mu,r), \wh{J}(\la,r) \subseteq [n]$. For a $k$-element subset $S = \{s_1, \ldots, s_k\} \subset \ZZ$, define $d_*(S) = d_1 + \cdots + d_k$, where $d_a$ is the unique integer such that $s_a + d_a n \in [n]$.

Given $k$-element subsets $I = \{i_k < \cdots < i_1\}, J = \{j_k < \cdots < j_1\} \subseteq [n]$, let $\mu_I$ and $\la_J$ be the $k$-cylindric partitions whose conjugates are given by
\begin{align*}
(\mu_I)' &= (i_1-k, \: i_2-(k-1), \: \ldots, \: i_k-1), \\
(\la_J)' &= (m+j_1-k, \: m+j_2-(k-1), \: \ldots, \: m+j_k-1).
\end{align*}

\begin{lemma}
\label{lem:subsets and partitions}
\
\begin{enumerate}
\item If $I$ and $J$ are $k$-element subsets of $[n]$, then $I(\mu_I,k) = I$ and $J(\la_J,k) = J$.
\item Suppose $\la$ is a $k$-cylindric partition and $J(\la,r) = \{j_k < \cdots < j_1\}$. If $\la^\flat$ is defined, then
\[
J(\la^\flat,r) = \{j_1-n < j_k < \cdots < j_2\}.
\]
This implies that $\wh{J}(\la^\flat,r) = \wh{J}(\la,r)$ and $d_*(J(\la^\flat,r)) = d_*(J(\la,r))+1$.
\end{enumerate}
\end{lemma}

The proof of Lemma~\ref{lem:subsets and partitions} is left to the reader. We come to the main result of this section.

\begin{thm}[Cylindric Jacobi--Trudi formula]
\label{thm:cyl JT}
Let $\wh{M}(t)$ be the folded version of $\widetilde{M}(\mb{x}_1, \ldots, \mb{x}_m)$.
\begin{enumerate}
\item Let $\mu \subseteq \la$ be $k$-cylindric partitions, and $r \in \ZZ/n\ZZ$ a color. We have
\[
cs^{(r)}_{\la/\mu; k}(\mb{x}_1, \ldots, \mb{x}_m) = (-1)^{(k-1)d_*} [t^{d_*}] \Delta_{\wh{I}(\mu,r), \wh{J}(\mu,r)}(\wh{M}(t)),
\]
where $[t^c]$ denotes the coefficient of $t^c$ and $d_* = d_*(J(\la,r)) - d_*(I(\mu,r))$.

\item Let $I$ and $J$ be $k$-element subsets of $[n]$. We have
\begin{equation}
\label{eq:cyl JT}
\Delta_{I,J}(\wh{M}(t)) = \sum_{d = 0}^{d_{\max}} ((-1)^{k-1}t)^d \cdot cs^{(k)}_{R^d(\la_J/\mu_I); k}(\mb{x}_1, \ldots, \mb{x}_m),
\end{equation}
where $d_{\max}$ is the largest $d$ such that $R^d(\la_J/\mu_I)$ is defined (if $\mu_I \not \subseteq \la_J$, we set $d_{\max} = -1$).
\end{enumerate}
\end{thm}

Note that if $\la_J/\mu_I$ is defined, Remark~\ref{rem:Postnikov} implies that $d_{\max}$ is equal to the length of the shortest diagonal in $\widetilde{\la_J}/\widetilde{\mu_I}$.

\begin{ex}
\label{ex:cyl JT}
\
\begin{enumerate}
\item[(a)] Let $m=7$ and $n=5$, and consider the $3$-cylindric shape $\la/\mu = (3,3,3,3,2,1)/(2)$. We have
\[
I(\mu,4) = \{2,4,5\}, \qquad J(\la,4) = \{-1,1,3\}, \qquad d_*(I(\mu,4)) = 0, \qquad d_*(J(\la,4)) = 1,
\]
so part (1) of Theorem~\ref{thm:cyl JT} asserts that $cs^{(4)}_{\la/\mu; 3}(\mb{x}_1, \ldots, \mb{x}_7)$ is the coefficient of $t$ in $\Delta_{245,134}(\wh{M}(t))$.

\item[(b)] Suppose $m=4,n=7$, and let $I = \{1,2,3,5,6\}, J= \{1,2,3,5,7\}$. Then $\la := \la_J$ and $\mu := \mu_I$ are the partitions from Example~\ref{ex:cyl shapes R}. Since $\la/\mu$ has a column of length 5, the cylindric loop Schur function $cs^{(5)}_{\la/\mu; 5}(\mb{x}_1, \mb{x}_2, \mb{x}_3, \mb{x}_4)$ vanishes, and part (2) of Theorem~\ref{thm:cyl JT} says that
\[
\Delta_{I,J}(\wh{M}(t)) = t \cdot cs^{(5)}_{R(\la/\mu); 5}(\mb{x}_1, \mb{x}_2, \mb{x}_3, \mb{x}_4) + t^2 \cdot cs^{(5)}_{R^2(\la/\mu); 5}(\mb{x}_1, \mb{x}_2, \mb{x}_3, \mb{x}_4).
\]
(Alternatively, the lower triangularity of $\wh{M}(0)$ implies that $\Delta_{I,J}(\wh{M}(t))$ has no constant term.)
\end{enumerate}
\end{ex}

To prove Theorem~\ref{thm:cyl JT}, we use a version of the Lindstr\"om/Gessel--Viennot Lemma for networks on a cylinder. Recall the infinite planar network $\Net_{m,n}$ introduced in \S\ref{sec:unfolded sum of minors}. We may ``fold'' or ``project'' this network onto a cylinder by identifying rows $i$ and $i+n$ for all $i$. We call the resulting network $\wh{\Net}_{m,n}$, and we depict it as in Figure~\ref{fig:cyl JT}. In this depiction, the dashed red lines are identified, and we call this line the \defn{distinguished chord}. Given a highway path $p$ in $\wh{\Net}_{m,n}$, we define its \defn{winding number}, denoted $\wind(p)$, to be the number of times that $p$ crosses the distinguished chord.\footnote{Paths in $\wh{\Net}_{m,n}$ can only cross the distinguished chord in one direction; for more general cylindric networks, one would count each crossing with a sign to account for its direction.} The winding number of a family of highway paths is the sum of the winding numbers of each of its paths. For example, the family of highway paths shown in Figure~\ref{fig:cyl JT} has winding number one.

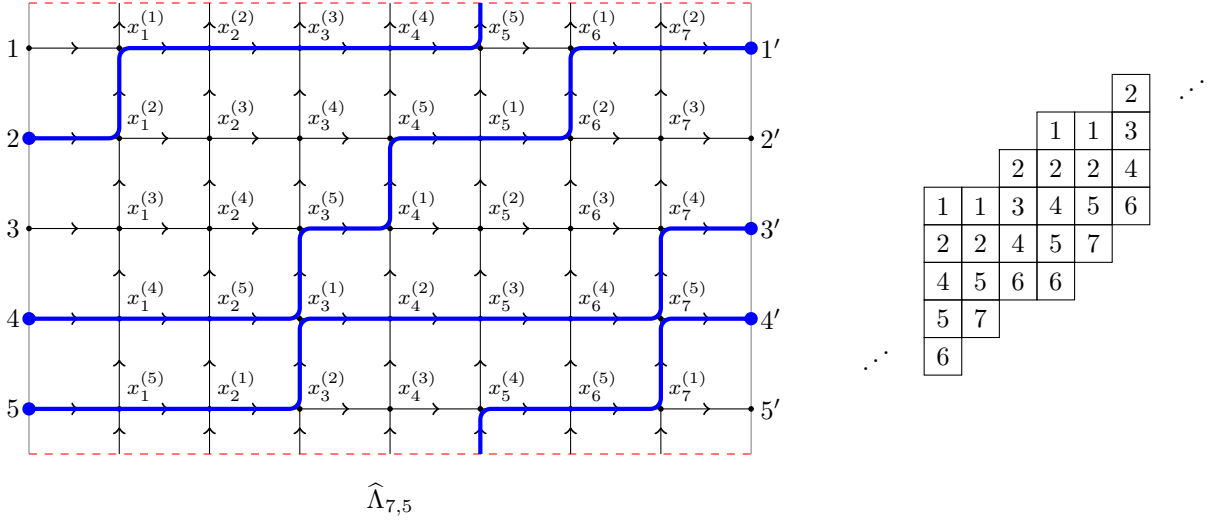
\begin{figure}
\begin{tikzpicture}[scale=0.6]

\pgfmathtruncatemacro{\m}{7};
\pgfmathtruncatemacro{\n}{5};

\draw (\m+1,-1) node{$\wh{\Net}_{7,5}$};

\draw[gray] (0,2*\n) -- (0,0);
\draw[gray] (2*\m+2,2*\n) -- (2*\m+2,0);
\draw[red,dashed] (0, 2*\n) -- (2*\m+2,2*\n);
\draw[red,dashed] (0,0) -- (2*\m+2,0);

\foreach \x in {1,...,\m} {
\draw (2*\x,0) -- (2*\x,2*\n);
\draw[thick,->] (2*\x,0.57) -- (2*\x,0.59);
\draw[thick,->] (2*\x,2*\n-0.5+0.07) -- (2*\x,2*\n-0.5+0.09);}

\foreach \y in {1,...,\n} {
\pgfmathtruncatemacro{\row}{\n-\y+1};
\filldraw (0,2*\y-1) circle [radius = .05] node[left] {$\row$};
\filldraw (2*\m+2,2*\y-1) circle [radius = .05] node[right] {$\row'$};
\draw (2*\m+2,2*\y-1) -- (0,2*\y-1);
\draw[thick,->] (2*\m+1.07,2*\y-1) -- (2*\m+1.09,2*\y-1);}

\foreach \x in {1,...,\m} {
\foreach \y in {1,...,\n} {
\draw[thick,->] (2*\x-0.93,2*\y-1) -- (2*\x-0.91,2*\y-1);
\ifnum \y < \n \draw[thick,->] (2*\x,2*\y+0.07) -- (2*\x,2*\y+0.09); \fi
\filldraw (2*\x,2*\y-1) circle [radius = .05];
\pgfmathtruncatemacro{\var}{\x};
\pgfmathtruncatemacro{\color}{mod(\n-\y+\x,\n)};
\ifnum \color < 1 \pgfmathtruncatemacro{\color}{\color+\n}; \fi
\footnotesize
\draw (2*\x,2*\y-1) node[above right]{$\lx{\var}{\color}$};}}
\normalsize

\foreach \x/\y in {5/1,13/3,    5/3,7/5,11/7,    1/7,9/9,13/1} {\draw[ultra thick,rounded corners, blue] (\x,\y) -- (\x+1,\y) -- (\x+1,\y+1);}
\foreach \x/\y in {6/2,14/4,    6/4,8/6,12/8,    2/8,10/0,14/2} {\draw[ultra thick,rounded corners, blue] (\x,\y) -- (\x,\y+1) -- (\x+1,\y+1);}
\foreach \x/\y in {1/1,3/1,7/3,9/3,11/3,    1/3,3/3,9/7,13/9,    3/9,5/9,7/9,11/1} {\draw[ultra thick, blue] (\x,\y) -- (\x+2,\y);}
\foreach \x/\y in {0/1,0/3,0/7} {\filldraw[ultra thick, blue] (\x,\y) circle [radius = 0.1] -- (\x+1,\y);}
\foreach \x/\y in {15/3,15/5,15/9} {\filldraw[ultra thick, blue] (\x,\y) -- (\x+1,\y) circle [radius = 0.1];}

\begin{scope}[xshift=9cm,yshift=5.5cm,scale=2.5]
\draw (3.9,-1.3) node{$\iddots$};
\draw (6.7,1.1) node{$\iddots$};

\draw (6,5/6) rectangle node{$2$} (6+1/3,7/6);

\draw (6-2/3,1/2) rectangle node{$1$} (6-1/3,5/6);
\draw (6-1/3,1/2) rectangle node{$1$} (6+0/3,5/6);
\draw (6,1/2) rectangle node{$3$} (6+1/3,5/6);

\draw (6-2/3,1/6) rectangle node{$2$} (6-1/3,1/2);
\draw (6-1/3,1/6) rectangle node{$2$} (6,1/2);
\draw (6,1/6) rectangle node{$4$} (6+1/3,1/2);

\draw (6-2/3,-1/6) rectangle node{$4$} (6-1/3,1/6);
\draw (6-1/3,-1/6) rectangle node{$5$} (6,1/6);
\draw (6,-1/6) rectangle node{$6$} (6+1/3,1/6);

\draw (6-2/3,-1/2) rectangle node{$5$} (6-1/3,-1/6);
\draw (6-1/3,-1/2) rectangle node{$7$} (6,-1/6);

\draw (6-2/3,-5/6) rectangle node{$6$} (6-1/3,-1/2);

\begin{scope}[xshift=-1cm,yshift=-0.667cm]
\draw (6,5/6) rectangle node{$2$} (6+1/3,7/6);

\draw (6-2/3,1/2) rectangle node{$1$} (6-1/3,5/6);
\draw (6-1/3,1/2) rectangle node{$1$} (6+0/3,5/6);
\draw (6,1/2) rectangle node{$3$} (6+1/3,5/6);

\draw (6-2/3,1/6) rectangle node{$2$} (6-1/3,1/2);
\draw (6-1/3,1/6) rectangle node{$2$} (6,1/2);
\draw (6,1/6) rectangle node{$4$} (6+1/3,1/2);

\draw (6-2/3,-1/6) rectangle node{$4$} (6-1/3,1/6);
\draw (6-1/3,-1/6) rectangle node{$5$} (6,1/6);
\draw (6,-1/6) rectangle node{$6$} (6+1/3,1/6);

\draw (6-2/3,-1/2) rectangle node{$5$} (6-1/3,-1/6);
\draw (6-1/3,-1/2) rectangle node{$7$} (6,-1/6);

\draw (6-2/3,-5/6) rectangle node{$6$} (6-1/3,-1/2);
\end{scope}
\end{scope}

\end{tikzpicture}
\caption{(Left) A non-intersecting family of highway paths of winding number 1 in the cylindric network $\wh{\Net}_{7,5}$. (Right) The periodic extension of the $3$-cylindric tableau of shape $(3,3,3,3,2,1)/(2)$ associated to the family of paths on the left.}
\label{fig:cyl JT}
\end{figure}

Associate to $\wh{\Net}_{m,n}$ the $n \times n$ matrix $A(\wh{\Net}_{m,n})$ with entries
\[
A(\wh{\Net}_{m,n})_{ij} = \sum_{p \colon i \to j'} t^{\wind(p)} \wt(p),
\]
where the sum is over highway paths in $\wh{\Net}_{m,n}$ from source $i$ to sink $j'$. This matrix is clearly the folded version of the $n$-periodic matrix associated to $\Net_{m,n}$, so $A(\wh{\Net}_{m,n}) = \wh{M}(t)$.

\begin{lemma}
\label{lem:cyl Lind}
For $k$-element subsets $I,J \subset [n]$, we have
\[
\Delta_{I,J}(\wh{M}(t)) = \sum_{d \geq 0} ((-1)^{k-1}t)^d \sum_{\mc{F} \in N^h(I,J;d)} \wt(\mc{F}),
\]
where $N^h(I,J;d)$ is the set of non-intersecting families of highway paths in $\wh{\Net}_{m,n}$ from sources $I$ to sinks $J$, of winding number $d$.
\end{lemma}

\begin{proof}
Convert $\wh{\Net}_{m,n}$ into an edge-weighted network $\wh{\Gamma}$ as in the proof of Lemma~\ref{lem:highway Lind}. It is clear that there is a weight and winding number preserving correspondence between non-intersecting families of highway paths in $\wh{\Net}_{m,n}$ and non-intersecting families of directed paths in $\wh{\Gamma}$. Therefore, the result follows by applying \cite[Thm.~3.5]{LP12} (a version of the Lindstr\"om/Gessel--Viennot Lemma for networks on a cylinder) to the network~$\wh{\Gamma}$.
\end{proof}

\begin{proof}[Proof of Theorem~\ref{thm:cyl JT}]
We first prove part (1), that
\[
cs^{(r)}_{\la/\mu; k}(\mb{x}_1, \ldots, \mb{x}_m) = (-1)^{(k-1)d_*} [t^{d_*}] \Delta_{\wh{I}(\mu,r), \wh{J}(\mu,r)}(\wh{M}(t)).
\]
Write $I(\mu,r) = \{i_k < \cdots < i_1\}$ and $J(\la,r) = \{j_k < \cdots < j_1\}$. Let $\mc{F} = (\ldots, p_{-1}, p_0, p_1, \ldots)$ be a family of highway paths in $\Net_{m,n}$, such that for each $a \in [k]$ and $N \in \ZZ$, the path $p_{a+Nk}$ starts at source $i_a + Nn$ and ends at sink $(j_a+Nn)'$. We know from the proof of Lemma~\ref{lem:A=M} that the highway paths in $\Net_{m,n}$ from source $i$ to sink $j'$ are in weight-preserving bijection with the strictly increasing fillings of a column of length $m+j-i$, whose top cell has color $i$. In the colored infinite skew shape $(\widetilde{\la}/\widetilde{\mu})^{(r)}$, column $a+Nn$ has length $\la_a' - \mu_a' = m+j_a-i_a$, and top cell of color $r-a+1+\mu_a' = i_a$, so there is a weight-preserving correspondence between fillings of this skew shape with strictly increasing columns and families of highway paths from $\{i_a+Nn\}$ to $\{(j_a+Nn)'\}$. It is straightforward to verify that the family $\mc{F}$ is non-intersecting if and only if the corresponding filling has weakly increasing rows.

Let $d_a,d'_a$ be the unique integers such that $i_a+d_an, j_a+d'_an \in [n]$. Let $F_1$ be the set of non-intersecting families $(p_1, \ldots, p_k)$ of highway paths in $\wh{\Net}_{m,n}$, such that $p_a$ starts at source $i_a+d_an$, ends at sink $(j_a + d'_an)'$, and has winding number $d'_a-d_a$. Let $F_2$ be the set of non-intersecting families of highway paths in $\Net_{m,n}$ which connect source $i_a+Nn$ to sink $(j_a+Nn)'$, and which are periodic, in the sense that the path out of source $i_a+Nn$ is a vertical translation of the path out of source $i_a+(N+1)n$. Viewing $\Net_{m,n}$ as the universal cover of $\wh{\Net}_{m,n}$, we get a simple bijection between $F_1$ and $F_2$ (it is easy to see that $d'_a - d_a$ is the number of times that a highway path in $\Net_{m,n}$ from source $i_a$ to sink $j'_a$ crosses a red dashed line). The monomials appearing in $cs^{(r)}_{\la/\mu; k}(\mb{x}_1, \ldots, \mb{x}_m)$ correspond to periodic semistandard fillings of $(\la/\mu)^{(r)}$, and it follows from the preceding paragraph that these monomials are in weight-preserving bijection with the elements of $F_1$. (See Figure~\ref{fig:cyl JT} for an example of this bijection in the case of the cylindric loop Schur function from Example~\ref{ex:cyl JT}(a).)

By definition, all families in $F_1$ have winding number
\[
(d_1'-d_1) + \cdots + (d_k'-d_k) = d_*.
\]
If $B = \{b_k < \cdots < b_1\}, C = \{c_k < \cdots < c_1\} \subseteq [n]$ and $\mc{F} \in N^h(B,C;d)$, then the topology of $\wh{\Net}_{m,n}$ requires that the path which starts at source $b_a$ end at sink $c'_{a+d \mod k}$; thus, the winding number and the sets of sources and sinks determine which sources and sinks are connected. We conclude that $F_1 = N^h(\wh{I}(\mu,r),\wh{J}(\la,r);d_*)$, and the desired result follows from Lemma~\ref{lem:cyl Lind}.

We now prove part (2), that
\[
\Delta_{I,J}(\wh{M}(t)) = \sum_{d = 0}^{d_{\max}} ((-1)^{k-1}t)^d \cdot cs^{(k)}_{R^d(\la_J/\mu_I); k}(\mb{x}_1, \ldots, \mb{x}_m),
\]
It follows immediately from part (1) and Lemma~\ref{lem:subsets and partitions} that for $d \leq d_{\max}$, $cs^{(k)}_{R^d(\la_J/\mu_I); k}(\mb{x}_1, \ldots, \mb{x}_m)$ is equal to the (appropriately signed) coefficient of $t^d$ in $\Delta_{I,J}(\wh{M}(t))$. It remains to show that the coefficient of $t^d$ vanishes for $d > d_{\max}$. Set $J^0 = J$, and for $d > 0$, let $J^d$ be the subset obtained from $J^{d-1}$ by subtracting $n$ from the largest element. Write $J^d = \{j^d_1 < \cdots < j^d_k\}$. The argument used to prove part (1) shows that the coefficient of $t^d$ vanishes if
\begin{equation}
\label{eq:no paths}
i_a - j_a^d > m \quad \text{ for some } a,
\end{equation}
since in this case there is no highway path in $\Net_{m,n}$ from source $i_a$ to sink $(j^d_a)'$. The reader may easily verify that~\eqref{eq:no paths} holds if and only if $d > d_{\max}$ (for example, $0 > d_{\max}$ if and only if $\mu_I \not \subseteq \la_J$, which means that $i_a-k+a-1 = (\mu'_I)_a > (\la'_J)_a = m+j_a-k+a-1$ for some $a$).
\end{proof}

\subsection{Formulas for cylindric pseudo-energies}
\label{sec:folded sum of minors}

Say that a colored $k$-cylindric shape $(\la/\mu)^{(r)}$ satisfies the \defn{cylindric corner color condition} if the colored infinite skew shape $(\widetilde{\la}/\widetilde{\mu})^{(r)}$ satisfies the usual corner color condition, that is, if its NW corners have color $n$ and its SE corners have color $m$. We define a \defn{cylindric pseudo-energy} to be a cylindric loop Schur function $s_{\la/\mu; k}^{(r)}$ such that $(\la/\mu)^{(r)}$ satisfies the cylindric corner color condition.

\begin{lemma}
\label{lem:cyl corner color}
Let $\mu \subseteq \la$ be $k$-cylindric partitions, $r \in \ZZ/n\ZZ$ a color, and $\wh{I}(\mu,r), \wh{J}(\la,r)$ the associated $k$-element subsets of $[n]$ (defined after Remark~\ref{rem:Postnikov}).
\begin{enumerate}
\item If $\widetilde{\la}/\widetilde{\mu}$ has no empty columns, then $(\la/\mu)^{(r)}$ satisfies the cylindric corner color condition if and only if $\wh{I}(\mu,r) = [n-k+1,n]$ and $\wh{J}(\la,r) = [k]$.
\item If $(\la/\mu)^{(r)}$ satisfies the cylindric corner color condition, then $cs_{\la/\mu; k}^{(r)}(\mb{x}_1, \ldots, \mb{x}_m)$ is $\ov{e}$-invariant.
\end{enumerate}
\end{lemma}

\begin{proof}
The proof of part (1) is essentially the same as that of Lemma~\ref{lem:initial final condition}. For part (2), first suppose $\widetilde{\la}/\widetilde{\mu}$ has an empty column. In this case, $(\widetilde{\la}/\widetilde{\mu})^{(r)}$ consists of disjoint copies of a colored skew shape $(\kappa/\rho)^{(s)}$, and $cs_{\la/\mu;k}^{(r)} = s_{\kappa/\rho}^{(s)}$. Clearly $(\kappa/\rho)^{(s)}$ satisfies the corner color condition, so $s_{\kappa/\rho}^{(s)}$ is $\ov{e}$-invariant by Proposition~\ref{prop:corner color condition}.

Now suppose $\widetilde{\la}/\widetilde{\mu}$ has no empty columns. By part (1) and Theorem~\ref{thm:cyl JT}(1), $\pm cs_{\la/\mu;k}^{(r)}$ is equal to a coefficient of a power of $t$ in $\Delta_{[n-k+1,n],[k]}\bigl(\wh{M}(t)\bigr)$. Lemma~\ref{lem:affine unipotent} implies that $\ov{e}_j$ acts on $\wh{M}(t)$ by adding a multiple of row $j+1$ to row $j$, and a multiple of column $j$ to column $j+1$, so $\Delta_{[n-k+1,n],[k]}\bigl(\wh{M}(t)\bigr)$---and each of its coefficients---is $\ov{e}$-invariant.
\end{proof}

We now give cylindric analogues of Theorems~\ref{thm:new det formula} and~\ref{thm:unfolded sum of minors}. Taking $I = [n-k+1,n]$ and $J = [k]$ in Theorem~\ref{thm:cyl JT}(2), we find $\mu_I = (k^{n-k})$ and $\la_J = (k^m)$, so $(\la_J/\mu_I)^{(k)}$ is undefined if $n-k > m$, and equal (as a colored shape) to $(k^{m-n+k})^{(n)}$ if $n-k \leq m$. In the latter case, $d_{\max}$ is the length of the shortest diagonal in $\widetilde{\la_J}/\widetilde{\mu_J}$, which is easily seen to be $\max\bigl(0,m-2(n-k)\bigr)$.

For $k \in [\min(1,n-m),n]$, set $i = n-k+1$, so that $i \in [1,\min(m,n)+1]$, and let
\[
\nu^i = (k^{m-n+k}) = ((n-i+1)^{m-i+1}).
\]
This partition is the shape of the shape invariant $S_i$. It follows from Theorem~\ref{thm:cyl JT}(2) and the preceding paragraph that
\begin{equation}
\label{eq:bottom left folded}
\Delta_{[i,n],[1,n-i+1]}\bigl(\wh{M}(t)\bigr) = \sum_{d = 0}^{\max(0,m-2i+2)} \bigl( (-1)^{n-i}t \bigr)^d \cdot cs_{R^d(\nu^i); n-i+1}^{(n)}(\mb{x}_1, \ldots, \mb{x}_m).
\end{equation}
In particular, since the coefficient of $t^0$ in $\Delta_{[i,n],[1,n-i+1]}(\wh{M}(t))$ is equal to $\Delta_{[i,n],[1,n-i+1]}(M(\mb{x}_1, \ldots, \mb{x}_m))$, we have
\[
cs_{\nu^i; n-i+1}^{(n)}(\mb{x}_1, \ldots, \mb{x}_m) = s_{\nu^i}^{(n)}(\mb{x}_1, \ldots, \mb{x}_m).
\]

\begin{thm}
\label{thm:new folded det formula}
Let $\wh{M}'(t)$ be the $n \times n$ matrix consisting of reduced $Q$-invariants and ratios of shape invariants, which was defined at the beginning of~\S\ref{sec:det formula}. For $i \leq \min(m,n)+1$, Equation~\eqref{eq:bottom left folded} holds with $\wh{M}(t)$ replaced by $\wh{M}'(t)$.
\end{thm}

\begin{proof}
It follows from the proof of Theorem~\ref{thm:new det formula} that $\wh{M}'(t) = U\wh{M}(t)V$, where $U,V$ are upper uni-triangular matrices. This implies that $\wh{M}(t)$ and $\wh{M}'(t)$ have the same bottom-left justified minors.
\end{proof}

\begin{ex}
Let $m=5,n=3$. Theorem~\ref{thm:new folded det formula}, together with Example~\ref{ex:special cyl loop Schurs}, gives the formula
\[
\det
\begin{pmatrix}
\lrQ{2}{1}t & \lrQ{3}{1}t + t^2 & \frac{S_1}{S_2} + \lrQ{4}{1}t + \lrQ{1}{1}t^2 \medskip \\
\lrQ{1}{2}t & -\frac{S_2}{S_3} + \lrQ{2}{2}t & \lrQ{3}{2}t + t^2 \medskip \\
S_3 + t & \lrQ{1}{3}t & \lrQ{2}{3}t
\end{pmatrix}
= \sum_{d = 0}^5 t^d e_{5-d}(\pi_1, \ldots, \pi_5),
\]
where $\pi_i = \lx{i}{1}\lx{i}{2}\lx{i}{3}$, and $e_a$ is the $a$-th elementary symmetric function. We can compute, for example,
\[
e_4(\pi_1, \ldots, \pi_5) = S_2 \lrQ{4}{1} - \dfrac{S_1S_3}{S_2} \lrQ{2}{2} + \dfrac{S_1}{S_3}.
\]
\end{ex}

\begin{thm}
\label{thm:folded sum of minors}
Suppose $1 \leq i \leq \min(m,n)$. For $0 \leq d \leq m-2i+2$, we have
\[
cs_{R^d(\nu^i); n-i+1}^{(n)}(\mb{x}_1, \ldots, \mb{x}_m) = \sum_{X \in \binom{[i,m-i+1]}{m-2i+2-d}} \Delta_{X \cup [m-i+2,m], [1,i-1] \cup X}(\ov{M}),
\]
where $\ov{M} = M(\mb{x}^1, \ldots, \mb{x}^n)$, and $\binom{S}{k}$ denotes the set of $k$-element subsets of the set $S$.

More generally, if $1 \leq a \leq b \leq m$, $1 \leq i \leq \min(b-a+1,n)$, and $0 \leq d \leq b-a-2i+3$, then we have
\[
cs_{R^d((n-i+1)^{b-a-i+2}); n-i+1}^{(n+a-1)}(\mb{x}_a, \ldots, \mb{x}_b) = \sum_{X \in \binom{[a+i-1,b-i+1]}{b-a-2i+3-d}} \Delta_{X \cup [b-i+2,b], [a,a+i-2] \cup X}(\ov{M}).
\]
\end{thm}

\begin{ex}
\label{ex:folded sum of minors}
Let $m=4, n=5$, and $i=2$. The shapes $\nu^i, R(\nu^i), R^2(\nu^i)$ are
\[
\ydiagram{4,4,4}
\qquad
\ydiagram{4,3}
\qquad
\ydiagram{2} \; ,
\]
and the matrix $\ov{M}$ looks like
\[
\ov{M} = M(\mb{x}^1, \ldots, \mb{x}^5) =
\begin{pmatrix}
\bE{5}{1} & 0 & 0 & 0 \\
\bE{4}{2} & \bE{5}{2} & 0 & 0 \\
\bE{3}{3} & \bE{4}{3} & \bE{5}{3} & 0 \\
\bE{2}{4} & \bE{3}{4} & \bE{4}{4} & \bE{5}{4}
\end{pmatrix}.
\]
By Theorem~\ref{thm:folded sum of minors} and the Jacobi--Trudi formula, we have
\ytableausetup{centertableaux}
\[
\begin{array}{r@{\;=\;}l@{\;=\;}l}
cs_{\ydiagram{4,4,4} \;;\, 4}^{(5)}(\mb{x}_1, \ldots, \mb{x}_4) & \Delta_{234,123}(\ov{M}) & \ov{s}_{\ydiagram{3,3,3,3}}^{(4)}(\mb{x}^1, \ldots, \mb{x}^5) \medskip \\
cs_{\ydiagram{4,3} \;;\, 4}^{(5)}(\mb{x}_1, \ldots, \mb{x}_4) & \Delta_{24,12}(\ov{M}) + \Delta_{34,13}(\ov{M}) & \ov{s}_{\ydiagram{1+1,2,2,2}}^{(3)}(\mb{x}^1, \ldots, \mb{x}^5) + \ov{s}_{\ydiagram{2,2,2,1}}^{(4)}(\mb{x}^1, \ldots, \mb{x}^5) \medskip \\
cs_{\ydiagram{2} \;;\, 4}^{(5)}(\mb{x}_1, \ldots, \mb{x}_4) & \Delta_{4,1}(\ov{M}) & \ov{s}_{\ydiagram{1,1}}^{(4)}(\mb{x}^1, \ldots, \mb{x}^5).
\end{array}
\]
\end{ex}

\begin{figure}
\begin{tikzpicture}[scale=0.4]

\pgfmathtruncatemacro{\m}{9};
\pgfmathtruncatemacro{\n}{6};

\filldraw[gray] (0,7) -- (5,12) -- (0,12) -- (0,7);
\filldraw[gray] (15,0) -- (20,5) -- (20,0) -- (15,0);

\draw[gray] (0,2*\n) -- (0,0);
\draw[gray] (2*\m+2,2*\n) -- (2*\m+2,0);
\draw[red,dashed] (0, 2*\n) -- (2*\m+2,2*\n);
\draw[red,dashed] (0,0) -- (2*\m+2,0);

\foreach \x in {1,...,\m} {
\draw (2*\x,0) -- (2*\x,2*\n);}

\foreach \y in {1,...,\n} {
\draw (2*\m+2,2*\y-1) -- (0,2*\y-1);}

\foreach \y in {3,4,5,6} {
\pgfmathtruncatemacro{\row}{\n-\y};
\filldraw[ultra thick, blue] (0,2*\row+1) circle [radius = .13];}

\foreach \y in {1,2,3,4} {
\pgfmathtruncatemacro{\row}{\n-\y};
\filldraw[ultra thick, blue] (2*\m+2,2*\row+1) circle [radius = .13];}

\foreach \y in {1,2,3,4,5,6} {
\pgfmathtruncatemacro{\row}{\n-\y};
\draw (0,2*\row+1) node[left] {$\y$};
\draw (2*\m+2,2*\row+1) node[right] {$\y'$};}

\foreach \x in {8,9} {\filldraw[ultra thick, UQpurple] (2*\x,2*\n) circle [radius = .13];}
\foreach \x in {1,2} {\filldraw[ultra thick, UQpurple] (2*\x,0) circle [radius = .13];}

\foreach \x in {1,2,3,5,6,8,9} {
\draw (2*\x,2*\n+0.2) node[above] {$\x$};
\draw (2*\x,-0.2) node[below] {$\x'$};}

\foreach \x/\y in {7/1,11/3,13/5,17/7,    5/3,7/5,9/7,15/9,    1/5,7/7,9/9,13/11,15/1,17/3,    1/7,5/9,7/11,11/1,13/3,17/5} {\draw[ultra thick,rounded corners, blue] (\x,\y) -- (\x+1,\y) -- (\x+1,\y+1);}
\foreach \x/\y in {8/2,12/4,14/6,18/8,    6/4,8/6,10/8,16/10,    2/6,8/8,10/10,14/0,16/2,18/4,    2/8,6/10,8/0,12/2,14/4,18/6} {\draw[ultra thick,rounded corners, blue] (\x,\y) -- (\x,\y+1) -- (\x+1,\y+1);}
\foreach \x\y in {1/1,3/1,5/1,9/3,15/7,    1/3,3/3,11/9,13/9,17/11,    3/7,5/7,11/11,    3/9,9/1,15/5} {\draw[ultra thick, blue] (\x,\y) -- (\x+2,\y);}
\foreach \x\y in {0/1,0/3,0/5,0/7,19/9,19/11,19/5,19/7} {\draw[ultra thick, blue] (\x,\y) -- (\x+1,\y);}

\foreach \x\y in {2/4,4/10,    4/4,6/8,8/10,    6/2,8/4,10/6,    10/4,12/6,14/10,    12/0,14/2,16/8} {\draw[ultra thick, UQpurple, dashed, rounded corners] (\x,\y) -- (\x,\y+1) -- (\x+2,\y+1) -- (\x+2,\y+2);}
\foreach \x\y in {2/0,2/2,4/6,4/8,    4/0,4/2,6/6,    6/0,12/8,12/10,    10/0,10/2,14/8,    16/4,16/6,18/10} {\draw[ultra thick, UQpurple, dashed] (\x,\y) -- (\x,\y+2);}

\foreach \x in {3,5,6} {
\filldraw (2*\x,2*\n) circle [radius=0.3];
\filldraw (2*\x,0) circle [radius=0.3];}

\end{tikzpicture}
\caption{An example of the bijection used to prove Theorem~\ref{thm:folded sum of minors}. The parameters for this example are $m=9, n=6, i=3, d=2$. The solid blue lines show the highway paths in $\mc{F}$, the dashed purple lines show the underway paths in $\ov{\mc{F}}$, and the large black vertices represent the elements of $X$.}
\label{fig:folded sum of minors}
\end{figure}
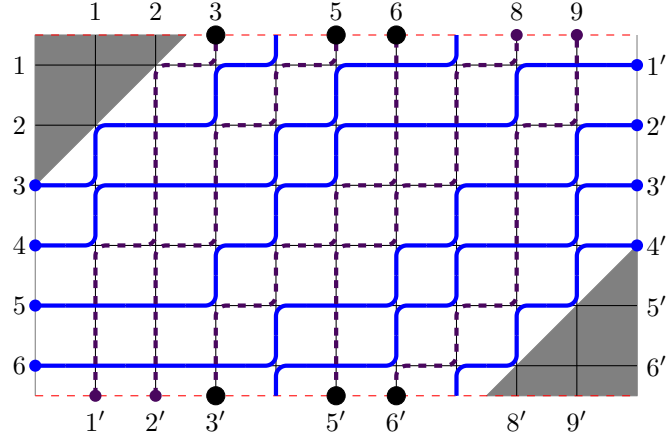

\begin{proof}[Proof of Theorem~\ref{thm:folded sum of minors}]
We know from \S\ref{sec:cyl loop Schur} and the above discussion that
\begin{equation}
\label{eq:highway side}
cs_{R^d(\nu^i); n-i+1}^{(n)}(\mb{x}_1, \ldots, \mb{x}_m) = \sum_{\mc{F} = (p_1, \ldots, p_{n-i+1})} \wt(\mc{F}),
\end{equation}
where the sum runs over non-intersecting families of highway paths in the cylindric network $\wh{\Net}_{m,n}$, of total winding number $d$, from sources $[i,n]$ to sinks $[1,n-i+1]$. We know from the proof of Theorem~\ref{thm:unfolded sum of minors} that
\begin{equation}
\label{eq:underway side}
\sum_{X \in \binom{[i,m-i+1]}{m-2i+2-d}} \Delta_{X \cup [m-i+2,m], [1,i-1] \cup X}(\ov{M}) = \sum_{\ov{\mc{F}} = (p_1, \ldots, p_{m-i+1-d})} \wt(\ov{\mc{F}}),
\end{equation}
where the sum on the left is over all non-intersecting families of underway paths in $\ov{\Net}_{m,n}$ which connect sources $X \cup [m-i+2,m]$ to sinks $[1,i-1] \cup X$, for some $X \in \binom{[i,m-i+1]}{m-2i+2-d}$.

The equality of~\eqref{eq:highway side} and~\eqref{eq:underway side} is proved by the same complementation bijection that was used to prove Theorem~\ref{thm:unfolded sum of minors}; see Figure~\ref{fig:folded sum of minors} for an example. In the cylindric setting, there are two triangular regions of fixed underway paths, which contain the first (resp., last) $i-1$ vertices on the north and west (resp., south and east) boundaries. The elements of $X$ correspond to the columns in the interval $[i,m-i+1]$ where a highway path does not cross the distinguished chord (i.e., does not ``wrap around''). This is why the number of elements of $X$ is equal to $m-2i+2$ minus the winding number $d$.

To prove the more general formula involving the sets of variables $\mb{x}_a, \ldots, \mb{x}_b$, remove columns $[1,a-1]$ and $[b+1,m]$ from the network and run the same argument. The indexing color of the cylindric loop Schur function is $n+a-1$ because that is the color of the first weight picked up by the highway path starting at source $n$.
\end{proof}


\section{Formulas for decoration, central charge, energy, and cocharge}
\label{sec:applications}

\subsection{Decoration and central charge}
\label{sec:central charge}

In Berenstein and Kazhdan's theory of geometric crystals, a \defn{decoration} is a positive function whose tropicalization cuts out the points of a combinatorial crystal from an integer lattice (for more details, see~\cite[\S\S 2.5, 3.3]{BFPS}).

The decoration of $\zz \in \GT_n^{\leq m}$ is defined by
\[
F(\zz) = \sum_{\substack{1 \leq i \leq m \\ \ i \leq j \leq n-1}} \frac{z_{i,j+1}}{z_{i,j}} \quad + \quad \sum_{\substack{1 \leq i \leq m-1 \\ \ i \leq j \leq n-1}} \frac{z_{i,j}}{z_{i+1,j+1}} \quad + \quad \mathbbm{1}_{m<n} z_{m,m},
\]
and the decoration of $\xx \in X^{\Mat}$ is defined by
\[
F(\xx) = \sum_{i \in [m], j \in [n]} \lx{i}{j} = \sum_{j \in [n]} \lE{1}{j}
\]
(this is the decoration on $X^{\Mat}$ with respect to both the $\GL_m$- and $\GL_n$-geometric crystal structures).

\begin{prop}[{\cite[\LemGTDec]{BFPS}}]
\label{prop:dec formula}
The decoration on $\GT_n^{\leq m}$ is given by
\begin{multline*}
F(\zz) = \sum_{k = 1}^{\min(m-1,n-1)} \dfrac{\Delta_{\{k\} \cup [k+2,n],[1,n-k]}(M) + \Delta_{[k+1,n],[1,n-k-1] \cup \{n-k+1\}}(M)}{\Delta_{[k+1,n],[1,n-k]}(M)} \\
+ \mathbbm{1}_{m<n} \left(\dfrac{\Delta_{\{m\} \cup [m+2,n],[1,n-m]}(M)}{\Delta_{[m+1,n],[1,n-m]}(M)} + \sum_{j = 1}^{n-m} \frac{\Delta_{[m+1,m+j],[1,j-1] \cup \{j+1\}}(M)}{\Delta_{[m+1,m+j],[1,j]}(M)}\right),
\end{multline*}
where $M = \Phi_n^{\leq m}(\zz)$.
\end{prop}

The formula in Proposition~\ref{prop:dec formula} is a special case of Berenstein and Kazhdan's general formula for the decoration of a unipotent bicrystal of parabolic type~\cite[Cor~1.25]{BK07}.

\begin{prop}[{\cite[Thm.~3.2]{OSZ}, \cite[\GRSKAndDecorations]{BFPS}}]
\label{prop:gRSK decorations}
Suppose $\xx \in \Mat_{m \times n}(\CC^*)$. If $\gRSK(\xx) = (P,Q)$, then
\[
F(\xx) = F(P) + F(Q) + \delta_{m,n} z_{n,n},
\]
where $\delta_{m,n}$ is the Kronecker delta, and $z_{n,n}$ is the last part of the shape of $P$ and $Q$ (when $m = n$).
\end{prop}

The difference $\Delta(\xx) = F(\xx) - F\bigl( P(\xx) \bigr)$ is called the {\defn{central charge}} of the $\GL_n$-geometric crystal $X^{\Mat} = (X_n)^m$ in~\cite{BK07}. In most cases, it is an open question whether the central charge of a product of geometric crystals is a positive rational function. In the case of $(X_n)^m$, we have
\begin{equation}
\label{eq:central charge}
\Delta(\xx) = F\bigl( Q(\xx) \bigr) + \delta_{m,n} S_n(\xx)
\end{equation}
by Proposition~\ref{prop:gRSK decorations}. The entries of the $Q$-pattern and the shape invariants are positive rational functions in the $\lx{i}{j}$, so the central charge of $(X_n)^m$ is positive.

Since the barred geometric crystal operators act on the $P$-pattern, it is clear from~\eqref{eq:central charge} that $\Delta(\xx)$ is $\ov{e}$-invariant (in general, the central charge of a $\GL_n$-geometric crystal is invariant under all $\GL_n$-geometric crystal operators). On the other hand, $F(\xx)$ is the sum of the loop elementary symmetric functions $\lE{1}{j}$, and $P(\xx)$ consists of ratios of loop Schur functions, so $\Delta(\xx)$ is $R$-invariant. We now give a formula which simultaneously demonstrates the $R$- and $\ov{e}$-invariance.

\begin{thm}
\label{thm:central charge}
We have the formulas
\[
\Delta(\xx) = \sum_{j=1}^{\min(m-1,n)} \lrQ{1}{j} + \delta_{m,n} \lE{1}{n}, \qquad\quad F\bigl( Q(\xx) \bigr) = \sum_{j=1}^{\min(m-1,n)} \lrQ{1}{j}.
\]
\end{thm}

Note that when $n=m$, the last shape invariant $S_n$ is equal to $\lE{1}{n}$, so the second formula follows from~\eqref{eq:central charge} and the first formula.

To prove this theorem, we introduce two new shapes. For $k = 1, \ldots, \min(m,n)$, the shape invariant $S_k$ is the loop Schur function of shape
\[
\rho_k^{(n)} = (\underbrace{n-k+1, \ldots, n-k+1}_{m-k+1 \text{ times}})^{(n)}.
\]
Define shapes
\[
\rho_{k,\downarrow}^{(n)} = (\underbrace{n-k+1, \ldots, n-k+1}_{m-k+1 \text{ times}}, 1)^{(n)},
\]
\[
\rho_{k,\uparrow}^{(n-1)} = ((\underbrace{n-k+1, \ldots, n-k+1}_{m-k+2 \text{ times}})/(n-k))^{(n-1)}
\]
by adding a box to the bottom of the first column (resp., to the top of the last column) of $\rho_k^{(n)}$.

\begin{lemma}
\label{lem:Q decomposition}
Suppose $1 \leq j \leq \min(m-1,n)$, and $K \in [0,n-1]$ satisfies $K \equiv j-m \mod n$. Then we have
\[
\lrQ{1}{j} = \lE{1}{j} - \mathbbm{1}_{j < n} \dfrac{s_{\rho_{j+1,\uparrow}^{(n)}}}{S_{j+1}} - \mathbbm{1}_{K > 0} \dfrac{s_{\rho_{n-K+1,\downarrow}^{(n-1)}}}{S_{n-K+1}}.
\]
\end{lemma}

\begin{proof}
The reduced $Q$-invariant $\lrQ{1}{j}$ is defined by
\[
\lrQ{1}{j} = \dfrac{s_{\tau^{(n-j+1)}}}{S_{j+1}S_{n-K+1}},
\]
where $\tau^{(n-j+1)}$ is the skew shape obtained by attaching a cell of color $j$ to the end of the first row of the rectangle $\rho_{j+1}^{(n)}$ and to the beginning of the last row of the rectangle $\rho_{n-K+1}^{(n)}$, where we use the convention that $\rho_{n+1}^{(n)}$ is the empty rectangle.

Suppose first that $j < n$ and $K > 0$, so both $\rho_{j+1}^{(n)}$ and $\rho_{n-K+1}^{(n)}$ are non-empty. The lemma asserts that
\begin{equation}
\label{eq:Q decomp to prove}
s_{\rho_{j+1}^{(n)}} \lE{1}{j} s_{\rho_{n-K+1}^{(n)}} = s_{\tau^{(n-j+1)}} + s_{\rho_{j+1,\uparrow}^{(n-1)}} s_{\rho_{n-K+1}^{(n)}} + s_{\rho_{j+1}^{(n)}} s_{\rho_{n-K+1,\downarrow}^{(n-1)}}.
\end{equation}
The product on the left-hand side of~\eqref{eq:Q decomp to prove} is a sum over triples $(T,b,U)$, where $T$ (resp., $U$) is an SSYT of shape $\rho_{j+1}^{(n)}$ (resp., $\rho_{n-K+1}^{(n)}$), and $b$ is a number in $[m]$. Let $a$ be the last entry of the first row of $T$, and $c$ the first entry of the last row of $U$. If $a \leq b \leq c$, then the triple corresponds to an SSYT of shape $\tau^{(n-j+1)}$. If $a > b$, we may attach $b$ above $a$ to form a tableau of shape $\rho_{j+1,\uparrow}^{(n)}$, and if $b > c$, we may attach $b$ below $c$ to form a tableau of shape $\rho_{n-K+1,\downarrow}^{(n-1)}$. The key point is that we cannot have both $a > b$ and $b > c$. Indeed, the height of the rectangle $\rho_{j+1}^{(n)}$ is $m-j$, and the height of the rectangle $\rho_{n-K+1}^{(n)}$ is $m-n+K$. By assumption, $j \leq m-1$, so $K = j-m + \alpha n$ for some $\alpha \geq 1$, which means that
\[
(m-j) + (m-n+K) = m + (\alpha-1) n \geq m.
\]
Since the tableaux are filled with entries in $[m]$, we must therefore have $a+1 \leq c$. We conclude that the correspondence described above gives a bijection between the terms on either side of~\eqref{eq:Q decomp to prove}, and it is clear that this bijection preserves colored weights.

A similar argument works when one or both of the rectangles is empty.
\end{proof}

\begin{lemma}
\label{lem:new dec formula}
For $\xx \in \Mat_{m \times n}(\CC^*)$, the decoration of the gRSK $P$-pattern associated to $\xx$ is given by
\[
F(P) = \begin{cases}
\ds \sum_{k = 2}^n \dfrac{s_{\rho_{k,\uparrow}^{(n)}} + s_{\rho_{k,\downarrow}^{(n-1)}}}{S_k} & \text{ if } m \geq n \medskip \\
\ds \sum_{k = 2}^m \dfrac{s_{\rho_{k,\uparrow}^{(n)}} + s_{\rho_{k,\downarrow}^{(n-1)}}}{S_k} + \sum_{j=m}^n \lE{1}{j} & \text{ if } m < n.
\end{cases}
\]
\end{lemma}

\begin{proof}
Proposition~\ref{prop:dec formula} expresses $F(P)$ in terms of minors of the matrix $M = \Phi_n^{\leq m}(P) = M(\mb{x}_1, \ldots, \mb{x}_m)$. Using Proposition~\ref{prop:JT matrix version}, it is straightforward to verify that the minors appearing in Proposition~\ref{prop:dec formula} are equal to the loop skew Schur functions appearing in this lemma.
\end{proof}

\begin{proof}[Proof of Theorem~\ref{thm:central charge}.]
First suppose $m > n$. As $j$ ranges over $1, \ldots, n$, $K$ takes on each of the values $0, \ldots, n-1$, so by Lemma~\ref{lem:Q decomposition}, we have
\begin{align*}
\sum_{j=1}^{n} \lrQ{1}{j} &= \sum_{j=1}^n \lE{1}{j} - \sum_{j = 1}^{n-1} \dfrac{s_{\rho_{j+1,\uparrow}^{(n)}}}{S_{j+1}} - \sum_{K=1}^{n-1} \dfrac{s_{\rho_{n-K+1,\downarrow}^{(n-1)}}}{S_{n-K+1}} \\
&= \sum_{j=1}^n \lE{1}{j} - \sum_{k = 2}^{n} \dfrac{s_{\rho_{k,\uparrow}^{(n)}} + s_{\rho_{k,\downarrow}^{(n-1)}}}{S_k},
\end{align*}
which is equal to $\Delta(\xx)$ by Lemma~\ref{lem:new dec formula}.

Now suppose $m \leq n$. In this case, as $j$ ranges over $1, \ldots, m-1$, $K$ takes on the values $n-m+1, n-m+2, \ldots, n-1$, so by Lemma~\ref{lem:Q decomposition}, we have
\begin{align*}
\sum_{j=1}^{m-1} \lrQ{1}{j} &= \sum_{j=1}^{m-1} \lE{1}{j} - \sum_{j = 1}^{m-1} \dfrac{s_{\rho_{j+1,\uparrow}^{(n)}}}{S_{j+1}} - \sum_{K=n-m+1}^{n-1} \dfrac{s_{\rho_{n-K+1,\downarrow}^{(n-1)}}}{S_{n-K+1}} \\
&= \sum_{j=1}^{m-1} \lE{1}{j} - \sum_{k = 2}^{m} \dfrac{s_{\rho_{k,\uparrow}^{(n)}} + s_{\rho_{k,\downarrow}^{(n-1)}}}{S_k}.
\end{align*}
By Lemma~\ref{lem:new dec formula}, this is equal to $\Delta(\xx)$ if $m < n$, and to $\Delta(\xx) - \lE{1}{n}$ if $m = n$.
\end{proof}

\begin{ex}
Let $m = n = 3$. Equation~\eqref{eq:central charge} says that the central charge of $\xx \in \Mat_{3 \times 3}(\CC^*)$ is given by
\[
\Delta(\xx) = \frac{z'_{12}}{z'_{11}} + \frac{z'_{13}}{z'_{12}} + \frac{z'_{23}}{z'_{22}} + \frac{z'_{11}}{z'_{22}} + \frac{z'_{12}}{z'_{23}} + \frac{z'_{22}}{z'_{33}} + z'_{33},
\]
where $\zz' = Q(\xx)$. From Example~\ref{ex:Q dec}, we have
\begin{align*}
\lrQ{1}{1} = \dfrac{z'_{12}}{z'_{23}} + \dfrac{z'_{11}}{z'_{22}} + \dfrac{z'_{12}}{z'_{11}} + \dfrac{z'_{23}}{z'_{22}},
\qquad\qquad
\lrQ{1}{2} = \dfrac{z'_{22}}{z'_{33}} + \dfrac{z'_{13}}{z'_{12}}.
\end{align*}
Noting that $S_3 = E_1^{(3)} = z'_{33}$, we see that
$
\Delta(\xx) = \lrQ{1}{1} + \lrQ{1}{2} + E_1^{(3)},
$
as claimed by Theorem~\ref{thm:central charge}.
\end{ex}

\subsection{Energy}
\label{sec:energy}

The energy function on a tensor product of finite crystals in affine type~\cite{KKMMNN91} plays an important role in various applications. The two crucial properties of the energy function are that it is invariant under combinatorial $R$-matrices which permute the tensor factors, and it is invariant under all non-affine crystal operators. In the case of an $m$-fold tensor product of single-row $\GL_n$-crystals, the first invariance property suggests that the energy function should have a geometric lift in $\Frac(\LSym)$, and Lam and Pylyavskyy showed that the energy is, in fact, the tropicalization of a loop Schur function.

\begin{dfn}
For $\mb{x} = (x_i^j)_{i \in [m], j \in [n]}$, define the \defn{geometric energy function}
\[
D = D(\mb{x}_1, \ldots, \mb{x}_m) = s^{(n)}_{(n-1)\delta_{m-1}}(\mb{x}_1, \ldots, \mb{x}_m),
\]
where $(n-1)\delta_{m-1}$ is the scaled staircase partition $((n-1)(m-1), (n-1)(m-2), \ldots, 2n-2, n-1)$.
\end{dfn}

\begin{thm}[{\cite[Thm.~1.2]{LP13}}]
\label{thm:tropical_energy}
The loop Schur function $D(\mb{x}_1, \ldots, \mb{x}_m)$ tropicalizes to the energy function on an $m$-fold tensor product of single-row $\GL_n$-crystals.
\end{thm}

On the other hand, since energy is invariant under the $\GL_n$-crystal operators $\ov{e}_1, \ldots, \ov{e}_{n-1}$, one would expect its geometric lift to lie in $\Inv_{\ov{e}}$. The stretched staircase $((n-1)\delta_{m-1})^{(n)}$ satisfies the corner color condition, so this is indeed the case for $D$. Moreover, Theorem~\ref{thm:new det formula} expresses $D$ as a polynomial in reduced $Q$-invariants and ratios of consecutive shape invariants via the determinantal formula
\[
D = \Delta_{I,J}(\widetilde{M}'),
\]
where $I = [m, (m-1)n]$ and $J = [1,n-1] \cup [n+1,2n-1] \cup \cdots \cup [(m-2)n+1, (m-1)n-1]$.

\ytableausetup{smalltableaux}
\begin{ex}
When $m=6$ and $n=2$, we have
\[
D = s^{(2)}_{(5,4,3,2,1)} =
\det \begin{pmatrix}
\lE{1}{2} & \lE{3}{2} &  \lE{5}{2} & 0 & 0 \\
1 & \lE{2}{1} & \lE{4}{1} & \lE{6}{1} & 0 \\
0 & \lE{1}{2} & \lE{3}{2} & \lE{5}{2} & 0 \\
0 & 1 & \lE{2}{1} & \lE{4}{1} & \lE{6}{1} \\
0 & 0 & \lE{1}{2} & \lE{3}{2} & \lE{5}{2}
\end{pmatrix}
=
\det \begin{pmatrix}
\lrQ{1}{2} & \lrQ{3}{2} & S_2 & 0 & 0 \\
1 & \lrQ{2}{1} & \lrQ{4}{1} & 0 & 0 \\
0 & \lrQ{1}{2} & \lrQ{3}{2} & S_2 & 0 \\
0 & 1 & \lrQ{2}{1} & \lrQ{4}{1} & 0 \\
0 & 0 & \lrQ{1}{2} & \lrQ{3}{2} & S_2 \end{pmatrix}.
\]
\end{ex}

Similarly, Theorem~\ref{thm:unfolded sum of minors} gives rise to an expression for the energy function as a polynomial in ($\ov{e}$-invariant) barred loop skew Schur functions. It turns out that we can get a factored expression for the energy function in terms of barred loop skew Schur functions by applying Theorem~\ref{thm:folded sum of minors} to a product formula for energy due to Lam and Pylyavskyy. Recall the matrix $\ov{M} = M(\mb{x}^1, \ldots, \mb{x}^n)$, whose entries are the barred $Q$-type loop elementary symmetric functions and whose minors are ($\ov{e}$-invariant) barred loop skew Schur functions.

\begin{thm}
\label{thm:energy}
The geometric energy function has the product formula
\[
D(\mb{x}_1, \ldots, \mb{x}_m) = \prod_{j = 1}^{m-1} \sum_{X \subseteq [j+1,m-1]} \pi_j^{m-1-j-\abs{X}} \Delta_{X \cup \{m\}, \{j\} \cup X}(\ov{M}),
\]
where $\pi_j = \ov{M}_{jj} = \bE{n}{j}(\mb{x}^1, \ldots, \mb{x}^n)$.
\end{thm}

\begin{proof}
Lam and Pylyavskyy \cite[Thm.~2.5]{LP13} proved that the geometric energy function has the product formula
\[
D = \sigma_{(n-1)(m-1)}^{(n)}(\mb{x}_1, \ldots, \mb{x}_m) \sigma_{(n-1)(m-2)}^{(n+1)}(\mb{x}_2, \ldots, \mb{x}_m) \cdots \sigma_{n-1}^{(n+m-2)}(\mb{x}_{m-1}, \mb{x}_m),
\]
where
\[
\sigma_N^{(r)}(\mb{x}_a, \ldots, \mb{x}_b) = \sum_{\substack{a \leq i_1 \leq i_2 \leq \cdots \leq i_N \leq b, \\ \#\{i_j = c\} \leq n-1 \text{ for } c > a}} \lx{i_1}{r} \lx{i_2}{r-1} \cdots \lx{i_N}{r-N+1}
\]
(the sum is over weakly increasing sequences of numbers in the interval $[a,b]$, such that each number other than $a$ appears at most $n-1$ times). The polynomials $\sigma_N^{(r)}(\mb{x}_a, \ldots, \mb{x}_b)$ are closely related to the cylindric loop Schur functions
\[
\tau_N^{(r)}(\mb{x}_a, \ldots, \mb{x}_b) = \sum_{\substack{a \leq i_1 \leq i_2 \leq \cdots \leq i_N \leq b, \\ \#\{i_j = c\} \leq n-1}} \lx{i_1}{r} \lx{i_2}{r-1} \cdots \lx{i_N}{r-N+1}
\]
mentioned in Example~\ref{ex:special cyl loop Schurs}(b). Indeed, by grouping together the terms of $\sigma_N^{(r)}$ in which the subscript $a$ occurs between $nd$ and $nd+n-1$ times, we obtain
\[
\sigma_N^{(r)}(\mb{x}_a, \ldots, \mb{x}_b) = \sum_{d \geq 0} (\lx{a}{r} \cdots \lx{a}{r-n+1})^d \cdot \tau_{N-nd}^{(r)}(\mb{x}_a, \ldots, \mb{x}_b).
\]
Using the identifications $\lx{i}{j+i-1} = x_i^j = \bx{j+i-1}{j}$ for $i \in [m], j \in [n]$, we find
\[
\lx{a}{r} \cdots \lx{a}{r-n+1} = x_a^1 \cdots x_a^n = \bx{a}{1} \cdots \bx{a+n-1}{n} = \bE{n}{a}(\mb{x}^1, \ldots, \mb{x}^n) = \pi_a.
\]

We know from Example~\ref{ex:special cyl loop Schurs}(b) that $\tau_N^{(r)}(\mb{x}_j, \ldots, \mb{x}_m) = cs_{\la(N);n-1}^{(r)}(\mb{x}_j, \ldots, \mb{x}_m)$, where $\la(N)$ is the unique $(n-1)$-cylindric partition with $N$ boxes. Thus, we have
\begin{align*}
\sigma_{(n-1)(m-j)}^{(n+j-1)}(\mb{x}_j, \ldots, \mb{x}_m) &= \sum_{d \geq 0} \pi_j^d \cdot \tau_{(n-1)(m-j)-nd}^{(n+j-1)}(\mb{x}_j, \ldots, \mb{x}_m) \\
&= \sum_{d \geq 0} \pi_j^d \cdot cs_{R^d((n-1)^{m-j});n-1}^{(n+j-1)}(\mb{x}_j, \ldots, \mb{x}_m) \\
&= \sum_{d \geq 0} \pi_j^d \sum_{X \in \binom{[j+1,m-1]}{m-1-j-d}} \Delta_{X \cup \{m\}, \{j\} \cup X}(\ov{M}),
\end{align*}
where the final equality comes from Theorem~\ref{thm:folded sum of minors}. This completes the proof.
\end{proof}

\ytableausetup{smalltableaux}
\begin{ex}
Suppose $m=4$ and $n=5$, as in Example~\ref{ex:folded sum of minors}. By definition,
\[
D = s_{\ytableaushort{\empty\empty\empty\empty\empty\empty\empty\empty\empty\empty\empty\empty,\empty\empty\empty\empty\empty\empty\empty\empty,\empty\empty\empty\empty}}^{(5)} \, .
\]
By Theorem~\ref{thm:energy}, we have $D = D_1 D_2 D_3$, where (writing $\Delta_{I,J}$ for $\Delta_{I,J}(\ov{M})$)
\begin{gather*}
D_1 = \Delta_{234,123} + \pi_1\Delta_{24,12} + \pi_1\Delta_{34,13} + \pi_1^2 \Delta_{4,1}
 = \ov{s}^{(4)}_{\ydiagram{3,3,3,3}} + \ov{s}^{(5)}_{\ydiagram{1,1,1,1,1}} \ov{s}^{(3)}_{\ydiagram{1+1,2,2,2}} + \ov{s}^{(1)}_{\ydiagram{1,1,1,1,1}} \ov{s}^{(4)}_{\ydiagram{2,2,2,1}} + \ov{s}^{(1)}_{\ydiagram{1,1,1,1,1}} \ov{s}^{(1)}_{\ydiagram{1,1,1,1,1}} \ov{s}^{(4)}_{\ydiagram{1,1}},
\\
D_2 = \Delta_{34,23} + \pi_2 \Delta_{4,2}
= \ov{s}^{(4)}_{\ydiagram{2,2,2,2}} + \ov{s}^{(2)}_{\ydiagram{1,1,1,1,1}} \ov{s}^{(4)}_{\ydiagram{1,1,1}},
\qquad\qquad
D_3 = \Delta_{4,3}
= \ov{s}^{(4)}_{\ydiagram{1,1,1,1}}. \qquad\qquad
\end{gather*}
\end{ex}

\begin{remark}
\
\begin{enumerate}
\item
It seems difficult to deduce Theorem~\ref{thm:energy} from Theorem~\ref{thm:unfolded sum of minors}, as it is not clear that the expression for energy coming from Theorem~\ref{thm:unfolded sum of minors} factors.
\item The individual factors appearing in Theorem~\ref{thm:energy} do not generally lie in $\LSym$, so they cannot be expressed in terms of $Q$-invariants and shape invariants.
\end{enumerate}
\end{remark}

\subsection{Cocharge}
\label{sec:cocharge}

Lascoux and Sch\"utzenberger introduced the \defn{cocharge} statistic on semistandard tableaux to give a combinatorial rule for the transition coefficients between Schur functions and Hall--Littlewood symmetric functions~\cite{LS78}. An important property of cocharge is its invariance under the symmetric group action on tableaux. Nakayashiki and Yamada proved that cocharge is closely related to the energy of a tensor product of one-row tableaux~\cite{NY97}. In this section, we use Theorem~\ref{thm:energy} and the connection between energy and cocharge to obtain a geometric lift of the cocharge statistic. We then show that our geometric lift is a (thinly disguised) version of a formula due to Kirillov and Berenstein~\cite{KB95}.

Let $\mb{z} = (z_{j,i}) \in \GT_m$ be a pattern of height $m-1$. For $i = 1, \ldots, m$, let
\[
\beta_i = \beta_i(\zz) = \dfrac{\prod_{a=1}^i z_{a,i}}{\prod_{a=1}^{i-1} z_{a,i-1}}
\]
denote the geometric lift of the number of $i$'s in the tableau associated to a GT pattern.

\begin{dfn}
Suppose $\mb{z} \in \GT_m$. For $i = 2, \ldots, m$, define
\[
\sigma_i(\mb{z}) = \sum_{X \subseteq [2,i-1]} \beta_i^{i-2-\abs{X}} \Delta_{X \cup \{i\}, \{1\} \cup X}(\Phi_m(\mb{z})),
\]
where $\Phi_m(\mb{z})$ is the $m \times m$ matrix introduced in \S\ref{sec:GT}. Define the \defn{geometric cocharge function}
\[
c_m(\mb{z}) = \sigma_2(\zz) \sigma_3(\zz) \cdots \sigma_m(\zz).
\]
(By convention, $c_1(\zz) = 1$.)
\end{dfn}

Recall from \S\ref{sec:GT} that the minors of $\Phi_m(\mb{z})$ are sums over non-intersecting families of paths in the triangle-shaped network $\Gamma_m$.

\begin{ex}
Suppose $\mb{z} \in \GT_4$. Using the network $\Gamma_4$ (shown in Figure~\ref{fig:network}), we compute
\[
\sigma_2(\mb{z}) = \Delta_{2,1} = z_{22}, \qquad\qquad
\sigma_3(\mb{z}) = \Delta_{23,12} + \beta_3\Delta_{3,1} = \dfrac{z_{23}(z_{33})^2}{z_{22}} \left(\dfrac{z_{22}}{z_{33}} + \dfrac{z_{13}}{z_{12}}\right),
\]
\begin{align*}
\sigma_4(\mb{z}) &= \Delta_{234,123} + \beta_4(\Delta_{24,12} + \Delta_{34,13}) + \beta_4^2 \Delta_{4,1} \medskip \\
&= \dfrac{z_{24}(z_{34})^2(z_{44})^3}{z_{23}(z_{33})^2} \left(\dfrac{z_{23}(z_{33})^2}{z_{34}(z_{44})^2} + \dfrac{z_{14}z_{23}z_{33}}{z_{13}z_{34}z_{44}} + \dfrac{z_{14}z_{22}}{z_{13}z_{44}} + \dfrac{z_{14}z_{33}}{z_{12}z_{44}} + \dfrac{z_{14}z_{24}z_{33}}{z_{13}z_{23}z_{44}} + \dfrac{(z_{14})^2z_{24}}{(z_{13})^2z_{23}} \right).
\end{align*}
\end{ex}

\begin{thm}
\label{thm:cocharge}
The geometric cocharge function tropicalizes to a piecewise-linear formula for cocharge.
\end{thm}

Before proving this theorem, we make a few remarks. A reader familiar with cocharge may easily verify that the cocharge of a tableau with entries at most 3 is given by
\[
\Trop(\sigma_2) + \Trop(\sigma_3) = z_{23} + 2z_{33}^2 + \min(z_{22}-z_{33}, z_{13} - z_{12}).
\]
An appealing feature of this formula is that $\sigma_i(\mb{z})$ depends only on the entries in rows $\leq i$ (this is clear from inspection of the network $\Gamma_m$). This means that $\Trop(\sigma_i)$ depends only on the part of a tableau consisting of entries less than or equal to $i$.

\begin{proof}[Proof of Theorem~\ref{thm:cocharge}]
Fix $n \geq m$, and let $\mb{a} = (a_i^j)$ be an $m \times n$ matrix with rows $\mb{a}_1, \ldots, \mb{a}_m \in (\ZZ_{\geq 0})^n$. As described in the introduction, we view each $\mb{a}_i$ as an element of a one-row $\GL_n$-crystal. By abuse of notation, let $D(\mb{a}) = D(\mb{a}_1 \otimes \cdots \otimes \mb{a}_m)$ be the energy of the tensor product $\mb{a}_1 \otimes \cdots \otimes \mb{a}_m$. Let $P'$ be the tableau obtained by \emph{column inserting} $\mb{a}_{m-1}$ into $\mb{a}_m$, and then column inserting $\mb{a}_{m-2}$ into the result, etc. Let $Q'$ be the recording tableau for this process, i.e., $Q'$ has $i$'s in the boxes added to $P'$ by the insertion of $\mb{a}_{m+1-i}$. The energy $D(\mb{a})$ is equal to the cocharge of $Q'$~\cite{NY97} (see also~\cite[Prop.~4.25]{ShimozonoDummies}). The map $(\mb{a}_m, \ldots, \mb{a}_1) \mapsto (P',Q')$ is known as the \defn{Burge correspondence}.

Now suppose that $(\mb{a}_1, \ldots, \mb{a}_m) \mapsto (P,Q)$ under RSK, as defined at the beginning of \S\ref{sec_gRSK} (i.e., $P$ is obtained by \emph{row inserting} $\mb{a}_2$ into $\mb{a}_1$, then row inserting $\mb{a}_3$ into the result, etc.). It is well known that $P' = P$ and $Q' = S(Q)$, where $S$ denotes the Sch\"utzenberger involution (see, e.g., \cite[Appendix]{Fulton}). We conclude that $D(\mb{a})$ is equal to the cocharge of $S(Q)$.

Theorem~\ref{thm:tropical_energy} states the geometric energy function $D(\mb{x})$ tropicalizes to a piecewise-linear formula for $D(\mb{a})$. Theorem~\ref{thm:energy} expresses $D(\mb{x})$ in terms of the minors of the matrix $\ov{M} = M(\mb{x}^1, \ldots, \mb{x}^n) = \Phi_m(\mb{z}')$, where $(\mb{z},\mb{z}')$ is the image of $\mb{x}$ under gRSK. Let $M^T$ denote the reflection of $M$ over the anti-diagonal, that is, $(M^T)_{i,j} = M_{n+1-j,n+1-i}$. Noumi and Yamada proved that this reflection is a geometric lift of the Sch\"utzenberger involution, in the sense that the entries of the pattern $\Psi_m(\Phi_m(\mb{z}')^T)$ tropicalize to formulas for the entries of the Sch\"utzenberger involution of a Gelfand--Tsetlin pattern \cite[\S 2.4]{NoumiYamada}.

Our definition of geometric cocharge is obtained from the formula in Theorem~\ref{thm:energy} by replacing $\ov{M}$ with $\ov{M}^T$. Since $\gRSK(\mb{x})$ tropicalizes to $\RSK(\mb{a})$, we conclude that $c(\mb{z}')$ tropicalizes to a formula for the cocharge of $Q$, as claimed.
\end{proof}

We now show that $c_m(\zz)$ agrees (up to a simple multiplicative factor) with the formula for cocharge given by Kirillov and Berenstein~\cite{KB95}.\footnote{Their version of $\sigma_i$ is missing the factor $\frac{z_{2i} (z_{3i})^2 \cdots (z_{ii})^{i-1}}{z_{2,i-1} (z_{3,i-1})^2 \cdots (z_{i-1,i-1})^{i-2}}$, so their version of $c_m$ is missing the factor $z_{2,m} (z_{3,m})^2 \cdots (z_{m,m})^{m-1}$. Thus, the tropicalization of their formula, applied to a tableau of shape $\la$, differs from the actual cocharge by $n(\la) = \sum (i-1)\la_i$.} A proof that this formula tropicalizes to cocharge was outlined in~\cite{Kir01}, but our proof is completely different.

Let $\mathcal{P}_{k-1}$ denote the set of Gelfand--Tsetlin patterns $\mb{p} = (p_{i,j})_{1 \leq i \leq j \leq k-1}$ satisfying the following conditions:
\begin{enumerate}
\item $p_{i,j} - 1 \leq p_{i+1,j}$ for $i < j$
\item $p_{i,k-1} - 1 \leq p_{i,i} \leq p_{i+1,k-1}$ for $i \leq k-2$
\item $p_{k-1,k-1} = 0$.
\end{enumerate}
Given $\zz \in \GT_m$ with $m \geq k$, define the weight of $\mb{p} \in \mathcal{P}_{k-1}$ with respect to $\zz$ by
\[
\wt_\zz(\mb{p}) = \prod_{1 \leq i \leq j \leq k-1} \phi_{k-j,k-i}(\zz)^{p_{i,j}},
\]
where
\[
\phi_{1,j}(\zz) = \dfrac{z_{1,j}z_{j,j}}{z_{1,j+1}z_{j+1,j+1}}, \qquad \text{ and for } i \geq 2, \quad \phi_{i,j}(\zz) = \dfrac{z_{i-1,j} z_{i,j}}{z_{i-1,j-1}z_{i,j+1}}
\]
is the diamond ratio used in \S\ref{sec:GT} to give explicit formulas for the geometric crystal structure. Note that $\wt_\zz(\mb{p})$ depends only on the first $k$ rows of $z$ (i.e., the entries $z_{i,j}$ with $j \leq k$).

\begin{ex}
\label{ex:KB cocharge}
The elements of $\mathcal{P}_3$ are shown below, together with their weights.
\[
\begin{matrix}
\quad
\begin{matrix}
0 \\
0 \quad 0 \\
0 \quad 0 \quad 0
\end{matrix}
\quad & \quad
\begin{matrix}
0 \\
0 \quad 0 \\
1 \quad 0 \quad 0
\end{matrix}
\quad & \quad
\begin{matrix}
0 \\
1 \quad 0 \\
1 \quad 0 \quad 0
\end{matrix}
\quad & \quad
\begin{matrix}
0 \\
1 \quad 0 \\
1 \quad 1 \quad 0
\end{matrix}
\quad & \quad
\begin{matrix}
1 \\
1 \quad 0 \\
1 \quad 1 \quad 0
\end{matrix}
\quad & \quad
\begin{matrix}
1 \\
1 \quad 0 \\
2 \quad 1 \quad 0
\end{matrix} \quad \bigskip \\
1 & \dfrac{z_{13}z_{33}}{z_{14}z_{44}} & \dfrac{z_{13}^2z_{23}z_{33}}{z_{12}z_{14}z_{24}z_{44}} & \dfrac{z_{13}z_{22}z_{23}}{z_{14}z_{24}z_{44}} & \dfrac{z_{14}z_{23}z_{33}}{z_{13}z_{34}z_{44}} & \dfrac{z_{13}^2z_{23}^2z_{33}^2}{z_{14}^2z_{24}z_{34}z_{44}^2}
\end{matrix}
\]
\end{ex}

As observed in~\cite{KB95}, the set $\mathcal{P}_{k-1}$ has cardinality $(k-1)!$. To see this, suppose
\[
\mb{c} = (c_1,c_2,\ldots,c_{k-1}) \in [0,k-2] \times [0,k-3] \times \cdots \times [0,0].
\]
We associate to $\mb{c}$ an element of $\mathcal{P}_{k-1}$ by recursively forming the SW-NE diagonals, from right to left. The right-most diagonal corresponds to $c_{k-1} = 0$, and consists of the single entry $p_{k-1,k-1} = 0$. For $i < k-1$, the diagonal with entries $p_{i,k-1}, \ldots, p_{i,i}$ is chosen so that its first $c_i$ entries from the top are $a$, and its remaining $k-i-c_i$ entries are $a+1$, where $a$ is the unique number such that condition (1) and the second inequality in condition (2) are not violated. For example, the element of $\mc{P}_6$ shown in Figure~\ref{fig:paths GT bijection} corresponds to $\mb{c} = (3,0,3,1,1,0)$.

\begin{prop}
\label{prop:KB cocharge}
For $k = 2, \ldots, m$, we have
\begin{equation}
\label{eq:KB cocharge to prove}
\sigma_k(\zz) = \beta_k(\zz)^{k-2} z_{k,k} \sum_{\mb{p} \in \mathcal{P}_{k-1}} \wt_\zz(\mb{p}).
\end{equation}
\end{prop}

\begin{proof}
Both sides of~\eqref{eq:KB cocharge to prove} depend only on the entries in the first $k$ rows of $\zz$, so we may assume that $k=m$. By the Lindstr\"om/Gessel--Viennot Lemma, $\sigma_m(\zz)$ is the sum over non-intersecting families of paths in $\Gamma_m$ which start at sources $X \cup \{m\}$ and end at sinks $\{1'\} \cup X'$, for some $X \subseteq [2,m-1]$. Let $\mc{F}$ be such a family of paths. Label each of the regions formed by the edges of $\Gamma_m$ with the number $a$, where $a+1$ is the number of paths in $\mc{F}$ to the left of that region. Let $\mb{p}$ be the triangular array of height $m-1$ obtained by reversing the order of the labels in each row (except the label to the left of sink $1'$, which is ignored). This is illustrated in Figure~\ref{fig:paths GT bijection}. We will show that this defines a weight-preserving bijection between the objects on either side of~\eqref{eq:KB cocharge to prove}.

\begin{figure}
\begin{center}
\begin{tikzpicture}[xscale=0.8,yscale=0.6]

\pgfmathtruncatemacro{\m}{7};

\foreach \i in {1,...,\m} {
\draw (0,\i) -- (\i,0);
\draw (\i,0) -- (\i,\m-\i);
\draw (-0.5,\i) node[left]{$\i$};
\draw (\i,-0.5) node[below]{$\i'$};}

\draw[blue,ultra thick] (0,2) -- (1,1) -- (1,0);
\draw[blue,ultra thick] (0,4) -- (2,2) -- (2,0);
\draw[blue,ultra thick] (0,6) -- (1,5) -- (1,4) -- (4,1) -- (4,0);
\draw[blue,ultra thick] (0,7) -- (4,3) -- (4,2) -- (6,0);

\foreach \a/\b in {1/-1,2/0,3/0,4/1,5/1,6/2} {\draw (0.5,\a) node{\textcolor{red}{$\b$}};}
\foreach \a/\b in {0/0,1/0,2/0,3/1,4/2,5/2} {\draw (1.5,\a) node{$\b$};}
\foreach \a/\b in {0/1,1/1,2/1,3/2,4/2} {\draw (2.5,\a) node{$\b$};}
\foreach \a/\b in {0/1,1/1,2/2,3/2} {\draw (3.5,\a) node{$\b$};}
\foreach \a/\b in {0/2,1/2,2/3} {\draw (4.5,\a) node{$\b$};}
\foreach \a/\b in {0/2,1/3} {\draw (5.5,\a) node{$\b$};}
\foreach \a/\b in {0/3} {\draw (6.5,\a) node{$\b$};}

\draw (8,3.5) node{$\longleftrightarrow$};
\draw (12,3.5) node{$\begin{matrix}
2 \\
2 \quad 2 \\
2 \quad 2 \quad 1 \\
3 \quad 2 \quad 1 \quad 0 \\
3 \quad 2 \quad 1 \quad 1 \quad 0 \\
3 \quad 2 \quad 2 \quad 1 \quad 1 \quad 0
\end{matrix}$};

\end{tikzpicture}
\end{center}
\caption{Bijection between non-intersecting families of paths in $\Gamma_7$ and GT patterns in $\mathcal{P}_6$. The labels to the left of sink $1'$ (shown in red) are not formally part of the bijection, but they are convenient for the proof.}
\label{fig:paths GT bijection}
\end{figure}
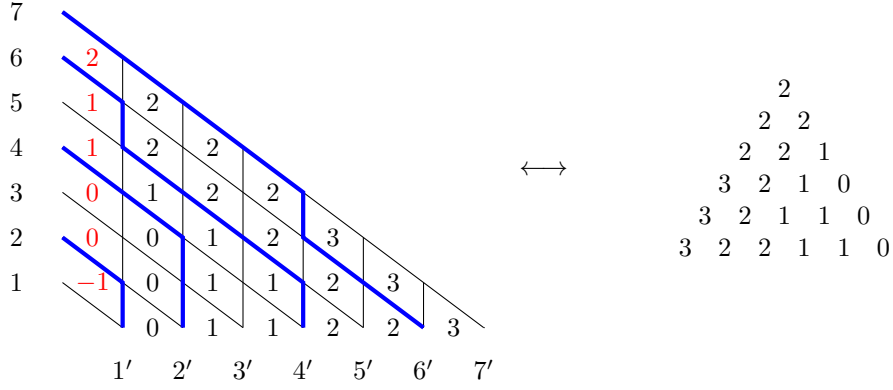

It is obvious that the array $\mb{p}$ associated to $\mc{F}$ is a Gelfand--Tsetlin pattern satisfying condition (1) in the definition of $\mc{P}_{m-1}$. It is also clear that $\mb{p}$ satisfies condition (3), since the vertical edge directly above sink $1'$ must be in $\mc{F}$. To see that condition (2) holds, suppose $X = \{x_1 < \cdots < x_\ell\}$, and let $p_0, p_1, \ldots, p_{\ell}$ be the paths in $\mc{F}$, with $p_i$ starting at source $x_{i+1}$ and ending at sink $x_i'$ (with $x_{\ell+1} = m$ and $x_0' = 1'$). The number of vertical steps in $p_i$ is equal to $x_{i+1}-x_i$. Moreover, there must be exactly one vertical step in $p_i$ on each of the NW-SE diagonals between the lines connecting $x_i$ to $x_i'$, and $x_{i+1}$ to $x_{i+1}'$. This shows that the label at the bottom of each NW-SE diagonal differs from the red label at the top of the diagonal by exactly 1, which is equivalent to condition (2). The reader may easily verify that the map $\mc{F} \mapsto \mb{p}$ is invertible.

It remains to show that the bijection is weight-preserving. Let $\mc{F}$ be a non-intersecting family of paths from $X \cup \{m\}$ to $\{1'\} \cup X'$. For each $i \in [m-1]$, let
\[
a_1 \leq a_2 \leq \cdots \leq a_i
\]
be the labels in the column between sinks $(m-i)'$ and $(m-i+1)'$ (from bottom to top), and let
\[
S_i = \{j \in [i] \mid a_j < a_{j+1}\},
\]
where we set $a_{i+1} := p_{1,m-1} = \abs{X}$ for all $i$. The diagonal edge in $\Gamma_m$ directly above the label $a_j$ has weight $\dfrac{z_{j,m-i+j}}{z_{j,m-i+j-1}}$, and this edge is in $\mc{F}$ and only if $j \in S_i$. The diagonal edge coming out of source $j$ has weight $z_{j,j}$, and is in $\mc{F}$ if and only if $j \in X \cup \{m\}$. Thus, the contribution of $\mc{F}$ to $\sigma_m(\zz)$ is
\begin{equation}
\label{eq:KB cocharge LHS}
\beta_m^{m-2-\abs{X}} \prod_{j \in X \cup \{m\}} z_{j,j} \prod_{i = 1}^{m-1} \prod_{j \in S_i} \dfrac{z_{j,m-i+j}}{z_{j,m-i+j-1}}.
\end{equation}

Let $\mb{p}$ be the GT pattern corresponding to $\mc{F}$. The labels $a_1 \leq a_2 \leq \cdots \leq a_i$ are the entries
\[
p_{i,m-1} \leq p_{i-1,m-2} \leq \cdots \leq p_{1,m-i}
\]
in the $i$th NW-SE diagonal of $\mb{p}$. The contribution of this diagonal to the weight of $\mb{p}$ is
\begin{align*}
\prod_{j = 1}^{m-1} (\phi_{j,m-i+j-1}(\zz))^{p_{i-j+1,m-j}} &= \left(\phi_{1,m-i}(\zz) \cdots \phi_{i,m-1}(\zz)\right)^{p_{i,m-1}} \prod_{j \in S_i \cap [i-1]} \phi_{j+1,m-i+j}(\zz) \cdots \phi_{i,m-1}(\zz) \\
&= \left(\dfrac{z_{m-i,m-i}}{z_{m-i+1,m-i+1}}\dfrac{z_{i,m-1}}{z_{i,m}}\right)^{p_{i,m-1}} \prod_{j \in S_i \cap [i-1]} \dfrac{z_{j,m-i+j}}{z_{j,m-i+j-1}}\dfrac{z_{i,m-1}}{z_{i,m}} \\
&= \left(\dfrac{z_{i,m-1}}{z_{i,m}}\right)^{\abs{X}} \left(\dfrac{z_{m-i,m-i}}{z_{m-i+1,m-i+1}}\right)^{p_{i,m-1}} \prod_{j \in S_i} \dfrac{z_{j,m-i+j}}{z_{j,m-i+j-1}}.
\end{align*}
The last equality comes from the observation that $p_{i,m-1}$ is either equal to $\abs{X}$ or $\abs{X}-1$; in the former case, $i \not \in S_i$, and in the latter case, $i \in S_i$, and contributes $z_{i,m}/z_{i,m-1}$ to the product. The entries $p_{i,m-1}$ in the bottom row of $\mb{p}$ start at $p_{1,m-1} = \abs{X}$ and decrease to $p_{m-1,m-1} = 0$ in increments of 1, with $p_{m-j,m-1} > p_{m-j+1,m-1}$ if and only if $j \in X$, so
\[
\prod_{i=1}^{m-1} \left(\dfrac{z_{m-i,m-i}}{z_{m-i+1,m-i+1}}\right)^{p_{i,m-1}} = \prod_{j \in X} \dfrac{z_{j,j}}{z_{m,m}}.
\]
Thus, the contribution of $\mb{p}$ to the right-hand side of~\eqref{eq:KB cocharge to prove} is
\[
\beta_m^{m-2} z_{m,m} \prod_{j \in X} \dfrac{z_{j,j}}{z_{m,m}} \prod_{i=1}^{m-1} \left(\left(\dfrac{z_{i,m-1}}{z_{i,m}}\right)^{\abs{X}} \prod_{j \in S_i} \dfrac{z_{j,m-i+j}}{z_{j,m-i+j-1}} \right),
\]
which agrees with~\eqref{eq:KB cocharge LHS}.
\end{proof}


\section{Final remarks}
\label{sec:final}

\subsection{Affine actions}
\label{sec:affine actions}
As mentioned in Remark~\ref{rem:affine geometric crystal}, one can extend the geometric crystal structures on $\Mat_{m \times n}(\CC^*)$ by introducing affine geometric crystal operators $e_0$ and $\ov{e}_0$. One can also extend the Weyl group actions to affine Weyl group actions by defining $R_0$ (resp., $\ov{R}_0$) to be the geometric $R$-matrices acting on the first and last row (resp., first and last column) of a matrix. The commutativity of barred and unbarred operators---including the affine operators---still follows from~\cite{KajNouYam} (for $R$-matrices) and~\cite{LP13II} (for crystal operators).

It is natural to extend our study of invariants to all possible combinations of barred and unbarred actions from the set $\{\emptyset, R, R_0, e, e_0\}$. For example $R\ov{e}_0$ denotes the simultaneous invariants of the action of non-affine $R$-matrices, and of all barred crystal operators (including the affine one). Some of the simplest questions are open, such as describing the $R_0$-invariants, which form a subfield of $\Frac(\LSym)$. 

One special case that has been considered in the literature is that of $R_0 \ov{R}_0$-invariants. In~\cite{KajNouYam} it was observed that all coefficients of the characteristic polynomial of the folded version of $\widetilde{M}(\mb{x}_1, \ldots, \mb{x}_m)$ are invariants of both $R_0$ and $\ov{R}_0$. This characteristic polynomial is a polynomial in two variables, one of which is $t$. It was studied under the name of \defn{spectral curve} in~\cite{ILP16}, due to its relation to a certain class of discrete Toda lattices.  The coefficients of the spectral curve were given a combinatorial interpretation in terms of certain paths on a torus. An alternative way to take measurements in toric networks that yields $R_0 \ov{R}_0$-invariants is given in~\cite{LP13II}. 

One can ask if any of these constructions produces enough invariants to generate the whole field (or ring) of $R_0 \ov{R}_0$-invariants. The answer is no. For example, consider the case $n=m=2$. In this case it is easy to see that the above methods give invariants of even degree, such as
\[
x_1^1x_2^1 + x_1^2x_2^2, \qquad x_1^1x_1^2 + x_2^1x_2^2, \qquad x_1^1x_2^1x_1^2x_2^2.
\]
However, the sum of all parameters $x_1^1+x_1^2+x_2^1+x_2^2$ is an $R_0 \ov{R}_0$-invariant.

\subsection{Comparison with classical invariant theory}
\label{sec:classical inv theory}
As mentioned in the abstract, one can consider $S_m$, $\SL_m$, $S_n \times S_m$, $\SL_n \times \, S_m$, and $\SL_n \times \SL_m$ naturally acting on a polynomial ring in $mn$ variables. One can then ask for a description of the invariants of those actions, which should be viewed as the ``classical'' versions of the questions asked in this paper. Surprisingly, while in some cases the invariants are well understood, in other cases very little is known about them.

The best understood case is that of $\SL_m$-invariants. The ring of invariants is trivial if $n < m$, and it is generated by the $m \times m$ minors of the matrix when $n \geq m$. This is the subject of classical invariant theory, and we refer the reader to~\cite{Weyl46} for a classical exposition, and to~\cite{PopVin94, GoodWall00} for more modern ones. In this case both first and second fundamental theorems are known, i.e., both generators of the ring of invariants and relations among these generators are known.

Less understood, but relatively well-studied, is the case of $S_m$-invariants~\cite{Dalbec99, Briand02, Vac05}. The ring of invariants in this case has many names: multisymmetric functions~\cite{Briand04}, diagonally symmetric polynomials~\cite{Allen94}, MacMahon symmetric functions~\cite{Gessel87}, etc. The first fundamental theorem in this case is known; see~\cite{Briand04}, where it is shown that the analogues of elementary symmetric functions generate the ring. The second fundamental theorem describing the relations, however, is not only not known, but is likely to be hard, because of its relation to Foulkes' conjecture~\cite{Briand02}. 

The ring of $\SL_n \times \SL_m$-invariants is non-trivial only if $m=n$, in which case it is generated by the determinant. As far as the authors can tell, almost nothing is known about the cases of $S_n \times S_m$- and $\SL_n \times S_m$-invariants (aside from the $\SL_n \times S_n$-invariants, which are generated by the square of the determinant).


\bibliographystyle{alpha}
\bibliography{gRSK_and_invariant_fields}

\end{document}